\documentclass[smallextended,referee,envcountsect,]{svjour3}
\smartqed

\usepackage{graphicx,psfrag,amsfonts,verbatim,  mathtools}
\usepackage{xurl}

\usepackage{lipsum}
\usepackage{amsfonts}
\usepackage{graphicx}
\usepackage{epstopdf}
\ifpdf
  \DeclareGraphicsExtensions{.eps,.pdf,.png,.jpg}
\else
  \DeclareGraphicsExtensions{.eps}
\fi

\usepackage{bbm}
\usepackage{amsthm}
\usepackage{amssymb}
\usepackage{enumitem}

\usepackage{dsfont}
\usepackage[caption=false]{subfig}

\usepackage{ upgreek } %

\usepackage[bbgreekl]{mathbbol}
\DeclareSymbolFontAlphabet{\mathbbm}{bbold}
\DeclareSymbolFontAlphabet{\mathbb}{AMSb}%
\usepackage[bb=px]{mathalfa}
\usepackage{bm}

\usepackage{diagbox}
\usepackage{booktabs}           %
\usepackage{adjustbox} %

\usepackage{tikz}
\usetikzlibrary{fit, arrows, shapes, positioning, shadows, matrix, calc}

\usepackage{algorithm2e} 

\RequirePackage{luatex85}
\usepackage{pgfplots}
\pgfplotsset{compat=newest}
\usepgfplotslibrary{groupplots}
\usepgfplotslibrary{polar}
\usepgfplotslibrary{smithchart}
\usepgfplotslibrary{statistics}
\usepgfplotslibrary{dateplot}
\usetikzlibrary{spy,backgrounds}

\global\long\def\rl{\mbox{{\bf R}}}%

\global\long\def\eu{\mathbf{E}}%

\global\long\def\ig{\mbox{{\bf Z}}}%

\global\long\def\nt{\mbox{{\bf N}}}%

\global\long\def\tr{\mathop{{\bf tr}}}%

\global\long\def\zer{\mathop{{\bf zer}}}%

\global\long\def\prox{\mathbf{prox}}%

\global\long\def\argmin{\mathop{{\rm argmin}}}%

\global\long\def\gra{\mathop{{\bf gra}}}%

\global\long\def\fix{\mathop{\mathop{{\bf fix}}}}%

\global\long\def\rank{\mathop{{\bf rank}}}%

\global\long\def\card{\mathop{{\bf card}}}%

\global\long\def\proj{\mathop{{\bf \Pi}_\mathcal{X}}}%

\global\long\def\ones{\mathbf{1}}%

\global\long\def\regf{f}%

\global\long\def\ind{\iota_\mathcal{X}}%

\global\long\def\morind{\prescript{\mu}{}{\iota}}%

\global\long\def\regind{\prescript{\mu}{}{\mathcal{I}}}%

\global\long\def\opA{\mathbb{A}}%

\global\long\def\opT{\mathbb{T}}%

\global\long\def\opJ{\mathbb{J}}%

\global\long\def\opR{\mathbb{R}}%

\global\long\def\id{\mathbb{I}}%

\theoremstyle{plain}
\newtheorem{lem}{\protect\lemmaname}
\theoremstyle{plain}

\theoremstyle{definition}
\newtheorem{defn}{\protect\definitionname}
\theoremstyle{plain}
\newtheorem{prop}{\protect\propositionname}
\theoremstyle{remark}
\newtheorem{rem}{\protect\remarkname}
\theoremstyle{plain}
\newtheorem{thm}{\protect\theoremname}
\theoremstyle{plain}

\newtheorem{assumption}{Assumption}

\journalname{}

\begin{document}

\title{Exterior-point Optimization for Sparse and Low-rank Optimization}

\author{Shuvomoy Das Gupta, Bartolomeo Stellato, and Bart P.G. Van Parys}

\institute{Shuvomoy Das Gupta \at
             Operations Research Center \\
             Massachusetts Institute of Technology\\
             sdgupta@mit.edu
           \and
              Bartolomeo Stellato  \at
              Department of Operations Research and Financial Engineering \\
              Princeton University \\
              bstellato@princeton.edu
            \and
            Bart P.G. Van Parys \\
            Sloan School of Management \\
            Massachusetts Institute of Technology \\
            vanparys@mit.edu
}

\date{}

\maketitle

\begin{abstract}
{Many problems of substantial current interest in machine learning, statistics, and data science can be formulated as sparse and low-rank optimization problems. In this paper, we present the nonconvex exterior-point optimization solver \textsf{(NExOS)}---a first-order algorithm tailored to sparse and low-rank optimization problems. We consider the problem of minimizing a convex function over a nonconvex constraint set, where the set can be decomposed as the intersection of a compact convex set and a nonconvex set involving sparse or low-rank constraints. Unlike the convex relaxation approaches, \textsf{NExOS} finds a locally optimal point of the original problem by solving a sequence of penalized problems with strictly decreasing penalty parameters by exploiting the nonconvex geometry. \textsf{NExOS} solves each penalized problem by applying a first-order algorithm, which converges linearly to a local minimum of the corresponding penalized formulation under regularity conditions. Furthermore, the local minima of the penalized problems converge to a local minimum of the original problem as the penalty parameter goes to zero. We then implement and test \textsf{NExOS} on many instances from a wide variety of sparse and low-rank optimization problems, empirically demonstrating that our algorithm outperforms specialized methods.}
\end{abstract}

\vspace{5mm} %
\noindent
\textit{Communicated by Cl\'{e}ment W. Royer}


\section{Introduction\label{sec:Introduction}}

{
This paper studies optimization problems involving a strongly convex and smooth cost
function over a closed nonconvex constraint set $\mathcal{X}$ involving sparse or low-rank constraints. We propose
a first-order algorithm \emph{nonconvex exterior-point optimization
solver} (\textsf{NExOS}) to solve such problems numerically. We can
write such problems as: 
\begin{equation}
\begin{array}{ll}
\underset{}{\mbox{minimize}} & f(x)+(\beta/2)\|x\|^{2}\\
\textrm{subject to} & x\in\mathcal{X},
\end{array}\tag{\ensuremath{\mathcal{P}}}\label{eq:original_problem-1}
\end{equation}
where $x$ takes value in a finite-dimensional
vector space $\eu$ over the reals, $f$ is a strongly convex and  smooth function. In Appendix \ref{subsec:nonsmooth_NExOS}, we generalize our framework to the case when $f$ is non-smooth convex.

The regularization
parameter $\beta>0$ is commonly introduced in statistics and machine learning problems to reduce the generalization error without increasing the
training error \cite[\S 5.2.2]{Goodfellow-et-al-2016}. In this paper, there is also a theoretical consideration behind including the term $\frac{\beta}{2} \|x\|^2$ in problem \eqref{eq:original_problem-1}. \textsf{NExOS} finds a locally optimal point of problem \eqref{eq:original_problem-1} by solving a sequence of penalized subproblems with strictly decreasing penalty parameters, where each penalized subproblem is solved by a first-order algorithm. Under the presence of $\frac{\beta}{2} \|x\|^2$ with $\beta > 0$, we can prove that each penalized subproblem is locally strongly convex and smooth admitting a unique local minimum (see Proposition \ref{Attainment-of-local-min}), which in turn ensure linear convergence of the first-order method to that local minimum (see Theorem \ref{Theorem_Main-convergence-result}). In the Numerical Experiments section, we demonstrate that $\beta$ can be set to a value as small as $10^{-8}$. This empirical evidence suggests that, in practice, the impact of $\beta$ on the objective value can be made negligible, yet one can still reap the theoretical benefits. Therefore, while $\beta$ plays a crucial role in the theoretical aspects of our algorithm, its influence on the problems considered in the Numerical Experiments section is minimal and can be adjusted as per the problem's requirements. 

Furthermore, $\eu$ is equipped
with inner product $\left\langle \cdot\mid\cdot\right\rangle $ and
norm $\|\cdot\|=\sqrt{\left\langle x\mid x\right\rangle }$. For  $\eu = \mathbf{R}^d$, we have $\left\langle x \mid y \right\rangle = x^\top y$ for  $x,y \in \mathbf{R}^d$, and for $\eu = \mathbf{R}^{m \times n}$, we have $\left\langle X \mid Y \right\rangle = \tr (X^\top Y)$,  for $X, Y \in \mathbf{R}^{m \times n}$.  The constraint set $\mathcal{X}$ is closed and nonconvex and can be decomposed as the intersection of a compact convex set and a nonconvex set involving sparse or low-rank constraints. Sparse and low-rank constraint sets are very important in modeling many machine learning problems,  because they allow for high interpretability, speed-ups in computation, and reduced memory requirements~\cite{jain2017non}.

\paragraph{Sparsity-constrained optimization}
Sparsity constraints have found applications in many
practical settings, \emph{e.g.}, gene expression analysis \cite[pp. 2--4]{Hastie2015},
sparse regression \cite[pp. 155--157]{jain2017non}, signal transmission
and recovery \cite{Candes2008,Tropp2006}, hierarchical sparse polynomial
regression \cite{Bertsimas2020}, and best subset selection \cite{bertsimas2016best},
just to name a few. In these problems, the constraint set $\mathcal{X}$
decomposes as $\mathcal{X}=\text{\ensuremath{\mathcal{C}}}\bigcap\mathcal{N}$,
where $\mathcal{C}$ is a compact convex set, and 
\begin{equation}
\mathcal{N}=\{x\in\rl^{d}\mid\card(x)\leq k\},\label{eq:sparse-set}
\end{equation}
where $\mathbf{card}(x)$ counts the number of nonzero elements in
$x$. In these optimization problems, $\mathcal{C}$ can be a polyhedron, infinity-norm ball, box constraint set, or probability simplex; these sets usually show up in applications involving econometrics, housing price prediction, air-quality prediction, signal processing, and meteorology \cite{bertsimas2021slowly,Diamond2018,bertsimas2020sparse,bertsimas2022scalable}. 

In this paper, we apply \textsf{NExOS
}to solve the sparse regression problem for both synthetic and real-world
datasets in \S\ref{subsec:Regressor-selection}, which is concerned
with approximating a vector $b\in\mathbf{R}^{m}$ with a linear combination
of at most $k$ columns of a matrix $A\in\mathbf{R}^{m\times d}$
with bounded coefficients. This problem has the form: 
\begin{equation}
\begin{array}{ll}
\textup{minimize} & \|Ax-b\|_{2}^{2}+(\beta/2)\|x\|_{2}^{2}\\
\textup{subject to} & \mathbf{card}(x)\leq k, \quad \|x\|_{\infty}\leq\Gamma,
\end{array}\tag{SR}\label{eq:reg-sel}
\end{equation}
where $x\in\mathbf{R}^{d}$ is the decision variable, and $A\in\mathbf{R}^{m\times d},\;b\in\mathbf{R}^{m},$
and $\Gamma>0$ are problem data.

\paragraph{Low-rank optimization}
We can write low-rank optimization problems in
the form of problem \eqref{eq:original_problem-1}, which are common in machine learning applications such as collaborative filtering
\cite[pp. 279-281]{jain2017non}, design of online recommendation
systems \cite{Mazumder2010,Candes2009}, bandit optimization \cite{jun2019bilinear},
data compression \cite{Gress2014,Mikolov2013,Srikumar2014}, and low
rank kernel learning~\cite{Bach2013}.
In these applications,
the constraint set $\mathcal{X}$ decomposes as $\mathcal{X}=\text{\ensuremath{\mathcal{C}}}\bigcap\mathcal{N}$,
where $\mathcal{C}$ is a compact convex set, and 
\begin{equation}
\mathcal{N}=\{X\in\rl^{m\times d}\mid\rank(X)\leq r\}. \label{eq:low-rank-set}
\end{equation}

In these optimization problems, $\mathcal{C}$ can be matrix-norm ball, Frobenius-norm ball, hyperplane/half-space induced by trace \cite{bertsimas2020mixed,bertsimas2023optimal}.
 In this paper,
we apply \textsf{NExOS }to solve the affine rank minimization problem:
\begin{equation}
\begin{array}{ll}
\textup{minimize} & \left\Vert \mathcal{A}(X)-b\right\Vert _{2}^{2}+(\beta/2)\|X\|_{F}^{2}\\
\textup{subject to} & \mathop{{\bf rank}}(X)\leq r, \quad \|X\|_{2}\leq\Gamma,
\end{array}\tag{RM}\label{eq:lrm-estimation}
\end{equation}
where $X\in\mathbf{R}^{m\times d}$ is the decision variable, $b\in\mathbf{R}^{k}$
is noisy measurement data, and $\mathcal{A}:\mathbf{R}^{m\times d}\to\mathbf{R}^{k}$
is a linear map. The parameter $\Gamma>0$ is the upper bound
for the spectral norm of $X$. The affine map $\mathcal{A}$ is
determined by $k$ matrices $A_{1},\ldots,A_{k}$ in $\mathbf{R}^{m\times d}$
where $\mathcal{A}(X)=(\mathbf{tr}(A_{1}^{T}X),\ldots,\mathbf{tr}(A_{k}^{T}X)).$
We present several numerical experiments to solve \eqref{eq:lrm-estimation}
using \textsf{NExOS }for both synthetic and real-world datasets
in \S\ref{subsec:Affine-rank-minimization}.

\subsection{Related work\label{subsec:Related-work}}

\paragraph{Convex relaxation approach} Due to the presence of the nonconvex set $\mathcal{X},$ the nonconvex problem \eqref{eq:original_problem-1} is $\mathcal{NP}$-hard \cite{Hardt2014}.
A common way to deal with this issue is to avoid this inherent nonconvexity
altogether by convexifying the original problem. The relaxation of
the sparsity constraint leads to the popular \textsf{LASSO} formulation
and its variants \cite{Hastie2015}, whereas relaxation of the
low-rank constraints produces the nuclear norm based convex models
\cite{Fazel2008}. 

The basic advantage of the convex relaxation technique
is that, in general, a globally optimal solution to a convex problem
can be computed reliably and efficiently \cite[\S 1.1]{boyd2004convex},
whereas for nonconvex problems a local optimal solution is often the
best one can hope for. Furthermore, if certain statistical assumptions
on the data generating process hold, then it is possible to recover
exact solutions to the original nonconvex problems with high probability
by solving the convex relaxations (see \cite{Hastie2015} and the
references therein). 

However, when stringent assumptions do not
hold, then solutions to the convex formulations can be of poor quality
and may not scale very well \cite[\S 6.3 and \S 7.8]{jain2017non}.
In this situation, the nonconvexity of the original problem must be
confronted directly, because such nonconvex formulations capture the
underlying problem structures more accurately than their convex counterparts. 

\paragraph{First-order methods}
To that goal, first-order algorithms such as hard thresholding algorithms, \emph{e.g.}, \textsf{IHT} \cite{blumensath2008iterative}, \textsf{NIHT} \cite{blumensath2010normalized}, \textsf{HTP} \cite{foucart2011hard}, \textsf{CGIHT} \cite{blanchard2015cgiht}, address nonconvexity
in sparse and low-rank optimization by implementing variants of projected
gradient descent with projection taken onto the sparse and/or low-rank
set.

While these first-order methods have been successful in recovering low-rank
and sparse solutions in underdetermined linear systems, they too require
assumptions on the data such as the \emph{restricted isometry property
}for recovering true solutions \cite[\S 7.5]{jain2017non}. Furthermore,
to converge to a local minimum, hard thresholding algorithms require
the spectral norm of the measurement matrix to be less than one, which
is a restrictive condition \cite{blumensath2008iterative}. 

Besides
hard thresholding algorithms, heuristics based on first-order algorithms
such as the alternating direction method of multipliers \textsf{(ADMM)}
have gained a lot of traction in the last few years. 
Though \textsf{ADMM}
was originally designed to solve convex optimization problems,
since
the idea of implementing this algorithm as a general purpose heuristic
to solve nonconvex optimization problems was introduced in \cite[\S 9.1-9.2]{Boyd2011}, \textsf{ADMM}-based heuristics have been applied successfully to approximately
solve nonconvex problems in many different application areas \cite{Takapoui2017,Diamond2018}.

However, the biggest drawback of these heuristics based on first-order methods comes from the fact that
they take an algorithm designed to solve convex problems and apply
it verbatim to a nonconvex setup. As a result, these algorithms often
fail to converge, and even when they do, it need not be a local minimum,
let alone a global one \cite[\S 2.2]{Takapoui2017a}. Also,
empirical evidence suggests that the iterates of these algorithms
may diverge even if they come arbitrarily close to a locally optimal
solution during some iteration. 
The main reason is that these heuristics do
not establish a clear relationship between the local minimum of problem \eqref{eq:original_problem-1}
and the fixed point set of the underlying operator that controls the
iteration scheme. An alternative approach that has been quite successful
empirically in finding low-rank solutions is to consider an unconstrained
problem with Frobenius norm penalty and then using alternating minimization
to compute a solution \cite{UdellGlrm}. However, the alternating
minimization approach may not converge to a solution and should be
considered a heuristic \cite[\S 2.4]{UdellGlrm}. 

\paragraph{Discrete optimization approach}
 For these reasons above, in the last few years, there has been significant
interest in addressing the nonconvexity present in many optimization problems
directly via a discrete optimization approach. In this way, a particular
nonconvex optimization problem is formulated exactly using discrete
optimization techniques and then specialized algorithms are developed
to find a certifiably optimal solution. This approach has found considerable
success in solving machine learning problems with sparse and low-rank
optimization \cite{BertsimasMLLens,tillmann2021cardinality}. A mixed
integer optimization approach to compute near-optimal solutions for
sparse regression problem, where problem dimension $d=1000$, is computed in \cite{bertsimas2016best}.
In \cite{bertsimas2020sparse}, the authors propose a cutting plane
method for a similar problem, which works well with mild sample correlations
and a sufficiently large dimension. In \cite{hazimeh2020fast}, the
authors design and implement fast algorithms based on coordinate descent
and local combinatorial optimization to solve sparse regression problem
with a three-fold speedup where $d\approx10^{6}$. In \cite{bertsimas2020mixed},
the authors propose a framework for modeling and solving low-rank
optimization problems to certifiable optimality via symmetric projection
matrices. 

However, the runtime of these discrete optimization based algorithms can often become
prohibitively long as the problem dimensions grow \cite{bertsimas2016best}.
Also, these discrete optimization algorithms have efficient implementations
only for a narrow class of loss functions and constraint sets; they
do not generalize well if a minor modification is made to the problem
structure, and in such a case they often fail to find a solution point
in a reasonable amount of time even for smaller dimensions \cite{BertsimasMLLens}.
Furthermore, one often relies on commercial softwares, such as \texttt{Gurobi},
\texttt{Mosek}, or \texttt{Cplex} to solve these discrete optimization
problems, thus making the solution process somewhat opaque \cite{bertsimas2016best,tillmann2021cardinality}.

\subsection{Contributions}
The main contribution of this work is to propose \textsf{NExOS}: a first-order algorithm
tailored for nonconvex optimization problems of the form \eqref{eq:original_problem-1}. The term \emph{exterior-point} originates from the fact that the iterates
approach a local minimum from outside of the feasible region; it is
inspired by the convex exterior-point method first proposed by Fiacco
and McCormick in the 1960s \cite[\S 4]{fiacco1990nonlinear}. By exploiting the underlying geometry of the constraint set, we construct an iterative
method that finds a locally optimal point of the original problem via an outer loop consisting of increasingly accurate penalized
formulations of the original problem by reducing only one penalty parameter. Each penalized problem is then solved by
applying an inner algorithm that implements a variant of the Douglas-Rachford splitting
algorithm. 

We prove that \textsf{NExOS}, besides avoiding the drawbacks of convex
relaxation and discrete optimization approach, has the following favorable features. First, the penalized problem has strong convexity
and smoothness around local minima, but can be made arbitrarily close
to the original nonconvex problem by reducing the penalty parameter.
Second, under mild regularity conditions, the inner algorithm finds local minima for the penalized problems
at a linear convergence rate, and as the penalty parameter goes to
zero, the local minima of the penalized problems converge to a local
minimum of the original problem. Furthermore, we show that, when those regularity conditions 
do not hold, the inner algorithm is still guaranteed to subsequentially converge to a first-order stationary point of the penalized problem at the 
rate $o(1/\sqrt{k})$.

We implement \textsf{NExOS} in the open-source \texttt{Julia} package
\texttt{NExOS.jl} and test it extensively on many synthetic and real-world instances of different
nonconvex optimization problems of substantial current interest. We demonstrate that \textsf{NExOS} very quickly computes
solutions that are competitive with or better than specialized algorithms on various performance measures. 
\texttt{NExOS.jl}
is available at \url{https://github.com/Shuvomoy/NExOS.jl}. %

}

\paragraph{Organization of the paper}

The rest of the paper is organized as follows. We describe our \textsf{NExOS} framework in \S\ref{subsec:our_approach}. We provide convergence
analysis of the algorithm in \S\ref{sec:Convergence-analysis}. Then
we demonstrate the performance of our algorithm on several nonconvex
optimization problems of significant current interest in \S\ref{sec:Numerical-experiments}.
The concluding remarks are presented in \S\ref{sec:Conclusion}.

\section{Our approach\label{subsec:our_approach}}

 The
backbone of our approach is to address the nonconvexity by working
with an asymptotically exact nonconvex penalization of problem \eqref{eq:original_problem-1},
which enjoys local convexity around local minima. We use the notation $\iota_\mathcal{X}(x)$ that denotes the indicator function of the set $\mathcal{X}$ at $x$, which is $0$
if $x\in\mathcal{X}$ and $\infty$ else. Using this, we can write problem \eqref{eq:original_problem-1} as an unconstrained optimization problem, where the objective is $f(x)+(\beta/2)\|x\|^{2}+\iota_\mathcal{X}(x)$. In our penalization,
we replace the indicator function $\ind$ with its \emph{Moreau envelope
}with positive parameter $\mu$:
\begin{align}
\morind(x) & =\min_{y}\{\ind(y)+(1/2\mu)\|y-x\|^{2}\}=(1/2\mu)d^{2}(x),\label{eq:samaja}
\end{align}
where $d(x)$ is the Euclidean distance of the point $x$ from the
set $\mathcal{X}$. 

{
\paragraph{Properties of Moreau envelope for a nonconvex set. }

The function $\morind$, though nonconvex, has many desirable attributes
that greatly simplify design and convergence analysis of our algorithm.
We summarize these properties below;  See \cite[Proposition 12.9]{Bauschke2017} for the first four properties,
and Proposition \ref{Attainment-of-local-min} in \S\ref{sec:Convergence-analysis}
for the last one.  
\begin{enumerate}
\item \textbf{Bounded.} The function $\morind$ is bounded on every compact
set. In contrast, $\ind$ is an extended valued function that takes
the value $+\infty$ outside the set $\mathcal{X}$. 
\item \textbf{Finite and jointly continuous.} For every $\mu>0$ and $x\in\eu$,
the function $\morind(x)$ is jointly continuous and finite. Therefore,
$\morind$ is continuous on $\eu$. In contrast, the indicator function
$\ind$ is not continuous. 
\item \textbf{Accuracy of approximation controlled by $\mu.$ }With decreasing
$\mu,$ the approximation $\morind$ monotonically increases to $\ind$,
\emph{i.e.}, for any positive $\mu_{1},\mu_{2}$ such that $0\leq\mu_{1}\leq\mu_{2}$,
we have 
\[
0\leq\prescript{\mu_{2}}{}{\ind}(x)\leq\prescript{\mu_{1}}{}{\ind}(x)\leq\ind(x)
\]
for any $x\in\mathbf{E}$. 
\item \textbf{Asymptotically equal to $\ind$.} The approximation $\morind$
is asymptotically equal to $\ind$ as $\mu$ goes to zero, \emph{i.e.},
we have the point-wise limit 
\[
\lim_{\mu\downarrow0}\morind(x)=\ind(x)
\]
for all $x\in E$. 
\item \textbf{Local convexity and differentiability around points of interest.}
Adding any quadratic regularizer to $\morind$ makes the sum locally
convex and differentiable around points of interest. To be precise,
if at $x$, the set $\mathcal{X}$ is prox-regular, then for any value
of $\beta>0,$ the function $\morind(x)+\frac{\beta}{2}\|x\|^{2}$
is convex and differentiable on a neighborhood of $x$. 
\end{enumerate}

The favorable features of $\morind$ motivate us
to consider the following penalization formulation of problem \eqref{eq:original_problem-1} denoted by problem \eqref{eq:smoothed-opt}, where the subscript $\mu$ indicates the penalty parameter:}
\begin{equation}
\begin{array}{ll}
\textup{minimize} & f(x)+\regind(x),\end{array}\tag{\ensuremath{\mathcal{P}_{\mu}}}\label{eq:smoothed-opt}
\end{equation}
where $\regind\equiv\morind+(\beta/2)\|\cdot\|^{2}$, $x\in\eu$ is the decision variable, and $\mu$ is a positive
\emph{penalty parameter}. We call the cost function in problem \eqref{eq:smoothed-opt}
an \emph{exterior-point minimization function}; the term is inspired
by \cite[\S 4.1]{fiacco1990nonlinear}. The notation $\regind\equiv\morind+(\beta/2)\|\cdot\|^{2}$
introduced in problem \eqref{eq:smoothed-opt} not only reduces notational
clutter, but also alludes to a specific way of splitting the objective
into two summands $f$ and $\regind$, which will ultimately allow
us to establish convergence of our algorithm in \S\ref{sec:Convergence-analysis}.
Because $\morind$ is an asymptotically exact approximation of $\iota_\mathcal{X}$
as $\mu\to0$, solving problem \eqref{eq:smoothed-opt} for a small enough
value of the penalty parameter $\mu$ suffices for all practical purposes.

\begin{figure}
\begin{centering}
\includegraphics[scale=0.9]{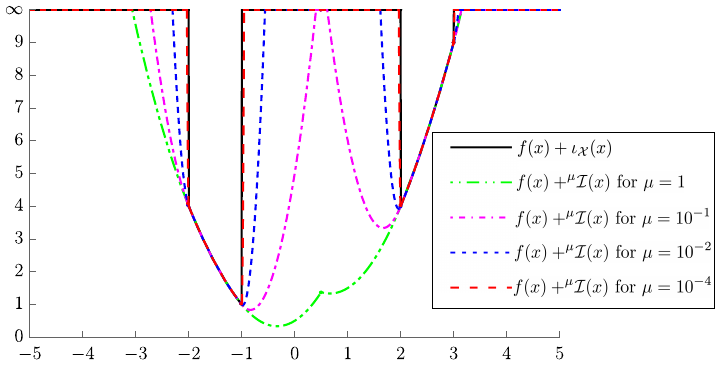}\caption{An illustration of how the penalized cost function in problem \eqref{eq:smoothed-opt} compares against the original cost function in problem \eqref{eq:original_problem-1}
for different values of $\mu$.  Note that the regularization parameter $\beta$ is kept fixed at its initial value $1$ throughout.  \label{fig:An-illustration-of-exterior-point-penalty-function} }
\par\end{centering}
\end{figure}

{
To provide intuition on how the exterior-point minimization function
in problem \eqref{eq:smoothed-opt} compares against the original minimization
function in problem \eqref{eq:original_problem-1}, we provide an illustrative
one-dimensional example in Figure \ref{fig:An-illustration-of-exterior-point-penalty-function}.
Figure \ref{fig:An-illustration-of-exterior-point-penalty-function}
captures all the key properties of our penalization scheme. In this figure, $f=(1/2)(\cdot)^{2}$,
$\beta=1$, $\mathcal{X}=[-2,-1]\bigcup[2,3]$. The problem has two local minima, one
at $-1$ and one at $-2$. We see that for larger values of $\mu$, problem \eqref{eq:smoothed-opt} is not a good approximation of problem \eqref{eq:original_problem-1},
but around each local minimum there is a relatively large region where
$f+\regind$ is strongly convex and smooth. As $\mu$ gets smaller,
problem \eqref{eq:smoothed-opt} becomes a more accurate approximation of 
problem \eqref{eq:original_problem-1}, though the regions of convexity and
smoothness around local minima shrink. For $\mu=10^{-4},$ problem \eqref{eq:smoothed-opt}
is identical to problem \eqref{eq:original_problem-1} for all practical purposes.  Note that the regularization parameter $\beta$ is kept fixed at its initial value $1$ throughout.

}

Now that we have intuitively justified intuition behind working with \eqref{eq:smoothed-opt},
we are in a position to present our algorithm.

\begin{algorithm} 
\hrule
\hrule
\medskip
 \DontPrintSemicolon    
\SetKwInput{KwIn}{given} 
\SetKwInput{KwOut}{output}    
\KwIn{regularization parameter $\beta>0$, an initial point $z_\textup{init}$, initial penalty parameter $\mu_\textup{init}$, minimum penalty parameter $\mu_\textup{min}$, tolerance for the fixed point gap $\epsilon$ for each inner iteration, tolerance for stopping criterion $\delta$ for the outer iteration, and multiplicative factor $\rho \in (0,1)$. }   
\hrule
\it{Initialization}. {$\mu \coloneqq \mu_\textup{init}$, and $z^0 \coloneqq z_\textup{init}$.} \;  
\it{Outer iteration}. \While{stopping criterion is not met} %
{    
\it{Inner iteration.} 
\rm{Using Algorithm \ref{algo:prs}, compute $x_\mu, y_\mu$, and $z_\mu$ that solve problem $(\mathcal{P}_\mu)$ with tolerance $\epsilon$, where $z_{\mu}^0 \coloneqq z^0$ is input as the initial point.}  \; 
\it{Stopping criterion.} 
\textbf{$\textup{quit}$} if \rm{$\lvert\left(f(\proj x_{\mu})+(\beta/2)\|\mathbf{\Pi}_\mathcal{X} x_{\mu}\|^{2}\right)-\left(f(x_{\mu})+\regind(x_{\mu})\right)\rvert\leq\delta$.} 

\it{Set initial point for next inner iteration.} 
\rm{$z^0 \coloneqq z_\mu$.} \; 
\it{Update $\mu$.} 
\rm{$\mu \coloneqq \rho \mu$.}
}   
\Return{$x_\mu, y_\mu,$ \textup{and} $z_\mu$}\;  
\caption{Nonconvex Exterior-point Optimization Solver (\textsf{NExOS}). Here $\proj(x)$ denotes the Euclidean projection of $x$ 
on the nonconvex set $\mathcal{X}$.}  
\label{alg:exterior-point-method}  
\medskip
\hrule
\hrule
\end{algorithm}

\begin{algorithm}  
\hrule
\hrule
\medskip
\DontPrintSemicolon   
\SetKwInput{KwIn}{given} 
\SetKwInput{KwOut}{output}   
\KwIn{starting point $z ^{0}$, tolerance for the fixed point gap $\epsilon$, and proximal parameter $\gamma>0$.}  
\it{Initialization}. {$n \coloneqq 0$, $\kappa  \coloneqq 1/(\beta\gamma+1)$, $\theta  \coloneqq \mu / (\gamma \kappa + \mu)$.}\\
\While{$\| x ^{n} - y ^{n} \| > \epsilon$ }
{    
{Compute $x ^{n+1}  \coloneqq\prox_{\gamma f}\left( z ^{n} \right)$.}\\
{Compute $\tilde{y} ^{n+1}  \coloneqq\kappa\left(2x ^{n+1} -z ^{n} \right)$.}\\
{Compute $y ^{n+1}   \coloneqq\theta \tilde{y} ^{n+1}  +\left(1-\theta\right)\proj\left(\tilde{y} ^{n+1} \right)$.}\\ 
{Compute $z ^{n+1}   \coloneqq z ^{n} +y ^{n+1} -x ^{n+1} $.}\\
{Update $n \coloneqq n+1$.}
}  
\Return{$x ^{n}, y ^{n},$ \textup{and} $z ^{n}$.}\; 
\caption{Inner Algorithm  for problem ($\mathcal{P}_\mu$). Here  $\proj(x)$ denotes the Euclidean projection of $x$ 
on the nonconvex set $\mathcal{X}$, and  $\prox_{\gamma f}$ denotes the proximal operator of $f$ with parameter $\gamma >0$ as defined in \eqref{eq:prox-def}.}  
\label{algo:prs}  
\medskip
\hrule
\hrule
\end{algorithm}

\paragraph{Algorithm description}

Algorithm \ref{alg:exterior-point-method} outlines \textsf{NExOS}.
The main part is an outer loop that solves a sequence of penalized
problems of the form problem \eqref{eq:smoothed-opt} with strictly decreasing
penalty parameter $\mu$, until the termination criterion is met,
at which point the exterior-point minimization function is a sufficiently
close approximation of the original cost function. For each $\mu$,
problem \eqref{eq:smoothed-opt} is solved by an inner algorithm, denoted
by Algorithm \ref{algo:prs}. 

{
One can derive Algorithm \ref{algo:prs}
by applying Douglas-Rachford splitting (\textsf{DRS}) \cite[page 401]{Bauschke2017}
to problem \eqref{eq:smoothed-opt} as follows. If we apply
Douglas-Rachford splitting \cite[page 401]{Bauschke2017} to problem (\ref{eq:smoothed-opt})
with penalty parameter $\mu$, we have the following variant with
three sub-iterations:
\begin{equation}
\begin{aligned}x^{n+1} & =\prox_{\gamma\regf}\left(z^{n}\right)\\
y^{n+1} & =\prox_{\gamma\regind}\left(2x^{n+1}-z^{n}\right)\\
z^{n+1} & =z^{n}+y^{n+1}-x^{n+1}.
\end{aligned}
\tag{DRS}\label{eq:ADMM_orig-1-1}
\end{equation}

The computational
cost for $\prox_{\gamma\regind}$ is the same as computing a projection
onto the constraint set $\mathcal{X}$, as stated in Lemma \ref{Lem:computing_complicated_projection}
below; this result follows from \cite[Theorem 6.13, Theorem 6.63]{beck2017first}.
It should be noted that \cite[Theorem 6.13, Theorem 6.63]{beck2017first}
assume convexity of the functions in the theorem statements, but its
proof does not require convexity and works for nonconvex functions
as well.

\begin{lem}[Computing $\prox_{\gamma\regind}(x)$ \label{Lem:computing_complicated_projection}]
Consider the nonconvex compact constraint set $\mathcal{X}$ in problem (\ref{eq:original_problem-1}).
Denote $\kappa=1/(\beta\gamma+1)\in[0,1]$ and $\theta=\mu/(\gamma\kappa+\mu)\in[0,1]$.
Then, for any $x\in\mathbf{E}$, and for any $\mu,\beta,\gamma>0$,
we have $\prox_{\gamma\regind}\left(x\right)  =\theta\kappa x+\left(1-\theta\right)\proj\left(\kappa x\right)$.
\end{lem}
Finally, combining (\ref{eq:ADMM_orig-1-1}), \cite[Theorem 6.13]{beck2017first},
and Lemma \ref{Lem:computing_complicated_projection}, we arrive at
Algorithm \ref{algo:prs}. }

\paragraph{Algorithm subroutines}
The inner algorithm requires two subroutines, evaluating (i) $\prox_{\gamma f}(x)$,
which is the proximal operator of the convex function $f$ at the
input point $x$, and (ii) $\proj(x)$, which is a projection of $x$ 
on the nonconvex set $\mathcal{X}$. We discuss now how we compute
them in our implementation. To that goal, we recall that, for a
function $g$ (not necessarily convex) its proximal operator $\prox_{\gamma g}$ and Moreau envelope  $\prescript{\gamma}{}{g}$, where $\gamma>0$, are defined
as: 
\begin{equation}
\begin{alignedat}{1} & \prox_{\gamma g}(x)=\mathop{\textrm{argmin}}_{y\in\mathbf{E}}\left(g(y)+(1/2\gamma)\|y-x\|^{2}\right),\\
 & \prescript{\gamma}{}{g}(x)=\textrm{min}_{y\in\mathbf{E}}\left(g(y)+(1/2\gamma)\|y-x\|^{2}\right).
\end{alignedat}
\label{eq:prox-def}
\end{equation}
\paragraph{Computing proximal operator of $f$} For the convex function $f$, $\mathbf{prox}_{\gamma f}$
is always single-valued and computing it is equivalent to solving
a convex optimization problem, which often can be done in closed form
for many relevant cost functions in machine learning \cite[pp. 449-450]{beck2017first}.
If the proximal operator of $f$ does not admit a closed form solution,
then we solve the corresponding convex optimization problem \eqref{eq:prox-def}
to a high precision solution. For this purpose, we can select any convex optimization solver
supported by \texttt{MathOptInterface,} which is the abstraction layer
for optimization solvers in \texttt{Julia}. 

\paragraph{Computing projection onto $\mathcal{X}$} The notation $\proj(x)$ denotes the\textbf{ }\emph{projection operator
}of $x$ onto the constraint set $\mathcal{X}$, defined as 

\[\mathbf{\Pi}_\mathcal{X}(x)=\prox_{\gamma\iota_\mathcal{X}}(x)=\textrm{argmin}_{y\in\mathcal{X}}(\|y-x\|^{2}).\]

A list of nonconvex sets that are easy to project onto
can be found in \cite[\S 4]{Diamond2018}, this includes nonconvex
sets such as boolean vectors with fixed cardinality, vectors with
bounded cardinality, quadratic sets, matrices with bounded singular
values, matrices with bounded rank etc. If $\mathcal{X}$ is in this
list, then we project onto $\mathcal{X}$ directly.

Now consider the
case where the constraint set $\mathcal{X}$ decomposes as $\mathcal{X}=\text{\ensuremath{\mathcal{C}}}\bigcap\mathcal{N}$,
where $\mathcal{N}$ is a nonconvex set with tractable projection
and $\mathcal{C}$ is any compact convex set. In this setup, let $\iota_{\mathcal{C}}$
and $\iota_{\mathcal{N}}$ be the indicator functions of $\mathcal{C}$
and $\mathcal{N}$, respectively. Defining $\phi=f+\mathcal{\iota}_{\mathcal{C}}$,
we write problem \eqref{eq:original_problem-1} as: $\min_{x\in\mathbf{E}}\phi(x)+(\beta/2)\|x\|^{2}+\iota_{\mathcal{N}}(x).$

For any convex function $\phi$, its Moreau envelope $\prescript{\nu}{}{\phi}$,
for any $\nu>0$, has the following three desirable features. 

\begin{enumerate}
    \item For every $x\in\eu$ we have $\prescript{\nu}{}{\phi}(x)\leq\phi(x)$
and $\prescript{\nu}{}{\phi}(x)\to\phi(x)$ as $\nu\to0$ \cite[Theorem 1.25]{Rockafellar2009}. 
\item  we have $x^{\star}\in\mathop{\textrm{argmin}}_{x\in\eu}\phi(x)$ if
and only if $x^{\star}\in\mathop{\textrm{argmin}}_{x\in\eu} \prescript{\nu}{}{\phi}(x)$
with the minimizer $x^{\star}$ satisfying $\phi(x^{\star})=\prescript{\nu}{}{\phi}(x^{\star})$
\cite[Corollary 17.5]{Bauschke2017}. 

\item the Moreau envelope $\prescript{\nu}{}{\phi}$ is convex, and smooth
(\emph{i.e.}, it is differentiable and its gradient is Lipschitz continuous)
everywhere irrespective of the differentiability or smoothness of
the original function $\phi$. The gradient is: $\prescript{\nu}{}{\phi}(x)=\left(x-\prox_{\nu\phi}(x)\right)/\nu,$
which is $(1/\nu)-$Lipschitz continuous \cite[Proposition 12.29]{Bauschke2017}. 
\end{enumerate}

These properties make $\prescript{\nu}{}{\phi}$ a smooth approximation
of $\phi$ for a small enough $\nu$. Hence, we work with the following approximation
of the original problem: $\min_{x}\prescript{\nu}{}{\phi}+(\beta/2)\|x\|^{2}+\iota_{\mathcal{N}}(x),$
where we replace $f$ with $\prescript{\nu}{}{\phi}$
and $\iota_\mathcal{X}$ with $\iota_{\mathcal{N}}$ in Algorithms \ref{alg:exterior-point-method}
and \ref{algo:prs}. The proximal operator of $\prescript{\nu}{}{\phi}$
can be computed using $\prox_{\gamma\prescript{\nu}{}{\phi}}(x)=x+(\gamma/(\gamma+\nu))(\prox_{(\gamma+\nu)\phi}(x)-x),$
where computing $\prox_{(\gamma+\nu)\phi}(x)$ corresponds to solving
the following convex optimization problem $\mathop{\textrm{argmin}}_{y\in C}\phi(y)+1/(2(\gamma+\nu))\|y-x\|^{2}$,
which follows from \cite[Proposition 24.8]{Bauschke2017}.

{

\begin{rem}[Reasons for choosing Douglas-Rachford splitting as the inner algorithm.]

Problem \eqref{eq:smoothed-opt} involves minimizing the sum of two
functions: a convex function $\regf$ and a nonconvex function $\regind$.
As the objective is split into two parts in problem \eqref{eq:smoothed-opt},
selecting any other two-operator splitting algorithm
(\emph{e.g.}, forward-backward splitting \cite[page 25]{ryu2016primer},
Chambolle-Pock algorithm \cite[page 32]{ryu2016primer}, ADMM \cite{boydProx} \emph{etc.})
 can work as the inner algorithm in principle. However, in
the context of our problem setup, Douglas-Rachford splitting might
be the most suitable choice for the following reasons. 

\begin{enumerate}
\item We have picked Douglas-Rachford splitting over ADMM, because Douglas-Rachford operates on the original nonconvex problem, whereas ADMM can be viewed as Douglas-Rachford splitting on the dual of the original nonconvex problem \cite{themelis2020douglas}. As strong duality usually does not hold when the primal problem is nonconvex, it seems more intuitive to work with the nonconvex problem directly over its dual.

\item We favored Douglas-Rachford splitting over proximal gradient method, because even when the problem is convex, Douglas-Rachford splitting converges under more general conditions, whereas proximal gradient method require more restrictive conditions to converge \cite[page 49]{ryu2022large}. Hence, we believe that Douglas-Rachford splitting represents the most natural choice for the inner algorithm over the proximal gradient method.

    \item Douglas-Rachford splitting is favorably unique in contrast
with other two-operator splitting methods, as Douglas-Rachford splitting is the only
two-operator splitting method that satisfies the following properties
simultaneously \cite{ryu2020uniqueness}: (i) it is constructed only with scalar multiplication, addition, and proximal
operators, (ii) it computes proximal operators only once every iteration, (iii) it converges unconditionally for maximally monotone operators, and (iv) it does not increase the problem size.

In \S\ref{sec:Convergence-analysis}, some of these desirable properties
of Douglas-Rachford splitting are exploited to establish convergence.
While other operator splitting algorithms may work to establish convergence
as well, some of the unique features of Douglas-Rachford splitting
will be lost \cite{ryu2020uniqueness}.

\end{enumerate}

\end{rem}
}

\section{Convergence analysis\label{sec:Convergence-analysis}}

This section is organized as follows. We start with the definition
of the key geometry property of sets involving sparse and low-rank optimization problems. Then we define the local minima of such problems, followed by the assumptions we use in our convergence
analysis. We next discuss the convergence roadmap, where the first
step involves showing that the exterior point minimization function
is locally strongly convex and smooth around local minima, and the
second step entails connecting the local minima with the underlying
operator controlling \textsf{NExOS}. Then, we present the main result, which shows that, under mild regularity conditions, the inner algorithm of \textsf{NExOS} finds local minima for the penalized problems at a linear convergence rate, and as the penalty parameter goes to zero, the local minima of the penalized problems converge to a local
minimum of the original problem. Furthermore, we show that, when those regularity conditions 
do not hold, the inner algorithm is still guaranteed to subsequentially converge to a first-order stationary point at the 
rate $o(1/\sqrt{k})$.

{
The key geometric property of sparse and low-rank constraint sets that we use in our convergence analysis is  {\it prox-regularity at local minima}, \emph{i.e.}, having single-valued
Euclidean projection around local minima \cite{PoliRocka2000}. Prox-regularity of a set at a point is defined as follows.

\begin{defn}[Prox-regular set \cite{PoliRocka2000}]
A nonempty closed set $\mathcal{S}\subseteq\mathbf{E}$ is prox-regular
at a point $x\in\mathcal{S}$ if projection onto $\mathcal{S}$ is
single-valued on a neighborhood of $x$. The set
$\mathcal{S}$ is prox-regular if it is prox-regular at every point
in the set. 
\end{defn}

If the constraint set $\mathcal{X}$ decomposes as $\mathcal{X}=\text{\ensuremath{\mathcal{C}}}\bigcap\mathcal{N}$,
where $\mathcal{C}$ is a compact convex set,
and $\mathcal{N}$ is prox-regular around
local minima, then the feasible set $\mathcal{X}$
inherits the prox-regularity property around local minima from the
set $\mathcal{N}$ (see Lemma \ref{Lem:Prox-regularity-of-X}
in \S\ref{sec:Convergence-analysis}). The set $\mathcal{N}$ in \eqref{eq:low-rank-set} is a prox-regular set at
any point $X\in\mathbf{R}^{m\times d}$ where $\mathop{{\bf rank}}(X)=r$
\cite[Proposition 3.8]{Luke2013}. One can show that $\mathcal{X}$
inherits the prox-regularity property at any $X$ with $\mathop{{\bf rank}}(X)=r$ from the set $\mathcal{N}$;
a formal proof is given in Lemma \ref{Lem:Prox-regularity-of-X}
in Appendix \ref{subsec:Proof-to-rank-sparse-reg}. Similarly, $\mathcal{N}$ in \eqref{eq:sparse-set} is prox-regular
at any point $x$ satisfying $\card(x)=k$ because we can write $\mathbf{card}(x)\leq k$
as a special case of the low-rank constraint by embedding the components
of $x$ in the diagonal entries of a matrix and then using the prox-regularity
of low-rank constraint set. 

In our convergence analysis, we use the prox-regularity property of sparse and low-rank optimization to establish our convergence results, hence \textsf{NExOS} can be applied to problems involving other constraint sets that are prox-regular at local minimal. Some other notable prox-regular sets are as follows. Closed convex
sets are prox-regular everywhere \cite[page 612]{Rockafellar2009}. 
Examples of well-known prox-regular sets that are not convex include sets involving bilinear constraints \cite{bauschke2022projections}, weakly convex sets \cite{vial1983strong}, proximally smooth sets
\cite{clarke1995proximal}, strongly amenable sets \cite[page 612]{Rockafellar2009},
and sets with Shapiro property \cite{shapiro1994existence}. Also,
a nonconvex set defined by a system of finitely many inequality and
equality constraints for which a basic constraint qualification holds
is prox-regular \cite[page 10]{rockafellar2020characterizing}.

}

 We next provide the definition of local minimum of problem (\ref{eq:original_problem-1}).
Recall that, according to our setup the set $\mathcal{X}$ is prox-regular
at local minimum. 
\begin{defn}[Local minimum of problem (\ref{eq:original_problem-1})\label{Def:Local-minima-of-P}]
A point $\bar{x}\in\mathcal{X}$ is a local minimum of problem (\ref{eq:original_problem-1})
if the set $\mathcal{X}$ is prox-regular at $\bar{x},$ and there
exists a closed ball with center $\bar{x}$ and radius $r$, denoted by $\overline{B}(\bar{x};r)$ such that for all
$y\in\mathcal{X}\cap\overline{B}(\bar{x};r)\setminus\{\bar{x}\},$
we have $f(\bar{x})+(\beta/2)\|\bar{x}\|^{2}<f(y)+(\beta/2)\|y\|^{2}$.
\end{defn}
In the definition above, the strict inequality is due to the strongly
convex nature of the objective $f+(\beta/2)\|\cdot\|^{2}$ and follows
from \cite[Proposition 2.1]{auslender1984stability} and \cite[Theorem 6.12]{Rockafellar2009}. We now state and justify the assumptions used in our convergence
analysis.

\begin{assumption}[Strong convexity and smoothness of $f$]\label{assum:strong_convex_smooth}

The function $f$ in problem (\ref{eq:smoothed-opt}) is $\alpha$-strongly
convex and $L$-smooth where $L>\alpha>0$, i.e., $f-(\alpha/2)\|\cdot\|^{2}$
is convex and $f-(L/2)\|\cdot\|^{2}$ is concave.

\end{assumption}

\begin{assumption}[Problem (\ref{eq:original_problem-1}) is not trivial\label{assum:unique_strongly_convex}]

The unique solution to the unconstrained strongly convex problem $\min_{x}f(x)+(\beta/2)\|x\|^{2}$
does not lie in $\mathcal{X}.$

\end{assumption}

Assumption \ref{assum:strong_convex_smooth} corresponds to the function $f+(\beta/2)\| \cdot \|^{2}$ being $(\alpha+\beta)$-strongly convex and $(L+\beta)$-smooth. In our convergence analysis, $\beta>0$ can be arbitrarily small, so it does not fall outside the setup described in \S\ref{sec:Introduction}.
The $L$-smoothness in $f$ is equivalent to its gradient $\nabla f$
being $L-$Lipschitz everywhere on $\eu$ \cite[Theorem 18.15]{Bauschke2017}.
In our convergence analysis, this assumption is required
in establishing linear convergence of the inner algorithms of \textsf{NExOS}.

Assumption \ref{assum:unique_strongly_convex} imposes that a local
minimum of problem (\ref{eq:original_problem-1}) is not the global minimum
of its unconstrained convex relaxation, which does not incur any loss
of generality. We can solve the unconstrained strongly convex optimization
problem $\min_{x}f(x)+(\beta/2)\|x\|^{2}$ and check if the corresponding
minimizer lies in $\mathcal{X}$; if that is the case, then that minimizer
is also the global minimizer of problem (\ref{eq:original_problem-1}), and
there is no point in solving the nonconvex problem. This can be
easily checked by solving an unconstrained convex optimization problem, so Assumption \ref{assum:unique_strongly_convex}
does not cause any loss of generality.

 {
 To discuss our convergence roadmap, we introduce some standard operator theoretic notions as follows. A set-valued operator\emph{ }$\opA:\eu\rightrightarrows\eu$ maps
an element $x$ in $\eu$ to a set $\opA(x)$ in $\eu$; its\emph{
}domain is defined as $\mathop{{\bf dom}}\opA=\{x\in\eu\mid\opA(x)\neq\emptyset\},$
its range is defined as $\mathop{{\bf ran}}\opA=\bigcup_{x\in\eu}\opA(x),$
and it is completely characterized by its graph: $\gra\opA=\{(u,x)\in\eu\times\eu\mid u\in\opA(x)\}.$
Furthermore, we define $\fix\opA=\{x\in\eu\mid x\in\opA(x)\},$ and
$\zer\opA=\{x\in\eu\mid0\in\opA(x)\}.$ For every $x,$ addition of
two operators $\opA_{1},\opA_{2}:\mathbf{E}\rightrightarrows\mathbf{E}$,
denoted by $\opA_{1}+\opA_{2}$, is defined as $(\opA_{1}+\opA_{2})(x)=\opA_{1}(x)+\opA_{2}(x),$
subtraction is defined analogously, and composition of these operators,
denoted by $\opA_{1}\opA_{2},$ is defined as $\opA_{1}\opA_{2}(x)=\opA_{1}(\opA_{2}(x))$;
note that order matters for composition. Also, if $\mathcal{S}\subseteq\mathbf{E}$
is a nonempty set, then $\opA(\mathcal{S})=\cup\{\opA(x)\mid x\in\mathcal{S}\}$.}

We next discuss our convergence roadmap. Convergence
of \textsf{NExOS} is controlled by the \textsf{DRS} operator of problem (\ref{eq:smoothed-opt}):
\begin{equation}
\opT_{\mu}=\prox_{\gamma\regind}\left(2\prox_{\gamma\regf}-\id\right)+\id-\prox_{\gamma\regf},\label{eq:NCDRS-operator}
\end{equation}
where $\mu>0,$ and $\id$ stands for the identity operator in $\eu$,
\emph{i.e.}, for any $x\in\eu$, we have $\id(x)=x$. Using $\opT_{\mu}$,
the inner algorithm---Algorithm \ref{algo:prs}---can be written as 
\begin{equation}
\begin{alignedat}{1}z^{n+1} & =\opT_{\mu}\left(z^{n}\right)\end{alignedat}
\tag{\ensuremath{\mathcal{A}_{\mu}}}\label{eq:convenient-form-1}
\end{equation}
where $\mu$ is the penalty parameter and $z^{n}$ is
initialized at the fixed point from the previous inner algorithm.

To show the convergence of \textsf{NExOS}, we first show that for
some $\mu_{\textrm{max}}>0$, for any $\mu\in(0,\mu_{\textrm{max}}]$,
the exterior point minimization function $\regf+\regind$ is strongly
convex and smooth on some open ball with center $x$ and radius $r_{\textup{max}}$, denoted by $B(\bar{x};r_{\textup{max}})$,
where it will attain a unique local minimum $x_{\mu}$. Then we show
that for $\mu\in(0,\mu_{\textrm{max}}],$ the operator $\opT_{\mu}(x)$
will be contractive in $x$ and Lipschitz continuous in $\mu$, and
connects its fixed point set $\fix\opT_{\mu}$ with the local minima
$x_{\mu}$, via the relationship $x_{\mu}=\prox_{\gamma\regf}(\fix\opT_{\mu}).$
In the main convergence result, we show that for a sequence
of penalty parameters $\mathfrak{M}=\{\mu_{1},\mu_{2},\mu_{3},\ldots,\mu_{N}\}$ and under proper initialization,
if we apply \textsf{NExOS }to $\mathfrak{M},$ then for all $\mu_{m}\in\mathfrak{M},$the
inner algorithm will linearly converge to $x_{\mu_{m}}$, and as $\mu_{N}\to0,$
we will have $x_{\mu_{N}}\to\bar{x}$. Finally, we show that, when the regularity conditions of the prior result
do not hold, the inner algorithm is still guaranteed to subsequentially converge to a first-order stationary point (not necessarily a local minimum) at the 
rate $o(1/\sqrt{k})$.

We next present a proposition that shows that the exterior point minimization
function in problem (\ref{eq:smoothed-opt}) will be locally strongly convex
and smooth around local minima for our selection of penalty parameters,
even though problem (\ref{eq:original_problem-1}) is nonconvex. Furthermore,
as the penalty parameter goes to zero, the local minimum of problem (\ref{eq:smoothed-opt})
converges to the local minimum of the original problem (\ref{eq:original_problem-1}).
So, under proper initialization\textsf{,
NExOS} can solve the sequence of penalized problems $\{\mathcal{P}_{\mu}\}_{\mu\in(0,\mu_{\textrm{init}}]}$
similar to convex optimization problems; we will prove this in our
main convergence result (Theorem \ref{Theorem_Main-convergence-result}). 

\begin{prop}[Attainment of local minimum by $\regf+\regind$\label{Attainment-of-local-min}]
Let Assumptions \ref{assum:strong_convex_smooth} and \ref{assum:unique_strongly_convex}
hold for problem (\ref{eq:original_problem-1}), and let $\bar{x}$ be a local minimum to problem (\ref{eq:original_problem-1}). Then the following hold. 
\begin{enumerate}[label=(\roman*)]
    \item There exist $\mu_{\textup{max}}>0$ and $r_{\textup{max}}>0$
such that for any $\mu\in(0,\mu_{\textup{max}}]$, the exterior point
minimization function $\regf+\regind$ in problem (\ref{eq:smoothed-opt})
is strongly convex and smooth in the open ball $B(\bar{x};r_{\textup{max}})$
and will attain a unique local minimum $x_{\mu}$ in this ball.
    \item As $\mu\to0,$ this local minimum $x_{\mu}$ will go to $\bar{x}$
in limit, i.e., $x_{\mu}\to\bar{x}$. 
\end{enumerate}
\end{prop}
\begin{proof}
See Appendix \ref{subsec:Proof-to-Lemma-Nov-4}.
\end{proof}

Because the exterior point minimization function is locally strongly
convex and smooth, intuitively the \textsf{DRS}
operator of problem (\ref{eq:smoothed-opt}) would behave similar
to that of a \textsf{DRS} operator of a composite convex optimization
problem, but locally. When we minimize a sum of two convex functions
where one of them is strongly convex and smooth, the corresponding
\textsf{DRS} operator is contractive \cite[Theorem 1]{Giselsson17}.
So, we can expect that the \textsf{DRS} operator for problem (\ref{eq:smoothed-opt})
would be locally contractive around a local minimum, which indeed
turns out to be the case as proven in the next proposition. Furthermore,
the next proposition shows that $\opT_{\mu}(x)$ is locally Lipschitz
continuous in the penalty parameter $\mu$ around a local minimum
for fixed $x$. As $\opT_{\mu}(x)$ is locally contractive in $x$
and Lipschitz continuous in $\mu$, it ensures that as we reduce the
penalty parameter $\mu,$ the local minimum $x_{\mu}$ of problem (\ref{eq:smoothed-opt})
found by \textsf{NExOS }does not change abruptly.
\begin{prop}[Characterization of $\opT_{\mu}$\label{Operators-contractive}]
Let Assumptions \ref{assum:strong_convex_smooth} and \ref{assum:unique_strongly_convex}
hold for problem (\ref{eq:original_problem-1}),  and let $\bar{x}$ be a local minimum to problem (\ref{eq:original_problem-1}). Then the following hold. 
\begin{enumerate}[label=(\roman*)]
    \item There exists a contraction factor
$\kappa'\in(0,1)$ such that for any $x_{1},x_{2}\in B(\bar{x};r_{\textup{max}})$
and $\mu\in(0,\mu_{\textup{max}}]$, we have $\left\Vert \opT_{\mu}(x_{1})-\opT_{\mu}(x_{2})\right\Vert   \leq\kappa'\left\Vert x_{1}-x_{2}\right\Vert$. 
    \item For any $x\in B(\bar{x};r_{\textup{max}}),$ the operator $\opT_{\mu}(x)$
is Lipschitz continuous in $\mu$, i.e., there exists an $\ell>0$
such that for any $\mu_{1},\mu_{2}\in(0,\mu_{\textup{max}}]$ and
$x\in B(\bar{x};r_{\textup{max}})$, we have 
$\left\Vert \opT_{\mu_{1}}(x)-\opT_{\mu_{2}}(x)\right\Vert \leq\ell\|\mu_{1}-\mu_{2}\|.$
\end{enumerate}
\end{prop}
\begin{proof}
See Appendix \ref{subsec:Proof-to-Proposition-operators-contractive}.
\end{proof}
If the inner algorithm (\ref{eq:convenient-form-1})
converges to a point $z_{\mu}$, then $z_{\mu}$ would be a fixed
point of the \textsf{DRS} operator
$\opT_{\mu}$. Establishing the convergence of \textsf{NExOS} necessitates
connecting the local minimum $x_{\mu}$ of problem (\ref{eq:smoothed-opt})
to the fixed point set of $\opT_{\mu}$, which is achieved by the
next proposition. Because our \textsf{DRS} operator locally behaves
in a manner similar to the \textsf{DRS }operator of a convex optimization
problem as shown by Proposition \ref{Operators-contractive}, it is
natural to expect that the connection between $x_{\mu}$ and $z_{\mu}$
in our setup would be similar to that of a convex setup, but in a
local sense. This indeed turns out to be the case as proven in the
next proposition. The statement of this proposition is structurally
similar to \cite[Proposition 25.1(ii)]{Bauschke2017} that establishes
a similar relationship globally for a convex setup, whereas our result
is established around the local minima of problem (\ref{eq:smoothed-opt}). 
\begin{prop}[Relationship between local minima of problem (\ref{eq:original_problem-1})
and $\fix\opT_{\mu}$\label{thm:fixed_point_relationship}]
Let Assumptions \ref{assum:strong_convex_smooth} and \ref{assum:unique_strongly_convex}
hold for problem (\ref{eq:original_problem-1}). Let $\bar{x}$ be a local minimum to problem (\ref{eq:original_problem-1}), and $\mu\in(0,\mu_{\textup{max}}]$. Then, $x_{\mu}=\argmin_{B(\bar{x};r_{\textup{max}})}f(x)+\regind(x)=\prox_{\gamma\regf}\left(\fix\opT_{\mu}\right)$,
where the sets $\fix\opT_{\mu},$ and $\prox_{\gamma\regf}\left(\fix\opT_{\mu}\right)$
are singletons over $B(\bar{x};r_{\textup{max}})$.
\end{prop}
\begin{proof}
See Appendix \ref{subsec:fixed_point_lemma}.
\end{proof}
Before we present the main convergence result, we provide a helper
lemma, which shows how the distances between $x_{\mu},z_{\mu}$ and
$\bar{x}$ change as $\mu$ is varied in Algorithm \ref{alg:exterior-point-method}.
Additionally, this lemma provides the range for the proximal parameter
$\gamma$. If $\mathcal{X}$ is a bounded set satisfying $\|x\|\leq D$
for all $x\in\mathcal{X},$ then term $\textup{max}_{x\in B(\bar{x};r_{\textup{max}})}\|\nabla f(x)\|$
in this lemma can be replaced with $L\times D$.  

\begin{lem}[Distance between local minima of problem (\ref{eq:original_problem-1})
with local minima of problem (\ref{eq:smoothed-opt})\label{thm:convergence_outer_iterations}]
Let Assumptions \ref{assum:strong_convex_smooth} and \ref{assum:unique_strongly_convex}
hold for problem (\ref{eq:original_problem-1}), and let $\bar{x}$ be a local minimum to problem (\ref{eq:original_problem-1})
over $B(\bar{x};r_{\textup{max}})$. Then the following hold.
\begin{enumerate}[label=(\roman*)]
    \item For any $\mu\in(0,\mu_{\textup{max}}]$, the unique local minimum
$x_{\mu}$ of problem (\ref{eq:smoothed-opt}) over $B(\bar{x};r_{\textup{max}})$
satisfies $\Vert x_{\mu}-\bar{x}\Vert < r_{\textup{max}}/\eta'$
for some $\eta'>1$.
    \item Let $z_{\mu}$ be the unique fixed point of $\opT_{\mu}$ over
$B(\bar{x};r_{\textup{max}})$ corresponding to $x_{\mu}$. Then for
any $\mu\in(0,\mu_{\textup{max}}]$, we have $r_{\textup{max}}-\|x_{\mu}-\bar{x}\|>(\eta'-1)r_{\textup{max}}/\eta'$ and $r_{\textup{max}}-\|z_{\mu}-\bar{x}\|>\psi$,
 where $\psi=(\eta'-1)r_{\textup{max}}/\eta'-\gamma\;\textup{max}_{x\in B(\bar{x};r_{\textup{max}})}\|\nabla f(x)\|>0$
 with the proximal parameter $\gamma$ taken to satisfy 
 \[0<\gamma<(\eta'-1)r_{\textup{max}}/\left(\eta'\textup{max}_{x\in B(\bar{x};r_{\textup{max}})}\|\nabla f(x)\|\right).\]
 
Furthermore, $\min_{\mu\in(0,\mu_{\textup{max}}]}\left\{ \left(r_{\textup{max}}-\|z_{\mu}-\bar{x}\|\right)-\psi\right\} >0$.
\end{enumerate}
\end{lem}
\begin{proof}
See Appendix \ref{subsec:Proof-to-outer-iterations}.
\end{proof}

We now present our main convergence results for \textsf{NExOS}. For
convenience, we denote the $n$-th iterates of the inner algorithm
of \textsf{NExOS }for penalty parameter $\mu$ by $\{x_{\mu}^{n},y_{\mu}^{n},z_{\mu}^{n}\}$.
In the theorem, an $\epsilon$-approximate fixed point $\widetilde{z}$
of $\opT_{\mu}$ is defined by $\max\{\|\widetilde{z}-\opT_{\mu}(\widetilde{z})\|,\|z_{\mu}-\widetilde{z}\| \} \leq\epsilon,$
where $z_{\mu}$ is the unique fixed point of $\opT_{\mu}$ over $\ensuremath{B(\bar{x};r_{\textup{max}})}$.
Furthermore, define:
\begin{equation}
\overline{\epsilon}\coloneqq\min\{\min_{\mu\in(0,\mu_{\textup{max}}]}((r_{\textup{max}}-\|z_{\mu}-\bar{x}\|)-\psi)/2,(1-\kappa')\psi\}>0,\label{eq:epsilon_upper_bound}
\end{equation}
where $\kappa'\in(0,1)$ is the contraction factor of $\opT_{\mu}$
for any $\mu>0$ (cf. Proposition \ref{Operators-contractive}) and
the right-hand side is positive due to the third and fifth equations of Lemma \ref{thm:convergence_outer_iterations}(ii).
Theorem \ref{Theorem_Main-convergence-result} states that if we have
a good initial point $z_{\textrm{init}}$ for the first penalty parameter
$\mu_{\textrm{init}}$, then \textsf{NExOS} will construct a finite
sequence of penalty parameters such that all the inner algorithms
for these penalty parameters will linearly converge to the unique local minima of
the corresponding inner problems.   

\begin{thm}[Convergence result for \textsf{NExOS}\label{Theorem_Main-convergence-result}]
Let Assumptions \ref{assum:strong_convex_smooth} and \ref{assum:unique_strongly_convex}
hold for problem (\ref{eq:original_problem-1}), and let $\bar{x}$ be a local minimum to problem (\ref{eq:original_problem-1}). Suppose that the fixed-point tolerance $\epsilon$ for Algorithm
\ref{algo:prs} satisfies $\epsilon\in[0,\overline{\epsilon}),$ where
$\overline{\epsilon}$ is defined in (\ref{eq:epsilon_upper_bound}).
The proximal parameter $\gamma$ is selected to satisfy the fourth equation of Lemma \ref{thm:convergence_outer_iterations}(ii).\textbf{}
In this setup, \textsf{NExOS} will construct a finite sequence of
strictly decreasing penalty parameters $\mathfrak{M}=\{\mu_{1}\coloneqq\mu_{\textup{init}},\mu_{2}=\rho\mu_{1},\mu_{3}=\rho\mu_{2},\ldots\},$
with $\mu_{\textup{init}}\leq\mu_{\textup{max}}$ and $\rho\in(0,1)$,
such that we have the following recursive convergence property. 

For
any $\mu\in\mathcal{M}$, if an $\epsilon$-approximate fixed point
of $\opT_{\mu}$ over $B(\bar{x};r_{\textup{max}})$ is used to initialize
the inner algorithm for penalty parameter $\rho\mu$, then the corresponding
inner algorithm iterates $z_{\rho\mu}^{n}$ linearly converges to
$z_{\rho\mu}$ that is the unique fixed point of $\opT_{\rho\mu}$
over $B\left(\bar{x},r_{\textup{max}}\right)$, and the iterates $x_{\rho\mu}^{n},y_{\rho\mu}^{n}$
linearly converge to $x_{\rho\mu}=\prox_{\gamma\regf}(z_{\rho\mu})$,
which is the unique local minimum to $(\mathcal{P}_{\rho\mu})$ over
$B(\bar{x};r_{\textup{max}})$. 
\end{thm}
\begin{proof}
See Appendix \ref{subsec:Proof-to-Theorem}. 
\end{proof}
From Theorem \ref{Theorem_Main-convergence-result}, we see that an
$\epsilon$-approximate fixed point of $\opT_{\rho\mu}$ over $B(\bar{x};r_{\textup{max}})$
can be computed and then used to initialize the next inner algorithm
for penalty parameter $\rho^{2}\mu$; this chain of logic makes each
inner algorithm linearly converge to the corresponding locally optimal solution.
Finally, for the convergence of the first inner algorithm we have
the following result, which states that if the initial point $z_{\textrm{init}}$
is not ``too far away" from $B(\bar{x};r_{\textrm{max}})$, then the
first inner algorithm of \textsf{NExOS }for penalty parameter $\mu_{1}$
converges to a locally optimal solution of $(\mathcal{P}_{\mu_{1}}).$

\begin{lem}[Convergence of the first inner algorithm\label{Convergence-result-for-fixed-mu}]
 Let $\bar{x}$ be a local minimum to problem (\ref{eq:original_problem-1}),
where Assumptions \ref{assum:strong_convex_smooth} and \ref{assum:unique_strongly_convex}
hold. Let $z_{\textup{init}}$ be the chosen initial point for $\mu_{1}\coloneqq\mu_{\textup{init}}$
such that $\overline{B}(z_{\mu_{1}};\|z_{\textup{init}}-z_{\mu_{1}}\|)\subseteq B(\bar{x};r_{\textup{max}})$,
where $z_{\mu_{1}}$ be the corresponding unique fixed point of $\opT_{\mu_{1}}$.
Then, $z_{\mu_{1}}^{n}$ linearly converges to $z_{\mu_{1}}$ and
both $x_{\mu_{1}}^{n}$ and $y_{\mu_{1}}^{n}$ linearly converge to
the unique local minimum $x_{\mu_{1}}$ of $(\mathcal{P}_{\mu_{1}})$
over $B(\bar{x};r_{\textup{max}})$. 
\end{lem}
\begin{proof}
See Appendix \ref{subsec:Proof-to-Theorem-fixed_mu}.
\end{proof}

We now discuss what can be said if the initial point $z_{\textrm{init}}$
does not necessarily satisfy the conditions stated in Theorem \ref{Theorem_Main-convergence-result}
or Lemma \ref{Convergence-result-for-fixed-mu}. Unfortunately, in
such a situation, we can only show subsequential convergence of the iterates.

\begin{thm}[Convergence result for \textsf{NExOS} for $z_{\textrm{init}}$ that
is far away from $B(\bar{x};r_{\textrm{max}})$\label{Theorem_Main-convergence-result-weak}]
Suppose, the proximal parameter $\gamma$ is selected to satisfy $0<\gamma<1/L$
and let $z_{\textup{init}}$ be the any arbitrarily chosen initial
point that does not satisfy the conditions of Lemma  \ref{Convergence-result-for-fixed-mu}.
Then, in this setup, \textsf{NExOS} will construct a finite sequence
of strictly decreasing penalty parameters $\mathfrak{M}=\{\mu_{1}\coloneqq\mu_{\textup{init}},\mu_{2}=\rho\mu_{1},\mu_{3}=\rho\mu_{2},\ldots\},$
and $\rho\in(0,1)$, such that we have the following recursive convergence
property. For any $\mu\in\mathcal{M}$, if an $\epsilon$-approximate
fixed point of $\opT_{\mu}$ over $B(\bar{x};r_{\textup{max}})$ is
used to initialize the inner algorithm for penalty parameter $\rho\mu$,
then the corresponding inner algorithm iterates $z_{\rho\mu}^{n}$
subsequentially converges to $z_{\rho\mu}$ that is a fixed point
of $\opT_{\rho\mu}$, and the iterates $x_{\rho\mu}^{n},y_{\rho\mu}^{n}$subsequentially
converge to a first-order stationary point to $(\mathcal{P}_{\rho\mu})$
denoted by $x_{\rho\mu}=\prox_{\gamma\regf}(z_{\rho\mu})$ with the
rate $\min_{n\leq k}\|\nabla\left(f+\regind\right)(x_{\rho\mu}^{n})\|\leq\frac{1-\gamma L}{2L}o(1/\sqrt{k}).$
\end{thm}
\begin{proof}
See Appendix \ref{subsec:Proof-to-Theorem-NExOS-weak}.
\end{proof}

\section{Numerical experiments \label{sec:Numerical-experiments}}

In this section, we apply \textsf{NExOS} to the following nonconvex
optimization problems of substantial current interest for both
synthetic and real-world datasets: sparse regression problem in
\S\ref{subsec:Regressor-selection}, affine rank minimization problem
in \S\ref{subsec:Affine-rank-minimization}, and low-rank factor
analysis problem in \S\ref{subsec:Factor-analysis-problem}. We illustrate
that \textsf{NExOS }produces solutions that are either competitive
or better in comparison with the other approaches on different performance
measures. We have implemented \textsf{NExOS} in\texttt{ NExOS.jl }solver,
which is an open-source software package written in the \texttt{Julia
}programming language. \texttt{NExOS.jl} can address any optimization
problem of the form of problem (\ref{eq:original_problem-1}). The code and documentation
are available online at: \url{https://github.com/Shuvomoy/NExOS.jl}.

{
In our numerical experiments, we present a comprehensive evaluation of \textsf{NExOS}, showing both statistical and optimization-theoretic evaluations. This dual approach is deliberate---while our primary contribution is in developing optimization methodology, the optimization problems considered in this section—such as sparse regression, affine rank minimization, matrix completion, and factor analysis—are deeply rooted in the fields of statistics and machine learning \cite{Hastie2015,friedman2001elements,hastie2017extended,jain2017non,Fazel2008,bertsimas2017certifiably}. Therefore, our numerical experiments are constructed not only to demonstrate \textsf{NExOS} efficiently computing local minima for nonconvex problems but also to highlight its ability to provide statistically robust solutions, which are also important in the application context. This dual capacity is of paramount importance for practical applications in statistics and machine learning, underlining the algorithm's versatility and effectiveness. By addressing these aspects, we aim to illustrate the broad applicability of \textsf{NExOS} across optimization-theoretic and applied statistical or learning domains. Here, we stress that while our optimization-theoretic evaluations are grounded in both theory
and empirical experiments, the statistical evaluations of \textsf{NExOS} are based on empirical observations made in the context of the
experiments conducted in this section.
}

To compute the proximal operator of a function $f$ with closed
form or easy-to-compute solution, \texttt{NExOS.jl} uses the open-source
package \url{ProximalOperators.jl} \cite{lorenzo_stella_2020_4020559}.
When $f$ is a constrained convex function (\emph{i.e.}, a convex
function over some convex constraint set) with no closed form proximal
map, \texttt{NExOS.jl} computes the proximal operator by using the
open-source \texttt{Julia} package\texttt{ JuMP }\cite{dunning2017jump}
and any of the commercial or open-source solver supported by it. The
set $\mathcal{X}$ can be any prox-regular nonconvex set fitting our
setup. Our implementation is readily extensible using \texttt{Julia}
abstract types so that the user can add support for additional convex
functions and prox-regular sets. The numerical study is executed on a MacBook Pro laptop with Apple M1 Max chip with 32 GB memory.  The datasets considered in this section, unless specified otherwise, are available online at:
\url{http://tinyurl.com/NExOSDatasets}.

In applying \textsf{NExOS}, we use the following values that we found to be the best performing based on our empirical observations. We take the
starting value of $\mu$ to be $2$, and reduce this value with a
multiplicative factor of $0.5$ during each iteration of the outer
loop until the termination criterion is met. The value of the proximal
parameter $\gamma$ is chosen to be $10^{-3}$. We initialize our
iterates at $\mathbf{0}.$ Maximum number of inner iterations for
a fixed value of $\mu$ is taken to be 1000. The tolerance for the
fixed point gap for each penalized problem is taken to be $10^{-4}$ and the tolerance for the
termination criterion is taken to be $10^{-6}$.

Value of $\beta$
is taken to be $10^{-8}$ for the following reasons. In \S\ref{sec:Convergence-analysis}, we showed that the presence of $\beta > 0$, ensures that each penalized subproblem is locally strongly convex and smooth, having a unique local minimum. This, in turn, helps to establish linear convergence of the inner algorithm for each subproblem. We empirically demonstrate in this section that the impact of the condition $\beta>0$, despite being critical in the theoretical analysis of our algorithm, seems to only be marginal as it can be made to be as small as $10^{-8}$. We use this extremely small value of $\beta$ to stress-test \textsf{NExOS} empirically and show that even for such a small value of $\beta$, our algorithm still works well in practice. Here, we stress that the default value of $\beta = 10^{-8}$ used in our numerical experiments should be viewed as a mere heuristic that seems to empirically work for the numerical experiments that we considered in our paper. We leave a more methodical investigation of the smallest admissible values of $\beta$ or the effect of completely omitting it to future work.

\subsection{Sparse regression\label{subsec:Regressor-selection}}

In (\ref{eq:reg-sel}), we set $\mathcal{X}:=\{x\mid\|x\|_{\infty}\leq\Gamma,\mathbf{card}(x)\leq k\},$
and $f(x):=\|Ax-b\|_{2}^{2}$. A projection onto $\mathcal{X}$ can
be computed using the formula in \cite[\S 2.2]{jain2017non}, whereas
the proximal operator for $f$ can be computed using the formula in
\cite[\S 6.1.1]{boydProx}. Now we are in a position to apply \textsf{NExOS}
to this problem.

\subsubsection{Synthetic dataset: comparison with \textsf{elastic net} and \texttt{Gurobi}\label{subsec:Synthetic-dataset:-comparison}}

We compare the solution found by \textsf{NExOS} with the solutions
found by \textsf{elastic net} (\textsf{glmnet} used for the implementation) and spatial branch-and-bound algorithm (\textsf{Gurobi} used for the implementation). \textsf{Elastic net} is a well-known method for computing an approximate solution to the regressor
selection problem (\ref{eq:reg-sel}), which empirically works extremely well in recovering support of the original signal. On the other hand, \textsf{Gurobi}'s spatial branch-and-bound algorithm is guaranteed to compute a globally optimal solution to (\ref{eq:reg-sel}).
\textsf{NExOS} is guaranteed to provide a locally optimal solution
under regularity conditions; so to investigate 
how close \textsf{NExOS} can get to the globally minimum value we consider a parallel implementation of 
\textsf{NExOS} running on multiple (20) worker processes, where each process runs \textsf{NExOS} with different random initialization, and we take the solution associated
with the least objective value. 

\paragraph{Elastic net}

\textsf{Elastic net} is a well-known method for solving the regressor
selection problem, that computes an approximate solution as follows.
First, \textsf{elastic net} solves: 
\begin{equation}
\begin{array}{ll}
\textup{minimize} & \|Ax-b\|_{2}^{2}+\lambda\|x\|_{1}+(\beta/2)\|x\|_{2}^{2},\end{array}\label{eq:lasso}
\end{equation}
where $\lambda$ is a parameter that is related to the sparsity of
the decision variable $x\in\mathbf{R}^{d}$. To solve (\ref{eq:lasso}),
we have used the open-source \texttt{R} pacakge \textsf{glmnet} \cite{friedman2009glmnet}.

To compute $\lambda$ corresponding
to $\mathbf{card}(x)\leq k$ we follow the
method proposed in \cite[\S 3.4]{friedman2001elements} and \cite[Example 6.4]{boyd2004convex}.
We solve the problem (\ref{eq:lasso}) for different values of $\lambda$,
and find the smallest value of $\lambda$ for which $\mathbf{card}(x)\leq k$,
and we consider the sparsity pattern of the corresponding solution
$\tilde{x}.$ Let the index set of zero elements of $\tilde{x}$ be
$\mathcal{Z}$, where $\mathcal{Z}$ has $d-k$ elements. Then the
\textsf{elastic net} solves: 
\begin{align}
\begin{array}{ll}
\textup{minimize} & \left\Vert Ax-b\right\Vert _{2}^{2}+(\beta/2)\|x\|_{2}^{2}\\
\textup{subject to} & \left(\forall j\in\mathcal{Z}\right)\quad x_{j}=0,
\end{array}\label{eq:lasso_heuristic_stage_2}
\end{align}
where $x\in\mathbf{R}^{d}$ is the decision variable. Solving this problem corresponds to solving a positive semidefinite linear system, which we solve using the built-in \texttt{LinearAlgebra} package in \texttt{Julia}. %

\paragraph{Spatial branch-and-bound algorithm}

The problem (\ref{eq:reg-sel}) can also be modeled equivalently as
the following mixed integer quadratic optimization problem \cite{bertsimas2016best}:
\[
\begin{array}{ll}
\textup{minimize} & \|Ax-b\|_{2}^{2}+(\beta/2)\|x\|_{2}^{2}\\
\textup{subject to} & |x_{i}|\leq\Gamma y_{i},\quad i=1,\ldots,d\\
 & \sum_{i=1}^{d}y_{i}\leq k, \quad x\in\mathbf{R}^{d},\quad y\in\{0,1\}^{d},
\end{array}
\]
which can be solved to a certifiable global optimality using \textsf{Gurobi}'s spatial branch-and-bound algorithm.

\paragraph{Data generation process and setup}

The data generation procedure is similar to \cite{Diamond2018} and \cite{hastie2017extended}. We consider two signal-to-noise ratio (SNR) regimes: SNR 1 and SNR 6, where for each SNR, we vary $m$ from 25 to 50, and for each
value of $m,$ we generate 50 random problem instances. We limit the size of the problems because
the solution time by \textsf{Gurobi}'s spatial branch-and-bound algorithm becomes too large for comparison
if we go beyond the aforementioned size. For a certain value
of $m$, the matrix $A\in\mathbf{R}^{m\times2m}$ is generated from
an independent and identically distributed normal distribution with
$\mathcal{N}(0,1)$ entries. We choose $b=A\widetilde{x}+v$, where
$\widetilde{x}$ is drawn uniformly from the set of vectors satisfying
$\mathbf{card}(\widetilde{x})\leq\lfloor m/5\rceil$ and $\|\widetilde{x}\|_{\infty}\leq \Gamma$ with $\Gamma = 1$.
The vector $v$ corresponds to noise, and is drawn from the distribution
$\mathcal{N}(0,\sigma^{2}I),$ where $\sigma^{2}=\|A\widetilde{x}\|_{2}^{2}/(\textrm{SNR}^2/m)$,
which keeps the signal-to-noise ratio to approximately equal to $\textrm{SNR}$. We consider a parallel implementation of 
\textsf{NExOS} where we have 100 runs of \textsf{NExOS} distrubuted over 20 independent worker processes on 10 cores. Each run is initialized with a random initial points chosen from the uniform distribution over the
interval $[-\Gamma,\Gamma]$. Gurobi's spatial branch-and-bound algorithm also uses 10 cores.

\begin{figure}

\centering
 \subfloat[Support recovery (SNR 6)]{
  \label{fig:Sup-rec-snr-6}
 \begin{tikzpicture}[spy using outlines=
	{rectangle, magnification=2.5, connect spies}]
\begin{axis}[legend style= {at={(0.29,0.95)},anchor=north, font=\tiny}, scale only axis, height=2.5cm, width=0.2*\textwidth, xmajorgrids, ymajorgrids, font=\tiny, ticklabel style = {font=\tiny}, error bars/y dir=both, error bars/y explicit, error bars/error bar style={line width=.1pt}, xlabel={$m$ (size of the matrix)}, ylabel={Support recovery (\%)}, ymin={60}, ymax={105}, xmin={14}, xmax={52}]
    \addplot+[no marks, style={{ultra thick}}, color={red}, each nth point=3]
        coordinates {
            (15,94.4666666666667) +- (0,0.7228690632055588)
            (16,94.8125) +- (0,0.4429327412433022)
            (17,95.47058823529407) +- (0,0.4219290827103591)
            (18,92.6666666666666) +- (0,0.6083735840233907)
            (19,93.57894736842108) +- (0,0.5533344208413263)
            (20,94.1) +- (0,0.5089003751587678)
            (21,94.33333333333336) +- (0,0.4804956023513007)
            (22,95.40909090909092) +- (0,0.39225294313088205)
            (23,94.2173913043478) +- (0,0.43304889474666775)
            (24,93.79166666666666) +- (0,0.45910571900651526)
            (25,94.32) +- (0,0.3890753035554102)
            (26,94.84615384615385) +- (0,0.35370278987760645)
            (27,95.22222222222217) +- (0,0.38186985345536767)
            (28,93.42857142857143) +- (0,0.3832987749182233)
            (29,93.58620689655176) +- (0,0.4852654327497435)
            (30,94.76666666666668) +- (0,0.4124205413747548)
            (31,94.00000000000001) +- (0,0.3092711840656686)
            (32,95.34375) +- (0,0.33849178727263146)
            (33,93.30303030303034) +- (0,0.41316485411905834)
            (34,94.64705882352935) +- (0,0.38682854744483747)
            (35,93.94285714285718) +- (0,0.32540622739583236)
            (36,94.83333333333329) +- (0,0.3572310296802193)
            (37,94.72972972972981) +- (0,0.3522835442257655)
            (38,94.05263157894738) +- (0,0.3388467470449248)
            (39,94.33333333333336) +- (0,0.36462989061619905)
            (40,94.225) +- (0,0.3348202380931218)
            (41,94.48780487804875) +- (0,0.30223342776587353)
            (42,95.02380952380958) +- (0,0.3344594334025809)
            (43,94.25581395348836) +- (0,0.3312443972977687)
            (44,93.70454545454544) +- (0,0.3952647061021667)
            (45,94.95555555555558) +- (0,0.2442794743264844)
            (46,94.84782608695643) +- (0,0.28935367755011865)
            (47,94.42553191489361) +- (0,0.32247506893312733)
            (48,94.1875) +- (0,0.2491659841184651)
            (49,94.734693877551) +- (0,0.28169964616166)
            (50,94.66) +- (0,0.30587445737339586)
        }
        ;
    % \addlegendentry {glmnet}
    \addplot+[no marks, style={{ultra thick}}, color={green}, each nth point=3]
        coordinates {
            (15,98.53333333333335) +- (0,0.39452028908967585)
            (16,98.5) +- (0,0.3813242232868325)
            (17,98.1176470588235) +- (0,0.4265068080823364)
            (18,96.77777777777776) +- (0,0.5742873814325566)
            (19,98.21052631578951) +- (0,0.44124513537206606)
            (20,98.1) +- (0,0.40127348300537524)
            (21,98.00000000000004) +- (0,0.4321736101841792)
            (22,98.13636363636363) +- (0,0.3655446219490278)
            (23,97.13043478260869) +- (0,0.44112499141775646)
            (24,96.83333333333331) +- (0,0.4219058368254605)
            (25,97.28) +- (0,0.40309010488110597)
            (26,97.46153846153844) +- (0,0.37443351730551205)
            (27,97.85185185185185) +- (0,0.36794367514737997)
            (28,97.28571428571429) +- (0,0.42759850693381535)
            (29,97.51724137931036) +- (0,0.3820831414712379)
            (30,97.36666666666665) +- (0,0.38712264927052675)
            (31,97.29032258064521) +- (0,0.38401228858985537)
            (32,97.6875) +- (0,0.33180186028682224)
            (33,97.93939393939395) +- (0,0.34020576111495254)
            (34,97.23529411764703) +- (0,0.4065872404505193)
            (35,97.77142857142859) +- (0,0.2861803778442138)
            (36,97.13888888888889) +- (0,0.3084280563781435)
            (37,97.86486486486493) +- (0,0.33679375342923346)
            (38,98.0789473684211) +- (0,0.27933621034977313)
            (39,97.7948717948718) +- (0,0.2931318538277297)
            (40,97.55) +- (0,0.24212073223826128)
            (41,97.46341463414636) +- (0,0.33391354845187265)
            (42,97.7619047619048) +- (0,0.2670447132842247)
            (43,97.30232558139534) +- (0,0.3475534041361335)
            (44,97.65909090909088) +- (0,0.2911416336336653)
            (45,97.68888888888888) +- (0,0.3042323441450398)
            (46,97.78260869565216) +- (0,0.24247479870736793)
            (47,97.48936170212767) +- (0,0.30811734705104404)
            (48,97.16666666666664) +- (0,0.3247186118329966)
            (49,96.89795918367346) +- (0,0.31530042762905586)
            (50,97.54) +- (0,0.35365150608346363)
        }
        ;
   % \addlegendentry {Gurobi}
   	\addplot[color=black, dashed, ultra thick, each nth point=1] coordinates {(15,100) (50,100)};
   	    \addplot+[no marks, style={{ultra thick}}, color={blue}, each nth point=3]
        coordinates {
            (15,98.53333333333335) +- (0,0.39452028908967585)
            (16,98.5) +- (0,0.3813242232868325)
            (17,98.35294117647055) +- (0,0.37730998858224574)
            (18,96.77777777777776) +- (0,0.5742873814325566)
            (19,98.00000000000004) +- (0,0.47290782536891357)
            (20,98.1) +- (0,0.40127348300537524)
            (21,98.09523809523813) +- (0,0.43024185852631)
            (22,98.13636363636363) +- (0,0.3655446219490278)
            (23,96.95652173913042) +- (0,0.45218073846462836)
            (24,96.75) +- (0,0.41734638435614524)
            (25,97.12) +- (0,0.42822843417589035)
            (26,97.46153846153844) +- (0,0.37443351730551205)
            (27,97.85185185185182) +- (0,0.3679436751473803)
            (28,97.21428571428572) +- (0,0.4363960159253691)
            (29,97.37931034482762) +- (0,0.37593663110027437)
            (30,97.36666666666665) +- (0,0.38712264927052675)
            (31,97.22580645161295) +- (0,0.40184754574362463)
            (32,97.75) +- (0,0.3098098495160263)
            (33,98.00000000000001) +- (0,0.307450751594194)
            (34,97.35294117647055) +- (0,0.38737581753331457)
            (35,97.77142857142859) +- (0,0.2975955287111539)
            (36,97.13888888888889) +- (0,0.332978836441241)
            (37,98.02702702702705) +- (0,0.3303584648746291)
            (38,98.0789473684211) +- (0,0.28927832486031013)
            (39,97.6923076923077) +- (0,0.3129671701581512)
            (40,97.5) +- (0,0.25753937681885636)
            (41,97.21951219512196) +- (0,0.35540069686411135)
            (42,97.71428571428577) +- (0,0.2717684946318118)
            (43,97.3953488372093) +- (0,0.3230630107945042)
            (44,97.5681818181818) +- (0,0.2830640067386636)
            (45,97.55555555555559) +- (0,0.2926839510251711)
            (46,97.82608695652172) +- (0,0.2385448990021305)
            (47,97.44680851063829) +- (0,0.2978103030739426)
            (48,97.04166666666666) +- (0,0.3371906913873465)
            (49,97.06122448979589) +- (0,0.32548771398140786)
            (50,97.5) +- (0,0.32103150077491754)
        }
        ;
   %  \addlegendentry {NExOS}
  \coordinate (spypoint) at (axis cs:25,97.5);
  \coordinate (magnifyglass) at (axis cs:35,80);
\end{axis}
\spy [black, width=1.6cm,height=0.95cm] on (spypoint)
   in node[fill=white] at (magnifyglass);
\end{tikzpicture}
 } \hfil
  \subfloat[Objective value (SNR 6)]{
  \label{fig:Ratio-of-objective-snr-6}
 \begin{tikzpicture}
\begin{axis}[legend style= {at={(0.29,0.95)},anchor=north, font=\tiny}, scale only axis, height=2.5cm, width=0.2*\textwidth, xmajorgrids, ymajorgrids, font=\tiny, ticklabel style = {font=\tiny}, error bars/y dir=both, error bars/y explicit, error bars/error bar style={line width=.1pt}, xlabel={$m$ (size of the matrix)}, ylabel={Normalized objective value}, ytick={0,1,3,5, 7, 9}]
    \addplot+[no marks, style={{ultra thick}}, color={blue}, each nth point=3]
        coordinates {
            (15,1.0016477001279156) +- (0,0.0008004696219931773)
            (16,1.0007816724964709) +- (0,0.00021442087741068317)
            (17,1.0012113095270023) +- (0,0.0006633659323270895)
            (18,1.003889726408251) +- (0,0.0027943198737886647)
            (19,1.002408513627724) +- (0,0.001353465597081712)
            (20,1.004680554157388) +- (0,0.0029798130990963596)
            (21,1.0009446286065742) +- (0,0.0007050682205491522)
            (22,1.002848446108547) +- (0,0.002236570012621716)
            (23,1.0050032176692039) +- (0,0.0029423727600295553)
            (24,1.0131358345282286) +- (0,0.0073650950933779194)
            (25,1.003341458439529) +- (0,0.001523278935205372)
            (26,1.0027320647606004) +- (0,0.0013220456451042428)
            (27,1.012998408867843) +- (0,0.0065768217073879885)
            (28,1.0525295070577914) +- (0,0.04448299314556211)
            (29,1.0103087108885378) +- (0,0.004508632553012121)
            (30,1.0075076520723407) +- (0,0.0033198862088302684)
            (31,1.0097619055243896) +- (0,0.0038460834230687156)
            (32,1.009008473085712) +- (0,0.003961090701528821)
            (33,1.0180133616729619) +- (0,0.007795110510259388)
            (34,1.016625679525326) +- (0,0.0070681190294682944)
            (35,1.014720139387194) +- (0,0.004463193837959137)
            (36,1.010832078372932) +- (0,0.004080535502334341)
            (37,1.0128280678324537) +- (0,0.005152816435016701)
            (38,1.0085178494581386) +- (0,0.003465368325252512)
            (39,1.0172165240642903) +- (0,0.004953679653959255)
            (40,1.0055159804475322) +- (0,0.0020688098482151075)
            (41,1.0184003292202406) +- (0,0.008353906323396663)
            (42,1.0075581608557607) +- (0,0.0024959121162003934)
            (43,1.0231632650656035) +- (0,0.007648873221779811)
            (44,1.0147477851301216) +- (0,0.004648183649356183)
            (45,1.018009422277137) +- (0,0.006546889902704262)
            (46,1.0125571271079865) +- (0,0.004399880976927543)
            (47,1.0211777083456544) +- (0,0.006984008495119486)
            (48,1.0285911598764104) +- (0,0.006997769850818383)
            (49,1.032393779162266) +- (0,0.009297908584120955)
            (50,1.022571906780459) +- (0,0.005976672003552044)
        }
        ;
   % \addlegendentry {$p^{\star}_{\textrm{Gurobi}}/p^{\star}_{\textrm{NExOS}}$}
    \addplot+[no marks, style={{ultra thick}}, color={red}, each nth point=3]
        coordinates {
            (15,6.559077912475141) +- (0,1.3547933380426207)
            (16,3.483628192900611) +- (0,0.6381701272295829)
            (17,3.857945672738636) +- (0,0.7537746584427883)
            (18,4.109863661487146) +- (0,0.7416860644162007)
            (19,3.780001577837952) +- (0,0.6507906948283803)
            (20,3.665983981608242) +- (0,0.632425431043273)
            (21,3.4232443823042105) +- (0,0.6243661872250847)
            (22,2.9513796657049585) +- (0,0.4159274653465091)
            (23,2.273336562109921) +- (0,0.22692576853027097)
            (24,2.6108264730076844) +- (0,0.28494272960461475)
            (25,2.421209277382266) +- (0,0.2633554757021143)
            (26,2.54981563081199) +- (0,0.34976538830581827)
            (27,2.653911291064773) +- (0,0.37398932137232177)
            (28,4.106307448880392) +- (0,0.5115000817455884)
            (29,3.4886901493816924) +- (0,0.4991198791140536)
            (30,2.7116693526966444) +- (0,0.3211104719075026)
            (31,2.654470042929211) +- (0,0.38387498632483064)
            (32,2.2726971635051894) +- (0,0.2979039350636274)
            (33,3.6375896253923345) +- (0,0.41853479830639706)
            (34,2.6010796512774976) +- (0,0.34689003188590595)
            (35,3.1511229191940293) +- (0,0.37971611252140297)
            (36,2.8608388022342526) +- (0,0.5968320375678827)
            (37,3.094508145560901) +- (0,0.45552721512596567)
            (38,3.4499628997321157) +- (0,0.37496320454391174)
            (39,2.694737704061174) +- (0,0.32024099296810205)
            (40,2.6250394016159238) +- (0,0.2935680286927622)
            (41,2.58828406957737) +- (0,0.2867318963378842)
            (42,2.3924700172079927) +- (0,0.23919744710597027)
            (43,3.256398790864546) +- (0,0.3909239344709963)
            (44,3.498284922119676) +- (0,0.48829320571019325)
            (45,2.181072047937966) +- (0,0.20634186040105404)
            (46,2.6620284113422525) +- (0,0.30606049725787976)
            (47,3.109833494864679) +- (0,0.40768481944885404)
            (48,2.5689255800537776) +- (0,0.2613209720925204)
            (49,2.147744037970927) +- (0,0.17527463439438404)
            (50,2.127228366296508) +- (0,0.2190670193841723)
        }
        ;
  %  \addlegendentry {$p^{\star}_{\textrm{Gurobi}}/p^{\star}_{\textrm{elastic net}}$}
        	\addplot[color=green, dashed, ultra thick, each nth point=1] coordinates {(15,1) (50,1)};
    %	\addlegendentry {${\textrm{Gurobi}}$}
\end{axis}
\end{tikzpicture}} \hfil
 \subfloat[Solution time (s) (SNR 6)]{
  \label{fig:Sol-time-snr-6}
 \begin{tikzpicture}
\begin{axis}[legend style= {at={(0.29,0.95)},anchor=north, font=\tiny}, scale only axis, height=2.5cm, width=0.2*\textwidth, xmajorgrids, ymajorgrids, font=\tiny, ticklabel style = {font=\tiny}, error bars/y dir=both, error bars/y explicit, error bars/error bar style={line width=.1pt}, xlabel={$m$ (size of the matrix)}, ylabel={Solution time (s) in log scale}, ymode={log}, ytick = {0.001, 0.01, 0.1, 1, 10}, each nth point=3]
    \addplot+[no marks, style={{ultra thick}}, color={blue}]
        coordinates {
            (15,0.22452221937999992) +- (0,0.07965675780124416)
            (16,0.15188546915999998) +- (0,0.0023774054121859316)
            (17,0.16006060516) +- (0,0.0028930999990188027)
            (18,0.19923833049999998) +- (0,0.0030815483565691665)
            (19,0.20336771056000003) +- (0,0.0024998252920248416)
            (20,0.21256400462) +- (0,0.002486760524957443)
            (21,0.21941647514000007) +- (0,0.002684138958056928)
            (22,0.23236930712000003) +- (0,0.0028920901427453216)
            (23,0.24884966186000004) +- (0,0.0028367418326814687)
            (24,0.25336180222) +- (0,0.003085671660087021)
            (25,0.26712144686) +- (0,0.0029627740189445105)
            (26,0.27040944466) +- (0,0.002647503874631151)
            (27,0.2810253711800001) +- (0,0.002859167993643402)
            (28,0.32103989012) +- (0,0.003188747283375379)
            (29,0.32856541478000006) +- (0,0.0036367836981929208)
            (30,0.3320751148400001) +- (0,0.002854540233736135)
            (31,0.3543948577399999) +- (0,0.00416798804270667)
            (32,0.3482534372) +- (0,0.0030476551863871247)
            (33,0.38204984864) +- (0,0.003258111468988056)
            (34,0.3900518984) +- (0,0.0039760508591930935)
            (35,0.38620869890000004) +- (0,0.0034701938611674545)
            (36,0.39226796045999995) +- (0,0.002750962774152172)
            (37,0.40926500660000004) +- (0,0.002868847718347781)
            (38,0.4525313935399998) +- (0,0.003683545480318558)
            (39,0.47299944872000005) +- (0,0.0028930029800953846)
            (40,0.45869729182) +- (0,0.002362883592376687)
            (41,0.45767342803999994) +- (0,0.0026745674882154733)
            (42,0.4552500711599999) +- (0,0.003392266012150523)
            (43,0.45428519255999994) +- (0,0.004534174403896932)
            (44,0.45204327818) +- (0,0.0046926382435217105)
            (45,0.4497776900799999) +- (0,0.0032327676014954417)
            (46,0.46215229298) +- (0,0.004225070844569965)
            (47,0.45953478482) +- (0,0.004161316717272426)
            (48,0.62286300664) +- (0,0.0026329731304095914)
            (49,0.6306366668800001) +- (0,0.0029625129775017658)
            (50,0.6236162439399999) +- (0,0.003140360840016195)
        }
        ;
    % \addlegendentry {NExOS}
    \addplot+[no marks, style={{ultra thick}}, color={green}]
        coordinates {
            (15,0.12727965878000008) +- (0,0.07134657167861028)
            (16,0.0553610515) +- (0,0.0019507161499546471)
            (17,0.0686900511) +- (0,0.002835140366096567)
            (18,0.09679659909999999) +- (0,0.004268284958568384)
            (19,0.09714963092) +- (0,0.005297553146647083)
            (20,0.09929857581999998) +- (0,0.003613320594513726)
            (21,0.10843924891999997) +- (0,0.005403793384642203)
            (22,0.12470809388000005) +- (0,0.006916719306500576)
            (23,0.16229491173999996) +- (0,0.007762449354594711)
            (24,0.19002673321999997) +- (0,0.01401277865801378)
            (25,0.1927653972199999) +- (0,0.011089536637942113)
            (26,0.20842620337999995) +- (0,0.011593914331400938)
            (27,0.2356692732) +- (0,0.017710789544861466)
            (28,0.30697912758000007) +- (0,0.026746407420409686)
            (29,0.61390551954) +- (0,0.10707260935182862)
            (30,0.4404375830399999) +- (0,0.048320573851976446)
            (31,0.48820452647999985) +- (0,0.05718526122341773)
            (32,0.46636773368000006) +- (0,0.04561932566797899)
            (33,0.6339847168000001) +- (0,0.07247503487407765)
            (34,0.8110255599199998) +- (0,0.10371106811010172)
            (35,0.7093873266000001) +- (0,0.08334384203526746)
            (36,1.15405128904) +- (0,0.18947872193654544)
            (37,0.7984958435399999) +- (0,0.10850782736755561)
            (38,1.2375997899600002) +- (0,0.17381475675999636)
            (39,1.4164025650799998) +- (0,0.13935063347190468)
            (40,2.0374312049400007) +- (0,0.28189603310475597)
            (41,2.4675811944599992) +- (0,0.5410732289327781)
            (42,2.09128602658) +- (0,0.36738319944632203)
            (43,3.02446787586) +- (0,0.7826507055508676)
            (44,3.492495866760002) +- (0,0.9954900338080426)
            (45,2.9178023381199996) +- (0,0.4163057394854315)
            (46,2.9990188019199984) +- (0,0.41498193530134925)
            (47,5.074390348580001) +- (0,1.3679997258281225)
            (48,15.889517006060005) +- (0,8.109452175851628)
            (49,7.5766083435399985) +- (0,1.2147823242425384)
            (50,8.14560523338) +- (0,1.2931073097941619)
        }
        ;
   %  \addlegendentry {Gurobi}
    \addplot+[no marks, style={{ultra thick}}, color={red}]
        coordinates {
            (15,0.015857958999999998) +- (0,0.005097379935651559)
            (16,0.01174429616) +- (0,0.0005910746212205533)
            (17,0.012289269740000002) +- (0,0.0005144033993033188)
            (18,0.013775456019999998) +- (0,0.00046080974668560984)
            (19,0.01520640268) +- (0,0.0006585769327364016)
            (20,0.016513138860000002) +- (0,0.0005610567243312125)
            (21,0.01680146766) +- (0,0.00046430716800513114)
            (22,0.01719040132) +- (0,0.00011708041807117174)
            (23,0.020271368399999993) +- (0,0.0007114699425488275)
            (24,0.02144872772) +- (0,0.0007209025938961534)
            (25,0.022699434699999993) +- (0,0.0007614006890994286)
            (26,0.024157984719999998) +- (0,0.0007628764656014068)
            (27,0.0253387641) +- (0,0.0006358437594875307)
            (28,0.028113530960000008) +- (0,0.000934437546780695)
            (29,0.03057749326) +- (0,0.0011088012615826264)
            (30,0.03198375648) +- (0,0.0010365279665324717)
            (31,0.031794931039999995) +- (0,0.0008575048291010658)
            (32,0.034158377420000004) +- (0,0.000976952103319223)
            (33,0.033649822380000007) +- (0,0.00018102818302206228)
            (34,0.03802962172000001) +- (0,0.0011719269832835945)
            (35,0.039328412959999996) +- (0,0.0010303199384712395)
            (36,0.04266513221999999) +- (0,0.001297728905105227)
            (37,0.04820985304) +- (0,0.001188688048605529)
            (38,0.04908601595999999) +- (0,0.0010780348780942421)
            (39,0.049067201740000004) +- (0,0.0006375974561352209)
            (40,0.053940215439999994) +- (0,0.0013760611819820076)
            (41,0.05743363984) +- (0,0.0013513696049572833)
            (42,0.05598901424000001) +- (0,0.0008508052559633696)
            (43,0.06065563476) +- (0,0.0012243033749360934)
            (44,0.060181013959999995) +- (0,0.0013324195336481374)
            (45,0.06059737328) +- (0,0.0012899719023723507)
            (46,0.06187017621999999) +- (0,0.0011701594776310708)
            (47,0.06666781972000002) +- (0,0.0015834383421530415)
            (48,0.07698751983999999) +- (0,0.0020306550711500443)
            (49,0.07530371502000002) +- (0,0.0013077094038772131)
            (50,0.0813429103) +- (0,0.0017536283346892487)
        }
        ;
    % \addlegendentry {elastic net}
\end{axis}
\end{tikzpicture}
 }

 \subfloat[Support recovery (SNR 1)]{
  \label{fig:Sup-rec-snr-1}
 \begin{tikzpicture}[spy using outlines=
	{rectangle, magnification=2.5, connect spies}]
\begin{axis}[legend style= {at={(0.29,0.95)},anchor=north, font=\tiny}, scale only axis, height=2.5cm, width=0.2*\textwidth, xmajorgrids, ymajorgrids, font=\tiny, ticklabel style = {font=\tiny}, error bars/y dir=both, error bars/y explicit, error bars/error bar style={line width=.1pt}, xlabel={$m$ (size of the matrix)}, ylabel={Support recovery (\%)}, ymin={60}, ymax={105}, xmin={14}, xmax={52}]
        \addplot+[no marks, style={ultra thick}, color={green}, each nth point=3]
        coordinates {
            (15,87.46666666666664) +- (0,0.6771957651288066)
            (16,89.25) +- (0,0.7791937224739796)
            (17,88.58823529411757) +- (0,0.6154565856672455)
            (18,86.77777777777777) +- (0,0.612500996240468)
            (19,86.00000000000001) +- (0,0.594127707222208)
            (20,87.4) +- (0,0.6583839212458805)
            (21,87.23809523809523) +- (0,0.6299775185191963)
            (22,88.59090909090914) +- (0,0.5528640222398259)
            (23,85.65217391304351) +- (0,0.5989845193163329)
            (24,85.33333333333334) +- (0,0.535581994247714)
            (25,87.04) +- (0,0.47891032776587367)
            (26,87.38461538461544) +- (0,0.45362126242967676)
            (27,87.55555555555549) +- (0,0.47039387874905514)
            (28,85.7857142857143) +- (0,0.5634464364195322)
            (29,86.00000000000001) +- (0,0.534295442875123)
            (30,86.69999999999999) +- (0,0.4324952649394112)
            (31,86.9677419354839) +- (0,0.4511424239823195)
            (32,87.6875) +- (0,0.403369892788447)
            (33,85.6363636363637) +- (0,0.4135048982254391)
            (34,86.41176470588232) +- (0,0.5172676771936937)
            (35,86.74285714285713) +- (0,0.5147390292386759)
            (36,88.24999999999999) +- (0,0.506959004854111)
            (37,88.18918918918924) +- (0,0.41370435664060623)
            (38,86.39473684210526) +- (0,0.45364273936274896)
            (39,85.5384615384616) +- (0,0.5167821857450079)
            (40,86.3) +- (0,0.4909382936411309)
            (41,86.87804878048784) +- (0,0.4345773331754392)
            (42,87.4761904761905) +- (0,0.3714186039379243)
            (43,86.51162790697676) +- (0,0.4791430266397326)
            (44,85.79545454545456) +- (0,0.37618929306027504)
            (45,85.99999999999994) +- (0,0.4175538551735205)
            (46,87.39130434782608) +- (0,0.40133068271180417)
            (47,87.5744680851064) +- (0,0.43139683888656943)
            (48,86.2916666666667) +- (0,0.44180141700032155)
            (49,87.2244897959184) +- (0,0.41970357814409576)
            (50,86.82) +- (0,0.3542050599438068)
        }
        ;
    % \addlegendentry {Gurobi}
     	\addplot[color=black, dashed, ultra thick, each nth point=1] coordinates {(15,100) (50,100)};
    \addplot+[no marks, style={{ultra thick}}, color={blue}, each nth point=3]
        coordinates {
            (15,87.06666666666665) +- (0,0.6975174637562113)
            (16,89.25) +- (0,0.7791937224739796)
            (17,88.7058823529411) +- (0,0.6022343820897115)
            (18,86.55555555555554) +- (0,0.6166007809382552)
            (19,86.31578947368422) +- (0,0.6200158835515279)
            (20,86.8) +- (0,0.6509020428511153)
            (21,86.95238095238095) +- (0,0.557823129251701)
            (22,88.86363636363642) +- (0,0.5040316459824684)
            (23,85.73913043478265) +- (0,0.5831897935058636)
            (24,84.91666666666669) +- (0,0.5309310266899466)
            (25,87.2) +- (0,0.48487322138506106)
            (26,87.1538461538462) +- (0,0.4478611840069193)
            (27,87.48148148148147) +- (0,0.5487361761889593)
            (28,85.14285714285715) +- (0,0.5714285714285718)
            (29,86.89655172413794) +- (0,0.5023664545411609)
            (30,86.16666666666666) +- (0,0.4158154570612616)
            (31,86.70967741935485) +- (0,0.45749948925311923)
            (32,87.875) +- (0,0.37499999999999994)
            (33,85.51515151515157) +- (0,0.4267089707481077)
            (34,86.11764705882348) +- (0,0.5044817149513825)
            (35,86.8) +- (0,0.44629998148038014)
            (36,88.69444444444441) +- (0,0.48277524012390544)
            (37,87.8648648648649) +- (0,0.3672827487401225)
            (38,85.97368421052633) +- (0,0.4598314536848836)
            (39,85.74358974358982) +- (0,0.5136570884383326)
            (40,86.0) +- (0,0.42857142857142855)
            (41,87.02439024390249) +- (0,0.35148488398379124)
            (42,87.14285714285712) +- (0,0.45634040357138617)
            (43,86.55813953488372) +- (0,0.4320980860568532)
            (44,86.11363636363637) +- (0,0.4019816195621958)
            (45,85.9555555555555) +- (0,0.42799490911371946)
            (46,87.04347826086955) +- (0,0.37632781266338194)
            (47,87.6595744680851) +- (0,0.4032370565781642)
            (48,86.04166666666669) +- (0,0.4091585441924855)
            (49,86.00000000000001) +- (0,0.345010196909318)
            (50,86.22) +- (0,0.3172458005382948)
        }
        ;
   % \addlegendentry {NExOS}
    \addplot+[no marks, style={{ultra thick}}, color={red}, each nth point=3]
        coordinates {
            (15,87.86666666666665) +- (0,0.5695206242943255)
            (16,88.625) +- (0,0.6105471784742413)
            (17,89.7058823529411) +- (0,0.4772191887227122)
            (18,86.88888888888886) +- (0,0.5888011743171877)
            (19,87.42105263157897) +- (0,0.6348376369266674)
            (20,87.75) +- (0,0.48048657310977533)
            (21,88.5238095238095) +- (0,0.534305999070992)
            (22,89.36363636363642) +- (0,0.5006320848674385)
            (23,86.78260869565221) +- (0,0.5077995884712969)
            (24,87.04166666666666) +- (0,0.46218235399533747)
            (25,88.52) +- (0,0.4230501490029622)
            (26,88.57692307692312) +- (0,0.4199254682260549)
            (27,88.8518518518518) +- (0,0.47468899084353183)
            (28,86.5357142857143) +- (0,0.4120036170045396)
            (29,86.75862068965517) +- (0,0.43349753694581294)
            (30,87.73333333333329) +- (0,0.4684126715183976)
            (31,88.83870967741933) +- (0,0.3823496500035309)
            (32,88.625) +- (0,0.476738895507589)
            (33,87.45454545454551) +- (0,0.3848677189765002)
            (34,87.99999999999996) +- (0,0.45128475906210264)
            (35,88.22857142857143) +- (0,0.3547969601975634)
            (36,88.47222222222221) +- (0,0.34193825277012085)
            (37,89.35135135135138) +- (0,0.375610394039083)
            (38,87.44736842105262) +- (0,0.3706583899460853)
            (39,87.20512820512828) +- (0,0.4135951272714504)
            (40,87.5) +- (0,0.3331206804674577)
            (41,88.31707317073173) +- (0,0.3537743755637166)
            (42,88.59523809523807) +- (0,0.3819381075961822)
            (43,87.74418604651163) +- (0,0.3732333026403177)
            (44,87.3181818181818) +- (0,0.3364112687529511)
            (45,88.24444444444441) +- (0,0.3238873233646719)
            (46,88.4782608695652) +- (0,0.36482422746803506)
            (47,88.7872340425532) +- (0,0.3428188773403735)
            (48,87.81250000000001) +- (0,0.3370461801460154)
            (49,88.06122448979593) +- (0,0.34894195928821653)
            (50,88.26) +- (0,0.3153359077141033)
        }
        ;
   % \addlegendentry {glmnet}
  \coordinate (spypoint) at (axis cs:40,87.5);
  \coordinate (magnifyglass) at (axis cs:30,70);
\end{axis}
\spy [black, width=1.6cm,height=0.95cm] on (spypoint)
   in node[fill=white] at (magnifyglass);
\end{tikzpicture}
 } \hfil
  \subfloat[Objective value (SNR 1)]{
  \label{fig:Ratio-of-objective-snr-1}
 \begin{tikzpicture}
\begin{axis}[legend style= {at={(0.29,0.95)},anchor=north, font=\tiny}, scale only axis, height=2.5cm, width=0.2*\textwidth, xmajorgrids, ymajorgrids, font=\tiny, ticklabel style = {font=\tiny}, error bars/y dir=both, error bars/y explicit, error bars/error bar style={line width=.1pt}, xlabel={$m$ (size of the matrix)}, ylabel={Normalized objective value}, ytick = {1, 1.2, 1.4, 1.6, 1.8, 2}]
    \addplot+[no marks, style={{ultra thick}}, color={blue}, each nth point=3]
        coordinates {
            (15,1.0120567203363195) +- (0,0.004614098680346607)
            (16,1.0060404581311755) +- (0,0.003978574300876993)
            (17,1.0117940776042562) +- (0,0.00659698984044237)
            (18,1.0130472983516434) +- (0,0.005959465043999211)
            (19,1.0297452370010527) +- (0,0.008685782519220637)
            (20,1.0425517498564283) +- (0,0.014669950208768103)
            (21,1.0359074659582) +- (0,0.010752864490906145)
            (22,1.0449853040840322) +- (0,0.008888616969252435)
            (23,1.0606014889794524) +- (0,0.011354581733682636)
            (24,1.071901420747057) +- (0,0.012668539658835737)
            (25,1.057184512695896) +- (0,0.010763550231318558)
            (26,1.0634659885029514) +- (0,0.012260269594940343)
            (27,1.076454808610606) +- (0,0.01373035731743001)
            (28,1.1143914577279572) +- (0,0.015800084001145825)
            (29,1.1021263222739206) +- (0,0.014193289657792488)
            (30,1.1099856290356866) +- (0,0.01572947269417152)
            (31,1.0754863678695479) +- (0,0.01295850779831273)
            (32,1.0830789388726492) +- (0,0.012387080827607483)
            (33,1.1470087266155167) +- (0,0.019056513443739764)
            (34,1.1528437208178248) +- (0,0.01619162273874443)
            (35,1.1221880535815807) +- (0,0.013216688379101815)
            (36,1.1361244299877549) +- (0,0.01426067376226187)
            (37,1.1035902346109325) +- (0,0.012209097320074249)
            (38,1.1639859571260547) +- (0,0.015563732329707756)
            (39,1.1684038247019952) +- (0,0.016375202229184527)
            (40,1.169090413531579) +- (0,0.01830396373791563)
            (41,1.1637328295372118) +- (0,0.01425623282003199)
            (42,1.148511933681853) +- (0,0.014516103328920012)
            (43,1.2023404506747917) +- (0,0.016315804462797078)
            (44,1.181234185587145) +- (0,0.013365289454987477)
            (45,1.1908632052692218) +- (0,0.0168873733242229)
            (46,1.1870685791102038) +- (0,0.011426123153794848)
            (47,1.1758331066847028) +- (0,0.014782892785967934)
            (48,1.1973127586872943) +- (0,0.01604420349640028)
            (49,1.221570778848276) +- (0,0.014534198894543637)
            (50,1.2003536830301693) +- (0,0.014287579518244522)
        }
        ;
    % \addlegendentry {${\textrm{NExOS}}$}
    \addplot+[no marks, style={{ultra thick}}, color={red}, each nth point=3]
        coordinates {
            (15,1.452267314986437) +- (0,0.058072471044421245)
            (16,1.3990069105726537) +- (0,0.059509608192368016)
            (17,1.32272662008373) +- (0,0.05067589617319276)
            (18,1.5015949072521417) +- (0,0.05583624184731151)
            (19,1.495562052280124) +- (0,0.051649135751155366)
            (20,1.4670401524919314) +- (0,0.07165675142478316)
            (21,1.432795325506145) +- (0,0.047017896540359996)
            (22,1.4220990437195389) +- (0,0.041125647634826894)
            (23,1.4229391948281) +- (0,0.04627206619625586)
            (24,1.5472846921396493) +- (0,0.05303633338542655)
            (25,1.4356690953692544) +- (0,0.03519964208574052)
            (26,1.4079307509447778) +- (0,0.03876803676845139)
            (27,1.4192526630487357) +- (0,0.043536865874654856)
            (28,1.5847683701040287) +- (0,0.05756013020501772)
            (29,1.576363417979891) +- (0,0.03960825842176314)
            (30,1.5311202168249245) +- (0,0.048447654189798725)
            (31,1.4418999359403273) +- (0,0.03755371085001168)
            (32,1.4685083005648958) +- (0,0.039445685495481865)
            (33,1.5375768101281861) +- (0,0.04183640634606852)
            (34,1.5109204458212901) +- (0,0.049781704514124635)
            (35,1.443478851812607) +- (0,0.028805163121457672)
            (36,1.509727568224008) +- (0,0.034207572221639435)
            (37,1.4578707193092304) +- (0,0.03370135742176717)
            (38,1.5834586616615383) +- (0,0.043325266153589774)
            (39,1.5889727990938105) +- (0,0.04512838467951343)
            (40,1.5784178650910712) +- (0,0.03886166348021826)
            (41,1.5253947132254424) +- (0,0.03821321540027149)
            (42,1.4554604573469947) +- (0,0.038965760580894265)
            (43,1.6003834322064532) +- (0,0.04148163293164837)
            (44,1.5332244710420522) +- (0,0.037491508913090964)
            (45,1.507963070368759) +- (0,0.036576795422376546)
            (46,1.4273750873335878) +- (0,0.029506827423703186)
            (47,1.4678007665662838) +- (0,0.026237968265226112)
            (48,1.5724085005239592) +- (0,0.039916445490382466)
            (49,1.5445106128629142) +- (0,0.033111889689449184)
            (50,1.4908295643570577) +- (0,0.02566598197000549)
        }
        ;
   % \addlegendentry {${\textrm{elastic net}}$}
    	\addplot[color=green, dashed, ultra thick, each nth point=1] coordinates {(15,1) (50,1)};
    	% \addlegendentry {${\textrm{Gurobi}}$}
\end{axis}
\end{tikzpicture}} \hfil
 \subfloat[Solution time (s) (SNR 1)]{
  \label{fig:Sol-time-snr-1}
 \begin{tikzpicture}
\begin{axis}[legend style= {at={(0.29,0.95)},anchor=north, font=\tiny}, scale only axis, height=2.5cm, width=0.2*\textwidth, xmajorgrids, ymajorgrids, font=\tiny, ticklabel style = {font=\tiny}, error bars/y dir=both, error bars/y explicit, error bars/error bar style={line width=.1pt}, xlabel={$m$ (size of the matrix)}, ylabel={Solution time (s) in log scale}, ymode={log}, ytick = {0.001, 0.01, 0.1, 1, 10}, each nth point=3]
    \addplot+[no marks, style={{ultra thick}}, color={blue}]
        coordinates {
            (15,0.3053642160599999) +- (0,0.1316260523501432)
            (16,0.18937755420000002) +- (0,0.0025393491067535155)
            (17,0.20158431304000005) +- (0,0.0022111674730625913)
            (18,0.24436798363999995) +- (0,0.0028322720650530956)
            (19,0.24342601918000004) +- (0,0.0030708854109489214)
            (20,0.2500540526000001) +- (0,0.003042880954015661)
            (21,0.25030316318000007) +- (0,0.0029616382181274535)
            (22,0.2650490880200001) +- (0,0.002624346951184519)
            (23,0.299346197) +- (0,0.0037050038727899576)
            (24,0.30812180552) +- (0,0.0038241805596974077)
            (25,0.30664735640000007) +- (0,0.0037516885894489154)
            (26,0.3143835688600001) +- (0,0.0032085065247055306)
            (27,0.33024969694) +- (0,0.003978698583726053)
            (28,0.3511856123000001) +- (0,0.0034759861662873232)
            (29,0.37519060138) +- (0,0.004847221071636649)
            (30,0.3584801537400001) +- (0,0.0033143134697867695)
            (31,0.40051097550000003) +- (0,0.005148722540892863)
            (32,0.41223116072000004) +- (0,0.0049447912553197025)
            (33,0.4730611328799999) +- (0,0.0045955849647794245)
            (34,0.4460868855400001) +- (0,0.005199046947577577)
            (35,0.44477649254) +- (0,0.0059373107709058985)
            (36,0.43859766154000007) +- (0,0.004331353376400012)
            (37,0.44776359084) +- (0,0.004892854820680536)
            (38,0.47039705602000004) +- (0,0.00458934913762619)
            (39,0.5214991535600002) +- (0,0.005591550951442787)
            (40,0.4831747337600001) +- (0,0.004491258126246854)
            (41,0.48825633163999993) +- (0,0.004135850917421541)
            (42,0.4984895562800001) +- (0,0.005326155238909076)
            (43,0.46760082436) +- (0,0.0035232129829275522)
            (44,0.48110233983999995) +- (0,0.004909057690802182)
            (45,0.4779055301400001) +- (0,0.004709974533358516)
            (46,0.4839687459199999) +- (0,0.00480554901311321)
            (47,0.48270108191999994) +- (0,0.003994383189832419)
            (48,0.6992690688800001) +- (0,0.008090084351664781)
            (49,0.7132934633200001) +- (0,0.008700669910228817)
            (50,0.71823742166) +- (0,0.009347590338046377)
        }
        ;
    % \addlegendentry {NExOS}
    \addplot+[no marks, style={{ultra thick}}, color={green}]
        coordinates {
            (15,0.2418809038399999) +- (0,0.11601421611493536)
            (16,0.10898331324000005) +- (0,0.007196383819645266)
            (17,0.14210868200000004) +- (0,0.0073869640344937405)
            (18,0.16071956642000004) +- (0,0.01151990020691928)
            (19,0.14745174653999996) +- (0,0.012525880483000927)
            (20,0.18613791547999997) +- (0,0.015249849036761218)
            (21,0.161523742) +- (0,0.009159439716637638)
            (22,0.21096982571999995) +- (0,0.01778579788231763)
            (23,0.25338608698) +- (0,0.028581970000496077)
            (24,0.3641255469400001) +- (0,0.0765530311833455)
            (25,0.34131720019999995) +- (0,0.06431380495147898)
            (26,0.35707502120000006) +- (0,0.03367472443333466)
            (27,0.31701067254000004) +- (0,0.03677685277719205)
            (28,0.4352281269800001) +- (0,0.08109577797896396)
            (29,0.7232293605600002) +- (0,0.16259124009511572)
            (30,0.49005696237999996) +- (0,0.06760707428350449)
            (31,0.6721294737200003) +- (0,0.09287238449344841)
            (32,0.68021620134) +- (0,0.05921359010105124)
            (33,0.9573926339999999) +- (0,0.12777224910170704)
            (34,1.5028646065599998) +- (0,0.25574479933247285)
            (35,1.4507990270199997) +- (0,0.23679295048696847)
            (36,1.5243376364400003) +- (0,0.22075431544826246)
            (37,1.2992101596000003) +- (0,0.21485239260146718)
            (38,2.2126197214000007) +- (0,0.28367741350957254)
            (39,2.23634261078) +- (0,0.30061961059197556)
            (40,3.2507556194200014) +- (0,0.7728653712936803)
            (41,3.643085636159999) +- (0,0.7432953745290475)
            (42,10.4435998401) +- (0,5.67300316928722)
            (43,5.9871548929) +- (0,1.7099600687141308)
            (44,4.941157824460001) +- (0,0.6671711260370462)
            (45,7.388246834739996) +- (0,1.222973702010374)
            (46,7.565676840679998) +- (0,1.639987959572876)
            (47,10.928940426299999) +- (0,2.793609055896612)
            (48,17.644363577899995) +- (0,4.790567550767755)
            (49,14.005640854400001) +- (0,3.0885145900762225)
            (50,23.74782526622) +- (0,6.908751424380061)
        }
        ;
    % \addlegendentry {Gurobi}
    \addplot+[no marks, style={{ultra thick}}, color={red}]
        coordinates {
            (15,0.014485045839999999) +- (0,0.0008841408203789709)
            (16,0.01836047924) +- (0,0.001170589075257544)
            (17,0.016520655959999995) +- (0,0.0008120476227031921)
            (18,0.020760376320000003) +- (0,0.0012605056646113454)
            (19,0.02183949826) +- (0,0.0012284209856565152)
            (20,0.022834329599999998) +- (0,0.0011170164906688259)
            (21,0.026947384880000008) +- (0,0.0014572652290339178)
            (22,0.023090705960000003) +- (0,0.00016846797044401792)
            (23,0.03252384592000001) +- (0,0.0018421307747785972)
            (24,0.03063711148) +- (0,0.00161431027810366)
            (25,0.03429218752) +- (0,0.0017575716799687802)
            (26,0.03657170310000001) +- (0,0.0022874437943614366)
            (27,0.04060387276) +- (0,0.002326952484952428)
            (28,0.056555486879999985) +- (0,0.0026178439241156554)
            (29,0.05700598987999997) +- (0,0.002935557639552355)
            (30,0.06596910151999999) +- (0,0.003492712597852397)
            (31,0.05073403376000002) +- (0,0.0028889286145509716)
            (32,0.04922206451999999) +- (0,0.002339683462656945)
            (33,0.051768121400000026) +- (0,0.002201132425559565)
            (34,0.059318188560000006) +- (0,0.003730329608731944)
            (35,0.061278746239999994) +- (0,0.0029919258118447415)
            (36,0.07454707609999998) +- (0,0.0048563016641312035)
            (37,0.07338918427999999) +- (0,0.002984021606325012)
            (38,0.06391948935999998) +- (0,0.0015682170171080563)
            (39,0.06717056232000002) +- (0,0.0021877845519153736)
            (40,0.09372226852000003) +- (0,0.005295033380227362)
            (41,0.09143809654) +- (0,0.004228263445002035)
            (42,0.10196245435999998) +- (0,0.005273298450064095)
            (43,0.08635558104000003) +- (0,0.0035861888109208623)
            (44,0.10149342102000004) +- (0,0.00546288197394361)
            (45,0.08766947128) +- (0,0.0036326336911238367)
            (46,0.10413924547999998) +- (0,0.005128695325619211)
            (47,0.09755995635999998) +- (0,0.004545731225163858)
            (48,0.10439048186) +- (0,0.004740175898077344)
            (49,0.11055871628000002) +- (0,0.005022853752314182)
            (50,0.10897732310000002) +- (0,0.004192188325208908)
        }
        ;
    % \addlegendentry {elastic net}
\end{axis}
\end{tikzpicture}
 }

\caption{Sparse regression problem: comparison between \textsf{NExOS} (shown in \textcolor{blue}{blue}), \textsf{glmnet} (shown in \textcolor{red}{red}), and \textsf{Gurobi} (shown in \textcolor{green}{green}). The first and second rows correspond to SNR 6 and SNR 1, respectively. For each SNR, the first column  compares support recovery, the second column shows how close the objective value of the solution found by each algorithm gets to the optimal objective value (normalized as 1), and the third column shows the solution time (s) of each algorithm. \label{fig:Figure-sparse-reg}}
\end{figure}

\paragraph{Results} Figure \ref{fig:Figure-sparse-reg} compares \textsf{NExOS} (shown in \textcolor{blue}{blue}), \textsf{glmnet} (shown in \textcolor{red}{red}) and \textsf{Gurobi} (shown in \textcolor{green}{green}) for solving (\ref{eq:reg-sel}). The results displayed in the figures are averaged over 50 simulations
for each value of $m$, and also show one-standard-error bands that
represent one standard deviation confidence interval around the mean.

Figures \ref{fig:Figure-sparse-reg}(a) and \ref{fig:Figure-sparse-reg}(d) show the support recovery
(\%) of the solutions found by \textsf{NExOS}, \textsf{glmnet}, and \textsf{Gurobi} for SNR 6 and SNR 1, respectively. Given a solution $x$ and true signal $x^{\textrm{True}},$
the support recovery is defined as $\sum_{i=1}^{d}1_{\{\textrm{sign}(x_{i})=\textrm{sign}(x_{i}^{\textrm{True}})\}}/d,$
where $1_{\{\cdot\}}$ evaluates to $1$ if $(\cdot)$ is true and
$0$ else\emph{, }and $\textrm{sign}(t)$ is $1$ for $t>0$, $-1$
for $t<0$, and $0$ for $t=0$\emph{.} So, higher the support recovery,
better is the quality of the found solution. For both SNRs, \textsf{NExOS} and \textsf{Gurobi} have almost identical support recovery. For the high SNR, \textsf{NExOS} 
recovers most of the original signal's support and is better than \textsf{glmnet} consistently. On average, \textsf{NExOS}
recovers 4\% more of the support than \textsf{glmnet}. However, this behaviour changes for the low SNR, where \textsf{glmnet} recovers  1.26\% more of the support than \textsf{NExOS}. This differing behavior in low and high SNR is consistent with the observations made in \cite{hastie2017extended}. %

Figures \ref{fig:Figure-sparse-reg}(b) and \ref{fig:Figure-sparse-reg}(e) compare the quality of the
solution found by the algorithms in terms of the normalized objective value (the objective value of the found solution divided by the otimal objective value) for SNR 6 and SNR 1, respectively. As \textsf{Gurobi}'s spatial branch-and-bound algorithm finds certifiably globally optimal solution to (\ref{eq:reg-sel}), its normalized objective value is always 1, though the runtime is orders of magnitude slower than \textsf{glmnet} and \textsf{NExOS} (see the next paragraph). The closer the normalized objective value is to 1, better is the quality of the solution in terms of minimizing the objective value. We see
that for the high SNR, on average \textsf{NExOS} is able to find a solution that is very close to the globally optimal solution, whereas the solution found by \textsf{glmnet} has worse objective value on average. For the low SNR, on average the normalized objective values of the solutions found by both \textsf{NExOS} and \textsf{glmnet} get worse, though \textsf{NExOS} does better than \textsf{glmnet} in this case as well.

Finally, in Figures \ref{fig:Figure-sparse-reg}(c) and \ref{fig:Figure-sparse-reg}(f), we compare the solution
times (in seconds and on log scale) of the algorithms for SNR 6 and SNR 1, respectively. 
We
see that\texttt{ }\textsf{glmnet} is slightly faster than \textsf{NExOS}. This slower performance is due to
the fact that \textsf{NExOS} is a general purpose method, %
whereas \textsf{glmnet} is specifically
optimized for the convexified sparse regression problem with a 
specific cost function. %
For smaller problems, \textsf{Gurobi} is somewhat faster than \textsf{NExOS}, however once we go beyond $m\geq27$, the solution time by \textsf{Gurobi} starts to increase drastically. Beyond $m\geq50$, comparing the solution times is not meaningful as \textsf{Gurobi} cannot find a solution in 2 minutes, whereas \textsf{NExOS} takes less than 30 seconds.

\subsubsection{Experiments and results for real-world dataset}

\paragraph{Description of the dataset}

We now investigate the performance of our algorithm on a real-world, publicly available dataset called the \texttt{weather prediction dataset}, where
we consider the problem of predicting the temperature half a day in advance in 30 US and Canadian Cities along with 6
Israeli cities. The dataset contains hourly measurements of weather attributes \emph{e.g.,} temperature, humidity, air pressure, wind speed, and so on. The dataset has $m = 45,231$ instances along with $d = 1,800$ attributes. The dataset is preprocessed in the same manner as described in \cite[\S 8.3]{bertsimas2021slowly}. Our goal is to predict the temperature half a day in advance as a linear %
 function of the attributes, where at most $k$ attributes can be nonzero.
We include a bias term in our model, \emph{i.e.}, in (\ref{eq:reg-sel})
we set $A=[\bar{A}\mid\ones]$. We randomly split 80\% of the data into the training 
set and 20\% of the data into the test set.

\paragraph{Results}

Figure \ref{fig:RMS-error-vs-k-m-weather} shows the RMS error
for the training datasets and the test datasets for both \textsf{NExOS} and \textsf{glmnet}. The results for training and test datasets are reasonably similar for each value of $k$. This
gives us confidence that the sparse regression model will have similar
performance on new and unseen data. This also suggests that our model
does not suffer from over-fitting. We also see that, for $k\geq20$ and $k \geq 5$,
none of the errors for \textsf{NExOS} and \textsf{glmnet} drop significantly, respectively. For smaller $k \leq 10$, \textsf{glmnet}
does better than \textsf{NExOS}, but beyond $k \geq 10$, \textsf{NExOS} performs better than \textsf{glmnet}.

\begin{figure}
\centering
\begin{tikzpicture}
\begin{axis}[legend style= {font=\tiny}, scale only axis, height=2.5cm, width=0.7*\textwidth, font=\tiny, ticklabel style = {font=\tiny}, xmajorgrids, ymajorgrids,  xlabel={$k$ (cardinality)}, ylabel={RMS error}]
    \addplot+[no marks, style={{ultra thick}}, color={blue}]
        coordinates {
            (3,19.918652413042143)
            (4,19.163053488878123)
            (5,18.43083126792745)
            (6,17.721443854042068)
            (7,17.00460578021895)
            (8,16.31150448055227)
            (9,15.637960428202609)
            (10,15.000819629300565)
            (11,14.4050731733844)
            (12,13.896453093769338)
            (13,13.443095024548697)
            (14,13.050767333877767)
            (15,12.725088874452467)
            (16,12.471285753401729)
            (17,12.29084159291348)
            (18,12.18474418607375)
            (19,12.148232282923907)
            (20,12.147032269874158)
            (21,12.146853163621339)
            (22,12.14675678713992)
            (23,12.146719893189172)
            (24,12.146816390869667)
            (25,12.14672195901336)
            (26,12.146642700493913)
            (27,12.14669652157496)
            (28,12.146652323954092)
            (29,12.14667162356629)
            (30,12.146645452943003)
        }
        ;
    \addlegendentry {\textsf{NExOS} training error}
    \addplot+[no marks, style={{ultra thick}}, color={teal}]
        coordinates {
            (3,23.99714953823415)
            (4,22.977910488662648)
            (5,21.978109273494244)
            (6,21.070681330418086)
            (7,20.22348969536787)
            (8,19.3190658652557)
            (9,18.36951558635656)
            (10,17.451972662409656)
            (11,16.572094958891718)
            (12,15.986096588808303)
            (13,15.446914768400816)
            (14,14.95985939061366)
            (15,14.529377456260884)
            (16,14.160611002514251)
            (17,13.734338076991804)
            (18,13.40428421314229)
            (19,13.165340383757266)
            (20,13.154876952812351)
            (21,13.153051286395742)
            (22,13.15231752267871)
            (23,13.151806413480367)
            (24,13.15096792094439)
            (25,13.152247355896465)
            (26,13.151977634577129)
            (27,13.151090757888335)
            (28,13.150775651106365)
            (29,13.150074161710828)
            (30,13.150329017926751)
        }
        ;
    \addlegendentry {\textsf{NExOS} test error}
    \addplot+[no marks, style={{ultra thick, dotted}}, color={brown}]
        coordinates {
            (3,19.948793854794708)
            (4,16.716963435519315)
            (5,15.782596723845502)
            (6,15.782596723845502)
            (7,15.76436550138061)
            (8,15.74495690631091)
            (9,15.735320877871048)
            (10,15.734742368969682)
            (11,15.73360694743513)
            (12,15.733313569745702)
            (13,15.730658536614369)
            (14,15.729421323613115)
            (15,15.728982230914587)
            (16,15.728700947444855)
            (17,15.727622744908635)
            (18,15.726949612265537)
            (19,15.726517281031944)
            (20,15.725631385291605)
            (21,15.725523160534893)
            (22,15.725430490785685)
            (23,15.725166592581415)
            (24,15.725079053677291)
            (25,15.724633497998994)
            (26,15.724633497998994)
            (27,15.724416144875063)
            (28,15.724147889908124)
            (29,15.724020309613467)
            (30,15.723975078759137)
        }
        ;
    \addlegendentry {\textsf{glmnet} training error}
    \addplot+[no marks, style={{ultra thick, dashed}}, color={red}]
        coordinates {
            (3,24.03567059253795)
            (4,19.57298812346571)
            (5,18.20919093739861)
            (6,18.20919093739861)
            (7,18.181607883796968)
            (8,18.151733714062754)
            (9,18.136812165479814)
            (10,18.135352122815693)
            (11,18.13227808759265)
            (12,18.13143912926094)
            (13,18.12354805907545)
            (14,18.120041250393133)
            (15,18.118878020760015)
            (16,18.11795398255339)
            (17,18.11476159372627)
            (18,18.112533650445556)
            (19,18.111402630964164)
            (20,18.108114912797586)
            (21,18.107724335297075)
            (22,18.1073754185149)
            (23,18.10696276857742)
            (24,18.106844420104846)
            (25,18.10617511316489)
            (26,18.10617511316489)
            (27,18.10569791344396)
            (28,18.10496120267165)
            (29,18.10469040749312)
            (30,18.104648863596534)
        }
        ;
    \addlegendentry {\textsf{glmnet} test error}
\end{axis}
\end{tikzpicture}

\caption{RMS error vs $k$ (cardinality) for the weather prediction problem.\label{fig:RMS-error-vs-k-m-weather}}
\end{figure}

\subsection{Affine rank minimization problem\label{subsec:Affine-rank-minimization}}

\paragraph{Problem description}

In (\ref{eq:reg-sel}), we set $\mathcal{X}:=\{X\in\mathbf{R}^{m\times d}\mid\mathop{{\bf rank}}(X)\leq r,\|X\|_{2}\leq\Gamma\},$
and $f(X):=\left\Vert \mathcal{A}(X)-b\right\Vert _{2}^{2}$. To compute
the proximal operator of $f$, we use the formula in \cite[\S 6.1.1]{boydProx}.
Finally, we use the formula in \cite[page 14]{Diamond2018} for projecting
onto $\mathcal{X}$. Now we are in a position to apply the \textsf{NExOS}
to this problem. 

\paragraph{Summary of the experiments performed}

\emph{First,} we apply \textsf{NExOS} to solve (\ref{eq:lrm-estimation})
for synthetic datasets, where we observe how the algorithm performs
in recovering a low-rank matrix given noisy measurements and also compare \textsf{NExOS} with \textsf{NCVX}---an \textsf{ADMM}-based algorithm \cite{Diamond2018}. \emph{Second,
}we apply \textsf{NExOS} to a real-world dataset (\texttt{MovieLens
1M Dataset}) to see how our algorithm performs in solving a matrix-completion
problem). 

\subsubsection{Experiments and results for synthetic dataset}

\paragraph{Data generation process and setup}

We generate the data as follows similar to \cite{Diamond2018}.
We vary $m$ (number of rows of the decision variable $X$) from 50
to 75 with a linear spacing of 5, where we take $d=2m$, and rank
to be equal to $m/10$ rounded to the nearest integer. For each value
of $m$, we create 25 random instances as follows. The operator $\mathcal{A}$
is drawn from an iid normal distribution with $\mathcal{N}(0,1)$
entries. Similarly, we create the low rank matrix $X_{\textrm{True}}$
with rank $r$, first drawn from an iid normal distribution with $\mathcal{N}(0,1)$
entries, and then truncating the singular values that exceed $\Gamma$ to $0$. Signal-to-noise
ratio is taken to be around 20 by following the same method described
for the sparse regression problem.

\begin{figure}[htp]
\centering
\subfloat[Fixed point gap representing convergence of the iterates vs $m$]{%
\label{fig:Normalized-fixed-point-gap}%
\begin{tikzpicture}
\begin{axis}[legend style= {at={(0.35,0.85)},anchor=north, font =\tiny}, scale only axis, height=3cm, width=0.25*\textwidth, font=\tiny, ticklabel style = {font=\tiny}, xmajorgrids, ymajorgrids, error bars/y dir=both, error bars/y explicit, xlabel={$m$ (matrix size)}, ylabel={Fixed point gap},
ymin = 1e-5, ymax = 4e-1,
% ytick = {1e-2, 1e-3, 1e-2, 1e-3, 1e-4, 1e-5},
% ymode = log
]
    \addplot+[no marks, style={{ultra thick}}, color={blue}]
        coordinates {
            (50,8.877821599999999e-5) +- (0,1.6929656284429802e-6)
            (55,8.9241452e-5) +- (0,1.4233310572517324e-6)
            (60,9.203202399999999e-5) +- (0,1.0954852870041967e-6)
            (65,8.914712e-5) +- (0,1.4050806336529822e-6)
            (70,9.198874799999998e-5) +- (0,1.3295495545248897e-6)
            (75,8.880418000000001e-5) +- (0,1.4081166837067636e-6)
        }
        ;
    \addlegendentry {NExOS}
    \addplot+[no marks, style={{ultra thick}}, color={red}]
        coordinates {
            (50,0.17579747281990327) +- (0,0.009129490774399406)
            (55,0.17723715494945833) +- (0,0.00903170929784641)
            (60,0.1790639971698293) +- (0,0.006140914072725028)
            (65,0.18143017983208082) +- (0,0.006666773680454051)
            (70,0.17096295831155847) +- (0,0.007937624139201205)
            (75,0.17696715055823922) +- (0,0.00845014836201521)
        }
        ;
    \addlegendentry {NCVX}
\end{axis}
\end{tikzpicture}
}
\qquad
\subfloat[Maximum absolute error in recovering the original matrix vs $m$]{%
\label{fig:Maximum-absolute-error-in-recovery-ARM}%
\begin{tikzpicture}
\begin{axis}[legend style= {at={(0.35,0.95)},anchor=north, font =\tiny}, scale only axis, height=3cm, width=0.25*\textwidth, font=\tiny, ticklabel style = {font=\tiny}, xmajorgrids, ymajorgrids, error bars/y dir=both, error bars/y explicit, xlabel={$m$ (matrix size)}, ylabel={$\| X_{\textrm{True}}-X^{\star}_\textrm{sol}\|_\textrm{max}$},
ymin = 0.0001, ymax = 1
]
    \addplot+[no marks, style={{ultra thick}}, color={blue}]
        coordinates {
            (50,0.004502298800000001) +- (0,0.00012324852008079718)
            (55,0.0049530036) +- (0,0.00010436583514394609)
            (60,0.0042427424) +- (0,9.234358672782859e-5)
            (65,0.003769103199999999) +- (0,8.948930266931348e-5)
            (70,0.003924134800000001) +- (0,6.932350138089294e-5)
            (75,0.004148300400000002) +- (0,6.268112585268392e-5)
        }
        ;
    \addlegendentry {NExOS}
        \addplot+[no marks, style={{ultra thick}}, color={red}]
        coordinates {
            (50,0.5536576658315994) +- (0,0.02108077463640419)
            (55,0.6283928625168838) +- (0,0.024114581321585893)
            (60,0.6169454932679834) +- (0,0.01460175849637596)
            (65,0.6130568172170452) +- (0,0.011957049380103568)
            (70,0.6495109670079592) +- (0,0.023393757975101424)
            (75,0.690443608143288) +- (0,0.029727205023025522)
        }
        ;
    \addlegendentry {NCVX}
\end{axis}
\end{tikzpicture}
}
\qquad
\subfloat[Ratio of training losses of the true matrix $X_{\textrm{True}}$ and the solution found by \textsf{NExOS} vs $m$ ]{%
\label{fig:Ratio-of-objective-ARM}%
\begin{tikzpicture}
\begin{axis}[legend style= {at={(0.35,0.6)},anchor=north}, scale only axis, height=3cm, width=0.25*\textwidth, xmajorgrids, ymajorgrids,font=\tiny, ticklabel style = {font=\tiny}, error bars/y dir=both, error bars/y explicit, xlabel={$m$ (matrix size)}, ylabel={$p^{\star}_{\textrm{True}}/p^{\star}_{\textrm{sol}}$}, ymin = 0.01, ymax = 2, xmin = 45, xmax = 80]
        \addplot+[no marks, style={{ultra thick}}, color={blue}]
        coordinates {
            (50,1.3954668000000001) +- (0,0.004953254207084468)
            (55,1.4546639999999997) +- (0,0.004323372950216841)
            (60,1.3978512) +- (0,0.004355852894669422)
            (65,1.3592255999999998) +- (0,0.004139639923149514)
            (70,1.3973892) +- (0,0.0035280365606571266)
            (75,1.4459139999999997) +- (0,0.0028772139301762043)
        }
        ;
    \addlegendentry {NExOS}
    \addplot+[no marks, style={{ultra thick}}, color={red}]
        coordinates {
            (50,0.05425952183725204) +- (0,0.004578565649628778)
            (55,0.05306117068020596) +- (0,0.006744681064457461)
            (60,0.04865386964658613) +- (0,0.002826853024806257)
            (65,0.04791113777654875) +- (0,0.0021309055926627177)
            (70,0.07207106491536845) +- (0,0.024802948940188318)
            (75,0.0650111037805949) +- (0,0.012809121617246396)
        }
        ;
        \addlegendentry {NCVX}
	\addplot[color=teal, dashed, thick] coordinates {(45,1) (80,1)};
\end{axis}
\end{tikzpicture}
}
\caption{Affine rank minimization problem: comparison between solutions found by \textsf{NExOS} and  \textsf{NCVX} algorithm by \cite{Diamond2018}. } \label{fig_aff_rank_min}
\end{figure}

\paragraph{Results}
{
The results displayed in Figure \ref{fig_aff_rank_min} average over 50 simulations
for each value of $m$ and also show one standard error band. We compare \textsf{NExOS}, with \textsf{NCVX}---an \textsf{ADMM}-based algorithm \cite{Diamond2018}. 

Fig~\ref{fig:Normalized-fixed-point-gap} plots the normalized fixed point gap of the iterates for both algorithms computed by $\|X_{\textsf{Alg}}^{\star} - Y_{\textsf{Alg}}^{\star} \|/\|X_\textrm{True}\|$with $\textsf{Alg} \in \{\textsf{NExOS}, \textsf{NCVX}\}$ and $X_{\textsf{Alg}}^{\star}, Y_{\textsf{Alg}}^{\star}$ representing the final iterates produced by the algorithms. This plot shows that \textsf{NCVX} iterates have a fixed point gap larger than $0.17$, \emph{i.e.}, the iterates do not converge within a reasonable fixed point gap. On the other hand, \textsf{NExOS} iterates converge with a normalized fixed-point gap reaching the desired tolerance of less than or equal to $10^{-4}$ for each instance. 

Figure \ref{fig:Maximum-absolute-error-in-recovery-ARM} shows how
well \textsf{NExOS} and \textsf{NCVX} recovers the original matrix
$X_{\textrm{True}}$. 
To quantify the recovery, we compute the max
norm of the difference matrix $\|X_{\textrm{True}}-X_{\textsf{Alg}}^{\star}\|_{\textrm{max}}=\max_{i,j}|X_{\textrm{True}}(i,j)-X_{\textsf{Alg}}^{\star}(i,j)|,$
where the solution found by \textsf{Alg} is denoted by $X_{\textsf{Alg}}^{\star}$.
We see that the worst-case component-wise error is very small (smaller than 0.005 for each instance) in all
the cases for \textsf{NExOS}, but for \textsf{NCVX}, it is larger than $0.5$ for each instance. In other words, the solution found by \textsf{NExOS} is much closer to the ground truth as compared to \textsf{NCVX}.   

Finally, we show how the training loss of the solutions
computed by \textsf{NExOS} and \textsf{NCVX} compare with the original matrix $X_{\textrm{True}}$
in Figure \ref{fig:Ratio-of-objective-ARM}.  Note that for \textsf{NExOS}, the ratio $p^\star_\textrm{True}/p^\star_{\textrm{sol}}$ is
larger than one in most cases, \emph{i.e.}, \textsf{NExOS}  find a
solutions with smaller cost compared to $X_{\textrm{True}}$.
This is due to the fact that under the signal-to-noise ratio that we consider, the problem data can be explained better by another matrix with
a lower training loss. That being said, $X_{\textsf{NExOS}}^{\star}$
is not too far from $X_{\textrm{True}}$ component-wise as we saw
in Figure \ref{fig:Maximum-absolute-error-in-recovery-ARM}. On the other hand, for \textsf{NCVX} algorithm, the ratio $p^\star_\textrm{True}/p^\star_{\textrm{sol}}$ is smaller than $0.05$ for each instance, \emph{i.e.}, the objective value of the solutions is $20$ times worse than that of the original signal.
}

\subsubsection{Experiments and results for real-world dataset: matrix completion
problem}

\paragraph{Description of the dataset}

To investigate the performance of our problem on a real-world dataset,
we consider the publicly available \texttt{MovieLens 1M Dataset}. This
dataset contains 1,000,023 ratings for 3,706 unique movies;
these recommendations were made by 6,040 MovieLens users. The rating
is on a scale of 1 to 5. If we construct a matrix of movie ratings
by the users (also called the preference matrix), denoted by $Z$,
then it is a matrix of 6,040 rows (each row corresponds to a user)
and 3,706 columns (each column corresponds to a movie) with only 4.47\%
of the total entries are observed, while the rest being
missing. 
Our goal is to complete this matrix, under the assumption that the
matrix is low-rank. For more details about the model, see \cite[\S 8.1]{jain2017non}.

To gain confidence in the generalization ability of this model, we
use an out-of-sample validation process. By random selection, we split
the available data into a training set (80\% of the total data) and
a test set (20\% of the total data). We use the training set as the
input data for solving the underlying optimization process, and the
held-out test set is used to compute the test error for each value
of $r$. The best rank $r$ corresponds to the point beyond which the
improvement is rather minor. We tested rank values
$r$ ranging in $\{1,3,5,7,10,20,25,30,35\}$. We compute the RMS error as follows. Let $\Omega_{\textrm{test}}$ be the index set corresponding to the
test data. If $X_\textsf{NExOS}^\star$ is the matrix returned by \textsf{NExOS},
then the corresponding RMS error is computed by using the formula 
\[
\textrm{RMS}=\sqrt{\frac{\sum_{(i,j)\in\Omega_{\textrm{test}}}\left(\left(X_\textsf{NExOS}^\star\right)_{ij}-Z_{ij}\right)^{2}}{|\Omega_{\textrm{test}}|}}, 
\]
where $|\Omega_{\textrm{test}}|$ is the number of elements in $\Omega_{\textrm{test}}$.

\paragraph{Matrix completion problem}

The matrix
completion problem is:
\begin{equation}
\begin{array}{ll}
\textup{minimize} & \sum_{(i,j)\in\Omega}(X_{ij}-Z_{ij})^{2}+(\beta/2)\|X\|_{F}^{2}\\
\textup{subject to} & \mathop{{\bf rank}}(X)\leq r, \quad \|X\|_{2}\leq\Gamma,
\end{array}\tag{MC}\label{eq:matrix_completion}
\end{equation}
where $Z\in\mathbf{R}^{m\times d}$ is the matrix whose entries $Z_{ij}$
are observable for $(i,j)\in\Omega$. Based on these observed entries,
our goal is to construct a matrix $X\in\mathbf{R}^{m\times d}$ that
has rank $r$. The problem above can be written as a special case
of affine rank minimization problem (\ref{eq:lrm-estimation}). %

\begin{figure}[htp]
\centering
\subfloat[Fixed point gap representing convergence of the iterates vs $r$]{%
\label{fig:Normalized-fixed-point-gap-matrix-comp}%
\begin{tikzpicture}
\begin{axis}[legend style= {at={(0.5,1.3)},anchor=north, font =\tiny}, font=\tiny, ticklabel style = {font=\tiny}, scale only axis, height=3cm, width=0.25*\textwidth, xmajorgrids, ymajorgrids, xlabel={$r$ (rank)}, ylabel={Fixed point gap (log scale)}, ymode={log}]
% \begin{axis}[legend style= {at={(0.5,1.25)},anchor=north}, scale only axis, height=5cm, width=0.85*\textwidth, xmajorgrids, ymajorgrids, xlabel={$r$ (rank)}, ylabel={Output}, ymode={log}]
% \begin{axis}[legend style= {at={(0.5,1.2)},anchor=north}, xmajorgrids, ymajorgrids, xlabel={$r$ (rank)}, ylabel={Output}, ymode={log}]
    \addplot+[no marks, style={{ultra thick}}, color={blue}]
        coordinates {
            (1,1.7877307573188617e-7)
            (3,2.1468138455826846e-7)
            (5,2.2503163021880823e-7)
            (7,2.688980536014185e-7)
            (10,2.756973804807217e-7)
            (20,3.3348647709985357e-7)
            (25,4.809719265086443e-7)
            (30,3.2295530805193096e-7)
            (35,3.3722291625792877e-7)
        }
        ;
    \addlegendentry {NExOS}
    \addplot+[no marks, style={{ultra thick}}, color={red}]
        coordinates {
            (1,0.6456249810101062)
            (3,0.06539917004073625)
            (5,1.6325874776073208)
            (7,0.19483541564013615)
            (10,1.1604699476938969)
            (20,1.533527325184832)
            (25,0.37503776988914583)
            (30,1.5989607944356474)
            (35,0.03480581196158074)
        }
        ;
    \addlegendentry {NCVX}
\end{axis}
\end{tikzpicture}
}
\qquad
\subfloat[Training and test error for \textsf{NExOS} vs $r$]{%
\label{fig:train-vs-test-NExOS}%
\begin{tikzpicture}
\begin{axis}[legend style= {at={(0.5,1.3)},anchor=north, font =\tiny}, font=\tiny, ticklabel style = {font=\tiny}, scale only axis, height=3cm, width=0.25*\textwidth, xmajorgrids, ymajorgrids, xlabel={$r$ (rank)}, ylabel={RMS error}]
    \addplot+[no marks, style={{ultra thick}}, color={blue}]
        coordinates {
            (1,0.9882026033642594)
            (3,0.9819979760922203)
            (5,0.9778998754920383)
            (7,0.9746889190478395)
            (10,0.9707307923981745)
            (20,0.9611455148401141)
            (25,0.9571407976857484)
            (30,0.9534178894852361)
            (35,0.9498602643507843)
        }
        ;
    \addlegendentry {\textsf{NExOS} training error}
    \addplot+[no marks, style={{ultra thick}}, color={red}]
        coordinates {
            (1,0.9886462102018714)
            (3,0.983482376933635)
            (5,0.980571518395447)
            (7,0.9784578474795791)
            (10,0.9761085276960186)
            (20,0.9735470044640683)
            (25,0.9729246790567139)
            (30,0.9725772294265248)
            (35,0.9725833740897399)
        }
        ;
    \addlegendentry {\textsf{NExOS} test error}
\end{axis}
\end{tikzpicture}
}
\qquad
\subfloat[Training and test error for \textsf{NCVX} vs $r$]{%
\label{fig:train-vs-test-NCVX}%
\begin{tikzpicture}
\begin{axis}[legend style= {at={(0.5,1.3)},anchor=north, font =\tiny}, font=\tiny, ticklabel style = {font=\tiny}, scale only axis, height=3cm, width=0.25*\textwidth, xmajorgrids, ymajorgrids, xlabel={$r$ (rank)}, ylabel={RMS error}]
    \addplot+[no marks, style={{ultra thick}}, color={cyan}]
        coordinates {
            (1,0.919529696530222)
            (3,0.8761892997180748)
            (5,0.8458250184207712)
            (7,0.8229431298385127)
            (10,0.7983341006107252)
            (20,0.734566512157295)
            (25,0.7059710488642934)
            (30,0.6860417712850664)
            (35,0.6555054564700517)
        }
        ;
    \addlegendentry {\textsf{NCVX} training error }
    \addplot+[no marks, style={{ultra thick}}, color={magenta}]
        coordinates {
            (1,0.9280343419630657)
            (3,0.8995358793722824)
            (5,0.8850199457500816)
            (7,0.878554176280105)
            (10,0.8755289323825077)
            (20,0.893644732711329)
            (25,0.907513309316959)
            (30,0.9158833755830187)
            (35,0.9331885891042236)
        }
        ;
    \addlegendentry {\textsf{NCVX} test error }
\end{axis}
\end{tikzpicture}
}
\caption{Matrix completion problem: comparison between solutions found by \textsf{NExOS} and  \textsf{NCVX} algorithm by \cite{Diamond2018}. } \label{fig_mat_comp}
\end{figure}

{
\paragraph{Results}

Figure \ref{fig_mat_comp} compares the solutions found by \textsf{NExOS} and \textsf{NCVX}. 

Fig~\ref{fig:Normalized-fixed-point-gap-matrix-comp} plots the normalized fixed point gap of the iterates for both algorithms calculated by $\|X_{\textsf{Alg}}^{\star} - Y_{\textsf{Alg}}^{\star} \|/\|X_\textrm{True}\|$with $\textsf{Alg} \in \{\textsf{NExOS}, \textsf{NCVX}\}$ and $X_{\textsf{Alg}}^{\star}, Y_{\textsf{Alg}}^{\star}$ representing the final iterates produced by the algorithms. This plot shows that \textsf{NCVX} iterates do not converge within a reasonable fixed point gap, whereas \textsf{NExOS} iterates converge for all the instances with a normalized fixed-point gap less than or equal to $10^{-6}$ for each instance. 

Figure \ref{fig:train-vs-test-NExOS} shows the RMS
error of \textsf{NExOS} for the training datatest and test dataset for each value of
rank $r$. The results for training and test datasets are reasonably
similar for each value of $r$. We observe that beyond rank 15, the
reduction in the test error is rather minor and going beyond this
rank provides only diminishing returns, which is a common occurrence
for low-rank matrix approximation \cite[\S 7.1]{lee2016llorma}. Thus
we can choose the optimal rank to be $15$ for all practical purposes. 

Figure \ref{fig:train-vs-test-NCVX} shows the RMS
error of \textsf{NCVX} for the training dataset and test dataset for each value of
rank $r$.  We see that, unlike \textsf{NExOS}, the test error for \textsf{NCVX} keeps increasing with $r$, whereas the training error \textsf{NCVX} is smaller. Here we note that, because \textsf{NCVX} iterates do not reach a reasonable fixed point gap, the training or test error of \textsf{NCVX} may not provide meaningful information. 
}

\subsection{Factor analysis problem\label{subsec:Factor-analysis-problem}}

\paragraph{Problem description}

The factor analysis model with sparse noise (also known as low-rank
factor analysis model) involves decomposing a given positive semidefinite
matrix as a sum of a low-rank positive semidefinite matrix and a diagonal
matrix with nonnegative entries \cite[page 191]{Hastie2015}. It can
be posed as \cite{bertsimas2017certifiably}:
\begin{equation}
\begin{array}{ll}
\textup{minimize} & \left\Vert \Sigma-X-D\right\Vert _{F}^{2}+(\beta/2)\left(\|X\|_{F}^{2}+\|D\|_{F}^{2}\right)\\
\textup{subject to} & D=\mathbf{diag}(d), \quad d\geq0, \quad X\succeq0, \quad \rank(X)\leq r\\
 & \Sigma-D\succeq0, \quad \|X\|_{2}\leq\Gamma,
\end{array}\tag{FA}\label{eq:factor_analysis}
\end{equation}
where $X\in\mathbf{S}^{p}$ and the diagonal matrix $D\in\mathbf{S}^{p}$
with nonnegative entries are the decision variables, and $\Sigma\in\mathbf{S}_{+}^{p}$,
$r\in\ig_{+},$ and $\Gamma\in\rl_{++}$ are the problem data. A proper solution for (\ref{eq:factor_analysis}) requires that both
$X$ and $D$ are positive semidefinite. The term $\Sigma-D$ has to
be positive semidefinite, else statistical interpretations
of the solution is not impossible \cite[page 326]{ten1998some}.

In (\ref{eq:factor_analysis}), we set $\mathcal{X}:=\{(X,D)\in\mathbf{S}^{p}\times\mathbf{S}^{p}\mid\|X\|_{2}\leq\Gamma,\rank(X)\leq r,D=\mathbf{diag}(d),d\geq0\},$
and $f(X,D)\coloneqq\left\Vert \Sigma-X-D\right\Vert _{F}^{2}+I_{\mathcal{P}}(X,D),$where
$I_{\mathcal{P}}$ denotes the indicator function of the convex set
$\mathcal{P}=\{(X,D)\in\mathbf{S}^{p}\times\mathbf{S}^{p}\mid X\succeq0,D=\mathbf{diag}(d),d\geq0,d\in\rl^{p}\}.$
To compute the projection onto $\mathcal{X}$, we use the formula
in \cite[page 14]{Diamond2018} and the fact that $\mathbf{\Pi}_{\{y\mid y\text{\ensuremath{\geq0}}\}}(x)=\max\{x,0\}$,
where pointwise max is used. The proximal operator for $f$ at $(X,D)$
can be computed by solving:

\[
\begin{array}{ll}
\textup{minimize} & \|\Sigma-\widetilde{X}-\widetilde{D}\|_{F}^{2}+(1/2\gamma)\|\widetilde{X}-X\|_{F}^{2}+(1/2\gamma)\|\widetilde{D}-D\|_{F}^{2}\\
\textup{subject to} & \widetilde{X}\succeq0, \quad \widetilde{D}=\mathbf{diag}(\widetilde{d}), \quad \Sigma-\widetilde{D}\succeq0, \quad \widetilde{d}\geq0,
\end{array}
\]
where $\widetilde{X}\in\mathbf{S}_{+}^{p},$ and $\widetilde{d}\in\rl_{+}^{p}$
(\emph{i.e.}, $\widetilde{D}=\mathbf{diag}(\widetilde{d}$)) are the
optimization variables. Now we are in a position to apply\textsf{
NExOS} to this problem.

\paragraph{Comparison with nuclear norm heuristic}

We compare the solution provided by \textsf{NExOS} to that of the
nuclear norm heuristic, which isthe most well-known heuristic
to approximately solve (\ref{eq:factor_analysis}) \cite{saunderson2012diagonal} via following convex relaxation:
\begin{equation}
\begin{array}{ll}
\textup{minimize} & \left\Vert \Sigma-X-D\right\Vert _{F}^{2}+\lambda\left\Vert X\right\Vert _{*}\\
\textup{subject to} & D=\mathbf{diag}(d),\quad d\geq0,\quad X\succeq0,\\
 & \Sigma-D\succeq0,\quad\|X\|_{2}\leq\Gamma,
\end{array}\label{eq:relaxation}
\end{equation}
where $\lambda$ is a positive parameter that is related to the rank
of the decision variable $X$. Note that, as $X$ is positive semidefinite,
we have its nuclear norm $\|X\|_{*}=\mathbf{tr}(X)$. %

\begin{figure}

\centering
\subfloat[\texttt{bfi} objective value]{\includegraphics[width=3.5cm]{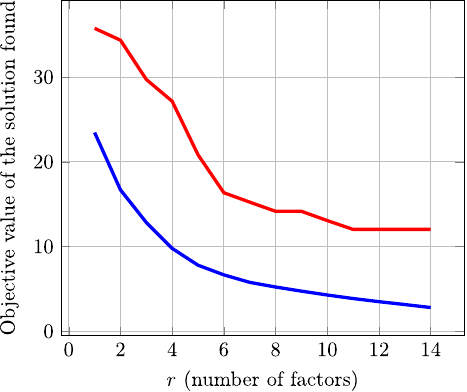}}\hfil
\subfloat[\texttt{neo} objective value]{\includegraphics[width=3.5cm]{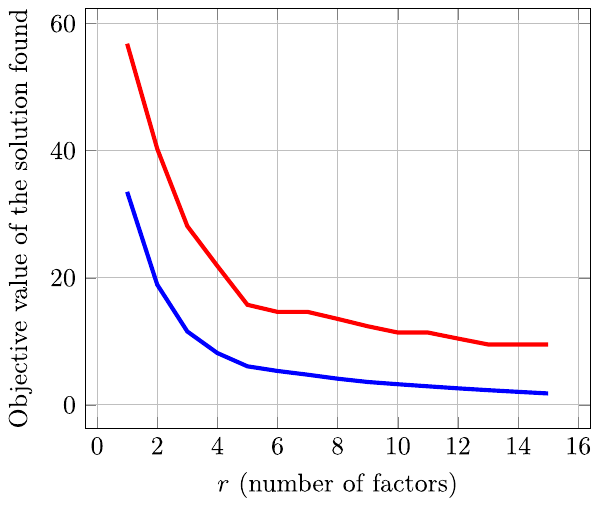}}\hfil 
\subfloat[\texttt{harman} objective value]{\includegraphics[width=3.5cm]{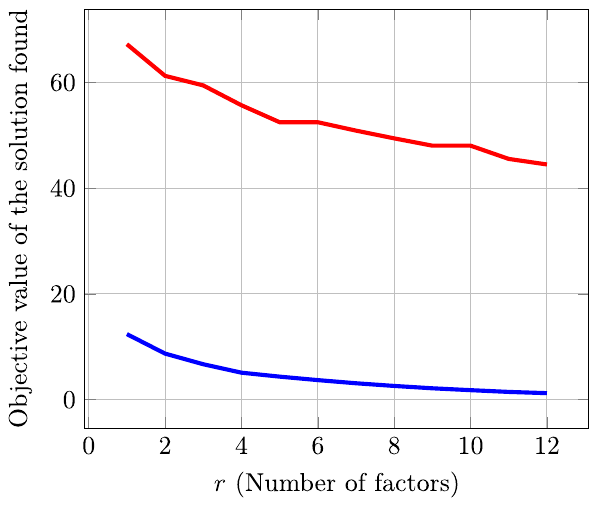}} 

\subfloat[\texttt{bfi} explained variance]{\includegraphics[width=3.5cm]{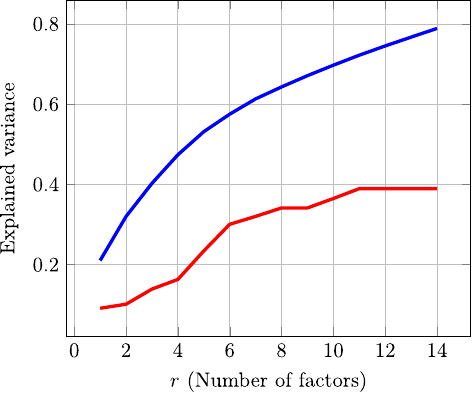}}\hfil   
\subfloat[\texttt{neo} explained variance]{\includegraphics[width=3.5cm]{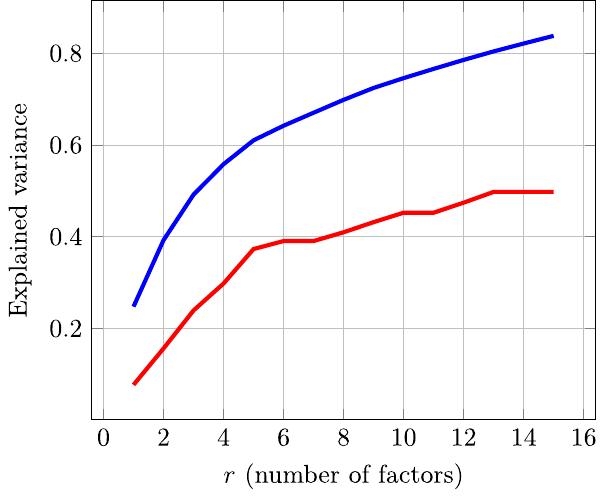}}\hfil
\subfloat[\texttt{harman} explained variance]{\includegraphics[width=3.5cm]{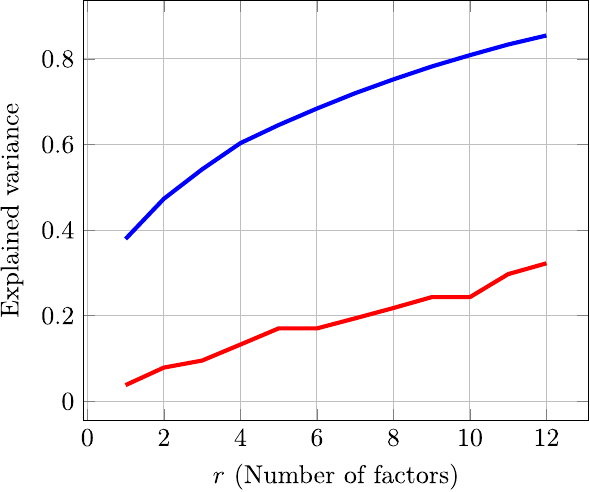}}
\caption{Figure showing performance of \textsf{NExOS} in solving factor analysis
problem for different datasets. Each column represents one dataset. The
first and second row compares training loss and proportion of the
variance explained of the solutions found by \textsf{NExOS} (shown in \textcolor{blue}{blue}) and the
nuclear norm heuristic (shown in \textcolor{red}{red}). \label{fig:Figure-showing-performance}}
\end{figure}

\paragraph{Performance measures}

We consider two performance measures. First, we compare the training
loss $\left\Vert \Sigma-X-D\right\Vert _{F}^{2}$ of the solutions
found by \textsf{NExOS} and the nuclear norm heuristic. As both \textsf{NExOS}
and the nuclear norm heuristic provide a point from the feasible set
of (\ref{eq:factor_analysis}), such a comparison of training losses
tells us which algorithm is providing a better quality solution. Second, we compute the\emph{ proportion of explained variance}, which
represents how well the $r$-common factors explain the residual covariance,\emph{
i.e.}, $\Sigma-D$. For a given $r$, input proportion of variance
explained by the $r$ common factors is given by: $\sum_{i=1}^{r}\sigma_{i}(X)/\sum_{i=1}^{p}\sigma_{i}(\Sigma-D),$
where $X,D$ are inputs, that correspond to solutions found by \textsf{NExOS}
or the nuclear norm heuristic. As $r$ increases, the explained variance
increases to $1$. The higher the value of the explained variance
for a certain solution, the better is the quality of the solution.

\paragraph{Description of the datasets}

We consider three different real-world bench-mark datasets that are
popularly used for factor analysis. 
The \texttt{bfi}, \texttt{neo }, and  \texttt{Harman74} datasets contain (2800 observations, 28 variables), (1000 observations, 30 variables), and (145 observations, 24 variables), respectively.

\paragraph{Setup}

In applying \textsf{NExOS} for the factor analysis problem, we initialize
our iterates with $Z_{0}:=\Sigma$ and $z_{0}:=\mathbf{0}$. All the
other parameters are kept at their default values as stated in the
beginning of \S\ref{sec:Numerical-experiments}. For each dataset,
we vary the number of factors from $1$ to $\lfloor p/2\rfloor$,
where $p$ is the size of the underlying matrix $\Sigma$.

\paragraph{Results}

Figure \ref{fig:Figure-showing-performance} shows performance of
\textsf{NExOS} in solving the factor analysis problem for different
datasets, with each row representing one dataset. The first row compares the training loss of the solution found
by \textsf{NExOS} and the nuclear norm heuristic. We see that for
all the datasets, \textsf{NExOS} finds a solution with a training
loss that is considerably smaller than that of the nuclear norm heuristic. The second row shows the proportion of variance explained by the
algorithms considered for the datasets considered (higher is better).
We see that in terms of the proportion of explained variance, \textsf{NExOS}
delivers larger values than that of the nuclear norm heuristic for
different values of $r$, which is indeed desirable. \textsf{NExOS}
consistently provides solutions with better objective value and explained
variance compared to the nuclear norm heuristic.

\section{Conclusion\label{sec:Conclusion}}
In this paper, we have presented \textsf{NExOS}, a first-order
algorithm to solve optimization problems with convex cost functions
over nonconvex constraint sets--- a problem structure that is satisfied
by a wide range of nonconvex optimization problems including sparse
and low-rank optimization. We have shown that, under mild technical
conditions, \textsf{NExOS} is able to find a locally optimal point
of the original problem by solving a sequence of penalized problems
with strictly decreasing penalty parameters. We have implemented our
algorithm in the \texttt{Julia} package \texttt{NExOS.jl} and have
extensively tested its performance on a wide variety of nonconvex
optimization problems. We have demonstrated that \textsf{NExOS} is able
to compute high quality solutions at a speed that is competitive with
tailored algorithms.

\bibliographystyle{plain}
\bibliography{Manuscript}

\appendix

\section{Proof and derivation to results in $\S$\ref{sec:Introduction}\label{sec:Proofs-to-results-in-Intro}}

\subsection{Lemma regarding prox-regularity of intersection of sets}
\begin{lem}\label{Lem:Prox-regularity-of-X}Consider the nonempty constraint
set $\mathcal{X}=\mathcal{\mathcal{C}\bigcap\mathcal{N}}\subseteq\eu$,
where $\mathcal{C}$ is compact and convex, and $\mathcal{N}$ is
prox-regular at $x\in\mathcal{X}$. Then $\mathcal{X}$ is prox-regular
at $x$.
\end{lem}

\paragraph{Proof to Lemma \ref{Lem:Prox-regularity-of-X}\label{subsec:Proof-to-rank-sparse-reg}}

To prove this result we record the following result from \cite{Bernard2011},
where by $d_{\mathcal{S}}(x)$ we denote the Euclidean distance of
a point $x$ from the set $\mathcal{S}$, and $\overline{\mathcal{S}}$
denotes closure of a set $\mathcal{X}$. 
\begin{lem}[{Intersection of prox-regular sets \cite[Corollary 7.3(a)]{Bernard2011}\label{Lem_Intersection-of-prox-regular}}]
 Let $\mathcal{S}_{1},\mathcal{S}_{2}$ be two closed sets in $\eu,$
such that $\mathcal{S}=\mathcal{S}_{1}\bigcap\mathcal{S}_{2}\neq\emptyset$
and both $\mathcal{S}_{1},\mathcal{S}_{2}$ are prox-regular at $x\in\mathcal{S}$.
If $\mathcal{S}$ is metrically calm at $x$, i.e., if there exist
some $\varsigma>0$ and some neighborhood of $x$ denoted by $\mathcal{B}$
such that $d_{\mathcal{S}}(y)\leq\varsigma(d_{\mathcal{S}_{1}}(y)+d_{\mathcal{S}_{2}}(y))$
for all $y\in\mathcal{B}$, then $\mathcal{S}$ is prox-regular at $x$.
\end{lem}
\begin{proof}
(proof to Lemma \ref{Lem:Prox-regularity-of-X}) By definition, projection
onto $\mathcal{N}$ is single-valued on some open ball $B(x;a)$ with
center $x$ and radius $a>0$ \cite[Theorem 1.3]{PoliRocka2000}.
The set $\mathcal{C}$ is compact and convex, hence projection onto
$\mathcal{C}$ is single-valued around every point, hence single-valued
on $B(x;a)$ as well \cite[Theorem 3.14, Remark 3.15]{Bauschke2017}.
Note that for any $y\in B(x;a)$, $d_{\mathcal{X}}(y)=0$ if and only
if both $d_{\mathcal{C}}(y)$ and $d_{\mathcal{N}}(y)$ are zero.
Hence, for any $y\in B(x;a)\bigcap\mathcal{X}$, the metrically calmness
condition is trivially satisfied.
Next, recalling that the distance from a closed set is continuous
\cite[Example 9.6]{Rockafellar2009}, over the compact set $\overline{B(x;a)\setminus\mathcal{X}}$,
define the function $h$, such that $h(y)=1$ if $y\in\mathcal{X}$, and $h(y)=d_{\mathcal{X}}(y)/(d_{\mathcal{C}}(y)+d_{\mathcal{N}}(y))$ else.
The function $h$ is upper-semicontinuous over $\overline{B(x;a)\setminus\mathcal{X}}$,
hence it will attain a maximum $\varsigma>0$ over $\overline{B(x;a)\setminus\mathcal{X}}$
\cite[Theorem 4.16]{rudin}, thus satisfying the metrically calmness
condition on $B(x;a)\setminus\mathcal{X}$
as well. Hence, using Lemma \ref{Lem_Intersection-of-prox-regular},
the constraint set $\mathcal{X}$ is prox-regular at $x$.
\end{proof}

\section{Proofs and derivations to the results in $\S$\ref{sec:Convergence-analysis}}

\subsection{Modifying \textsf{NExOS }for nonsmooth and convex loss function\label{subsec:nonsmooth_NExOS}}

We now discuss how to modify \textsf{NExOS} when the loss function
is nonsmooth and convex. The key idea is working with  a strongly convex, smooth, and arbitrarily close approximation of $f$; such smoothing techniques are very common in optimization \cite{nesterov2005smooth,beck2017first}. The optimization problem in this case, where
the positive regularization parameter is denoted by $\widetilde{\beta}$,
is given by: $\min_{x}\phi(x)+(\widetilde{\beta}/2)\|x\|^{2}+\iota_\mathcal{X}(x)$,
where the setup is same as problem (\ref{eq:original_problem-1}), except
the function $\phi:\eu\to\mathbf{R}\cup\left\{ +\infty\right\} $
is lower-semicontinuous, proper (its domain is nonempty), and convex. Let $\beta\coloneqq\widetilde{\beta}/2$.
For a $\nu$ that is arbitrarily small, define the following $\beta$
strongly convex and $(\nu^{-1}+\beta)$-smooth function: $f  \coloneqq\prescript{\nu}{}{\phi}(\cdot)+(\beta/2)\|\cdot\|^{2}$
where $\prescript{\nu}{}{\phi}$ is the Moreau envelope of $\phi$ with paramter $\nu$. Following the properties of the Moreau envelope of a convex function discussed in \S\ref{subsec:our_approach}, the following optimization problem acts as an arbitrarily close
approximation to the first nonsmooth convex problem: $\min_{x}f+(\beta/2)\|x\|^{2}+\iota_\mathcal{X}(x)$, which has the same setup as problem (\ref{eq:original_problem-1}).

We can compute $\prox_{\gamma f}(x)$ using the formula in
by \cite[Theorem 6.13, Theorem 6.63]{beck2017first}.
Then, we apply \textsf{NExOS }to $\min_{x}f+(\beta/2)\|x\|^{2}+\iota_\mathcal{X}(x)$ and proceed
in the same manner as discussed earlier. 

\subsection{Proof to Proposition \ref{Attainment-of-local-min}\label{subsec:Proof-to-Lemma-Nov-4}}

\subsubsection{Proof to Proposition \ref{Attainment-of-local-min}(i)}

We prove (i) in three steps. In the\emph{ first step}, we show that
for any $\mu>0$, $f+\regind$ will be differentiable on some $B(\bar{x};r_{\textrm{diff}})$
with $r_{\textrm{diff}}>0$. In the \emph{second step}, we then show
that, for any $\mu\in(0,1/\beta],$ $f+\regind$ will be strongly
convex and differentiable on some $B(\bar{x};r_{\textrm{cvxdiff}})$.
In the \emph{third step}, we will show that there exist $\mu_{\textrm{max}}>0$
such that for any $\mu\in(0,\mu_{\textrm{max}}]$, $\regf+\regind$
will be strongly convex and smooth on some $B(\bar{x};r_{\textrm{max}})$
and will attain the unique local minimum $x_{\mu}$ in this ball.

\paragraph{Proof of the first step}

To prove the first step, we start with the following lemma regarding
differentiability of $\morind$.
\begin{lem}[Differentiability of $\morind$\label{Differentiability-of-morind}]
 Let $\bar{x}$ be a local minimum to problem (\ref{eq:original_problem-1}),
where Assumptions \ref{assum:strong_convex_smooth} and \ref{assum:unique_strongly_convex}
hold. Then there exists some $r_{\textup{diff}}>0$ such that for
any $\mu>0$: (i) the function $\morind$ is differentiable on $B(\bar{x};r_{\textup{diff}})$
with derivative $\nabla\morind=(1/\mu)(\id-\proj),$ and (ii) the projection operator $\proj$ onto $\mathcal{X}$ is single-valued
and Lipschitz continuous on $B(\bar{x};r_{\textup{diff}})$.
\end{lem}
\begin{proof}
From \cite[Theorem 1.3(e)]{PoliRocka2000}, there exists some $r_{\textrm{diff}}>0$
such that the function $d^{2}$ is differentiable on $B(\bar{x};r_{\textrm{diff}})$.
As $\morind=(1/2\mu)d^{2}$ from (\ref{eq:samaja}), it follows that
for any $\mu>0,$ $\morind$ is differentiable on $B(\bar{x};r_{\textrm{diff}})$
which proves the first part of (i). The second part of (i) follows
from the fact that $\nabla d^{2}(x)=2\left(x-\proj(x)\right)$ whenever
$d^{2}$ is differentiable at $x$ \cite[page 5240]{PoliRocka2000}.
Finally, from \cite[Lemma 3.2]{PoliRocka2000}, whenever $d^{2}$
is differentiable at a point, projection $\proj$ is single-valued
and Lipschitz continuous around that point, and this proves (ii).
\end{proof}
Due to the lemma above, $f+\regind$ will be differentiable on $B(\bar{x};r_{\textrm{diff}})$
with $r_{\textrm{diff}}>0$, as $f$ and $(\beta/2)\|\cdot\|^{2}$
are differentiable. Also, due to Lemma \ref{Differentiability-of-morind}(ii),
projection operator $\proj$ is $\widetilde{L}$-Lipschitz continuous
on $B(\bar{x};r_{\textrm{diff}})$ for some $\widetilde{L}>0$. This proves the first step. 

\paragraph{Proof of the second step}

To prove this step, we are going to record: (1) the notion of general
subdifferential of a function, followed by (2) the definition of prox-regularity
of a function and its connection with prox-regular set, and (3) a
helper lemma regarding convexity of the Moreau envelope under prox-regularity. 
\begin{defn}[Fenchel, Fr\'echet, and general subdifferential]
For any lower-semicontinuous function $h:\rl^{n}\to\rl\cup\{\infty\}$,
its Fenchel subdifferential $\partial h$ is defined as \cite[page 1]{correa1992characterization}: $u\in\partial h(x)\Leftrightarrow h(y)\geq h(x)+\left\langle u\mid y-x\right\rangle$ for all $y \in \rl^{n}$.
For the function $h$, its Fr\'echet subdifferential $\partial^{F}h$
(also known as regular subdifferential) at a point $x$ is defined
as \cite[Definition 2.5]{correa1992characterization}: $u\in\partial^{F}h(x)\Leftrightarrow\liminf_{y\to0}\,(h(x+y)-h(x)-\langle u\mid y\rangle)/\|y\|\geq0$.
Finally, the general subdifferential of $h$, denoted by $\partial^{G}h$,
is defined as \cite[Equation (2.8)]{rockafellar2020characterizing}: $u\in\partial^{G}h(x)\Leftrightarrow u_{n}\to u,x_{n}\to x,f(x_{n})\to f(x),$ for some $(x_{n},u_{n})\in\gra\partial^{F}h$. If $h$ is additionally convex, then $\partial h=\partial^{F}h=\partial^{G}h$ \cite[Property (2.3), Property 2.6]{correa1992characterization}.
\end{defn}

\begin{defn}[{Connection between prox-regularity of a function and a set \cite[Definition 1.1 ]{poliquin1996prox}}]
\label{def:prox-reg-def}A function $h:\mathbf{R}^{n}\to\mathbf{R}\cup\{\infty\}$
that is finite at $\tilde{x}$ is prox-regular at $\tilde{x}$ for
$\tilde{\nu}$, where $\tilde{\nu}\in\partial^{G}h(\tilde{x})$, if
$h$ is locally l.s.c. at $\tilde{x}$ and there exist a distance
$\sigma>0$ and a parameter $\rho>0$ such that whenever $\|x'-\tilde{x}\|<\sigma$
and $\|x-\tilde{x}\|<\sigma$ with $x'\neq x$, $\|h(x)-h(\tilde{x})\|<\sigma$,
$\|\nu-\tilde{\nu}\|<\sigma$ with $\nu\in\partial^{G}h(x)$, we have
$h(x')>h(x)+\left\langle \nu\mid x'-x\right\rangle -(\rho/2)\|x'-x\|^{2}$. Also, a set $\mathcal{S}$ is prox-regular at $\tilde{x}$ for $\tilde{\nu}$
if we have the indicator function $\iota_{\mathcal{S}}$ is prox-regular
at $\tilde{x}$ for $\tilde{\nu}\in\partial^{G}\iota_{\mathcal{S}}(\tilde{x})$
\cite[Proposition 2.11]{poliquin1996prox}. The set $\mathcal{S}$
is prox-regular at $\tilde{x}$ if it is prox-regular at $\tilde{x}$
for all $\tilde{\nu}\in\partial^{G}\iota_{\mathcal{S}}(\tilde{x})$
\cite[page 612]{Rockafellar2009}. 
\end{defn}
We have the following helper lemma from \cite{poliquin1996prox}.
\begin{lem}[{\cite[Theorem 5.2]{poliquin1996prox}}]
\label{thm:Consider-a-function}Consider a function $h$ which is
lower semicontinuous at $0$ with $h(0)=0$ and there exists $\rho>0$
such that $h(x)>-(\rho/2)\|x\|^{2}$ for any $x\neq0$. Let
$h$ be prox-regular at $\tilde{x}=0$ and $\tilde{\nu}=0$ with respect
to $\sigma$ and $\rho$ ($\sigma$ and $\rho$ as described in Definition \ref{def:prox-reg-def}), and let $\lambda\in(0,1/\rho).$ Then, on
some neighborhood of $0$, the function
\begin{equation}
^{\lambda}h+\rho/(2-2\lambda\rho)\|\cdot\|^{2} \label{eq:moreau-fun-new}
\end{equation}
is convex, where $^{\lambda}h$ is the Moreau envelope of $h$ with
parameter $\lambda$. 
\end{lem}
Now we start proving step 2 earnestly. To prove this result, we assume
$\bar{x}=0$. This does not cause any loss of generality because this
is equivalent to transferring the coordinate origin to the optimal
solution and prox-regularity of a set and strong convexity of a function
is invariant under such a coordinate transformation. 

First, note that the indicator function of our constraint closed set
$\mathcal{X}$ is lower semicontinuous due to \cite[Remark after Theorem 1.6, page 11]{Rockafellar2009},
and as $\bar{x},$ the local minimizer lies in $\mathcal{X},$ we
have $\iota_\mathcal{X}(\bar{x})=0$. The set $\mathcal{X}$ is prox-regular at
$\bar{x}$ for all $\nu\in\partial^{G}\iota_\mathcal{X}(x)$ per our setup, so
using Definition \ref{def:prox-reg-def}, we have $\iota_\mathcal{X}$ prox-regular
at $\bar{x}=0$ for $\bar{\nu}=0\in\partial^{G}\iota_\mathcal{X}(\bar{x})$ (because
$\bar{x}\in\mathcal{X}$, we will have $0$ as a subgradient of $\partial\iota_\mathcal{X}(\bar{x})$) with respect to some distance $\sigma>0$
and parameter $\rho>0$. 

Note that the indicator function satisfies $\iota_\mathcal{X}(x)=c\iota_\mathcal{X}(x)$ for
any $c>0$ due to its definition, so $u\in\partial^{G}\iota_\mathcal{X}(x)\Leftrightarrow cu\in c\partial^{G}\iota_\mathcal{X}(x)=\partial(c\iota_\mathcal{X}^{G}(x))=\partial\iota_\mathcal{X}^{G}(x)$
\cite[Equation 10(6)]{Rockafellar2009} In our setup, we have $\mathcal{X}$
prox-regular at $\bar{x}$. So, setting $h\coloneqq\iota_mathcal{X},\tilde{x}\coloneqq\bar{x}=0,$
$\tilde{\nu}\coloneqq\bar{\nu}=0$, and $\nu\coloneqq u/(\beta/2\rho)$
in Definition \ref{def:prox-reg-def}, we have $\iota_\mathcal{X}$ is also prox-regular
at $\bar{x}=0$ for $\bar{\nu}=0$ with respect to distance $\sigma\min\{1,\beta/2\rho\}$
and parameter $\beta/2$.

Next, because the range of the indicator function is $\{0,\infty\}$,
we have 
$\iota_\mathcal{X}(x)>-(\rho/2)\|x\|^{2}$
for any $x\neq0$. So, we have all the conditions of Theorem \ref{thm:Consider-a-function}
satisfied. Hence, applying Lemma \ref{thm:Consider-a-function},
we have $(1/2\mu)\left(d^{2}+\beta\mu/(2-\beta\mu)\|\cdot\|^{2}\right)$
convex and differentiable on 
\[B\left(\bar{x};\min\left\{ \sigma\min\{1,\beta/2\rho\},r_{\textrm{diff}}\right\} \right)\]
for any $\mu\in(0,2/\beta)$,
where $r_{\textrm{diff}}$ comes from Lemma \ref{Differentiability-of-morind}.
As $r_{\textrm{diff}}$ in this setup does not depend on $\mu$, the
ball %
does not depend on $\mu$ either. Finally, note that in our exterior-point
minimization function we have $\regind=(1/2\mu)\left(d^{2}+\beta\mu\|\cdot\|^{2}\right)$.

So if we take $\mu\leq\frac{1}{\beta}$, then we have $(\beta/2)\mu/\left(1-\mu(\beta/2)\right)\leq\beta\mu$,
and on the ball $B\left(\bar{x};\min\left\{ \sigma\min\{1,\beta/2\rho\},r_{\textrm{diff}}\right\} \right)$,
the function $\regind$ will be convex and differentiable. But $f$
is strongly-convex and smooth, so $f+\regind$ will be strongly convex
and differentiable on $B\left(\bar{x};\min\left\{ \sigma\min\{1,\beta/2\rho\},r_{\textrm{diff}}\right\} \right)$
for $\mu\in(0,1/\beta]$. This proves step 2. 

\paragraph{Proof of the third step}

As point $\bar{x}\in\mathcal{X}$ is a local minimum of problem (\ref{eq:original_problem-1}),
from Definition \ref{Def:Local-minima-of-P}, there is some $r>0$
such that for all $y\in\overline{B}(\bar{x};r),$ we have $f(\bar{x})+(\beta/2)\|\bar{x}\|^{2}<f(y)+(\beta/2)\|y\|^{2}+\ind(y).$

Then, due to the first two steps, for any $\mu\in(0,1/\beta]$, the
function $\regf+\regind$ will be strongly convex and differentiable
on $B\left(\bar{x};\min\left\{ \sigma\min\{1,\beta/2\rho\},r_{\textrm{diff}}\right\} \right)$.
For notational convenience, denote $r_{\textrm{max}}\coloneqq\min\left\{ \sigma\min\{1,\beta/2\rho\},r_{\textrm{diff}}\right\},$
which is a constant. As $\regf+\regind$ is a global underestimator
of and approximates the function $f+(\beta/2)\|\cdot\|^{2}+\ind$
with arbitrary precision as $\mu\to0$, the previous statement and
\cite[Theorem 1.25]{Rockafellar2009} imply that there exist some
$0<\mu_{\textrm{max}}\leq1/\beta$ such that for any $\mu\in(0,\mu_{\textrm{max}}],$
the function $\regf+\regind$ will achieve a local minimum $x_{\mu}$
over $B(\bar{x};r_{\textrm{max}})$ where $\nabla(\regf+\regind)$
vanishes, \emph{i.e.},
\begin{align}
\nabla(\regf+\regind)(x_{\mu}) & =\text{\ensuremath{\nabla}}f(x_{\mu})+\beta x_{\mu}+ (1/\mu)\left(x_{\mu}-\proj\left(x_{\mu}\right)\right)=0\label{eq:optimality condition}\\
\Rightarrow x_{\mu} & =(1/ (\beta \mu + 1)) \left(\proj(x_{\mu})-\mu\nabla f(x_{\mu})\right).\label{eq:x_mu_compact}
\end{align}
As the right hand side of the last equation is a singleton, this minimum
must be unique. Finally to show the smoothness $f+\regind$, for any
$x\in B(\bar{x};r_{\textrm{max}}),$ we have 

\begin{align}
\nabla\left(f+\regind\right)(x) \overset{a)}{=}\nabla f(x)+\left(\beta+ (1/\mu)\right)x- (1/\mu)\proj(x),\label{eq:differentiable_formula}
\end{align}
where $a)$ uses Lemma \ref{Differentiability-of-morind}. Thus, for
any $x_{1},x_{2}\in B(\bar{x};r_{\textrm{max}})$ we have
$\|\nabla(f+(\beta/2)\|\cdot\|^{2}+\morind)(x_{1})-\nabla(f+(\beta/2)\|\cdot\|^{2}+\morind)(x_{2})\|\leq(L+\beta+(1/\mu)+\widetilde{L})\|x_{1}-x_{2}\|$, where we have used the following: $\nabla f$ is $L$-Lipschitz
everywhere due to $f$ being an $L-$smooth function in $\eu$ (\cite[Theorem 18.15]{Bauschke2017}),
and $\proj$ is $\widetilde{L}$-Lipschitz continuous on $B(\bar{x};r_{\textrm{max}})$,
as shown in step 1. This completes the proof for (i).

(ii): Using \cite[Theorem 1.25]{Rockafellar2009}, as $\mu\to0,$
we have $x_{\mu}\to\bar{x},\textrm{ and }\left(\regf+\regind\right)(x_{\mu})\to f(\bar{x})+ (\beta/2)\|\bar{x}\|^{2}.$
Note that $x_{\mu}$ reaches $\bar{x}$ only in limit, as otherwise
Assumption \ref{assum:unique_strongly_convex} will be violated. 

\subsection{Proof to Proposition \ref{Operators-contractive}\label{subsec:Proof-to-Proposition-operators-contractive}}

\subsubsection{Proof to Proposition \ref{Operators-contractive}(i)}
{
We will need the notions of nonexpansive and firmly nonexpansive operators in this proof. An operator $\opA:\eu\to\eu$ is nonexpansive on some set $\mathcal{S}$
if it is Lipschitz continuous with Lipschitz constant $1$ on $\mathcal{S}$; the operator is contractive if the Lipschitz constant is strictly smaller than $1$.
On the other hand, $\opA$ is firmly nonexpansive on $\mathcal{S}$
if and only if its reflection operator $2\opA-\id$ is nonexpansive
on $\mathcal{S}$. A firmly nonexpansive operator is always nonexpansive
\cite[page 59]{Bauschke2017}.}

We next introduce the following definition.

\begin{defn}[{Resolvent and reflected resolvent \label{Resolvent-and-reflected-resolvent}\cite[pages 333, 336]{Bauschke2017}}]
 For a lower-semicontinuous, proper, and convex function $h,$ the
resolvent and reflected resolvent of its subdifferential operator
are defined by $\opJ_{\gamma\partial h}=\left(\id+\gamma\partial h\right)^{-1}$
and $\opR_{\gamma\partial h}=2\opJ_{\gamma\partial h}-\id,$ respectively.
\end{defn}
The proof of (i) is proven in two steps. First, we show that the reflection
operator of $\opT_{\mu},$ defined by
\begin{equation}
\opR_{\mu}=2\opT_{\mu}-\id,\label{eq:DRS-refelction-operator}
\end{equation}
is contractive on $B(\bar{x},r_{\textrm{max}})$, and using this we
show that $\opT_{\mu}$ in also contractive there in the second step.
To that goal, note that $\opR_{\mu}$ can be represented as:
\begin{equation}
\opR_{\mu}=(2\prox_{\gamma\regind}-\id)(2\prox_{\gamma\regf}-\id), \label{eq:alternative-representation-reflection}
\end{equation}
which can be proven by simply using (\ref{eq:NCDRS-operator}) and
(\ref{eq:DRS-refelction-operator}) on the left-hand side and by expanding
the factors on the right-hand side. Now, the operator $2\mathbf{prox}_{\gamma f}-\id$
associated with the $\alpha$-strongly convex and $L$-smooth function
$f$ is a contraction mapping for any $\gamma>0$ with the contraction
factor $\kappa=\max\left\{ (\gamma L-1)/(\gamma L+1),(1-\gamma\alpha)/(\gamma\alpha+1)\right\} \in(0,1)$,
which follows from \cite[Theorem 1]{Giselsson17}. Next, we show that $2\prox_{\gamma\regind}-\id$ is nonexpansive on
$B(\bar{x};r_{\textup{max}})$ for any $\mu\in(0,\mu_{\textrm{max}}]$.
For any $\mu\in(0,\mu_{\textrm{max}}]$, define the function $g$
as follows. We have $g(y)=\regind(y)$ if $y\in B(\bar{x};r_{\textup{max}}),$$g(y)=\liminf_{\tilde{y}\to y}\regind(\tilde{y})$
if $\|y-\bar{x}\|=r_{\textrm{max}}$, and $g(y)=\infty$ else.
The function $g$ is lower-semicontinuous, proper, and convex everywhere due to
\cite[Lemma 1.31 and Corollary 9.10 ]{Bauschke2017}. As a result
for $\mu\in(0,\mu_{\textrm{max}}]$, we have $\mathbf{prox}_{\gamma g}=\opJ_{\gamma\partial g}$
on $\mathbf{E}$ and $\mathbf{prox}_{\gamma g}$ is firmly nonexpansive
and single-valued everywhere, which follows from \cite[Proposition 12.27, Proposition 16.34, and Example 23.3]{Bauschke2017}.
But, for $y\in B(\bar{x};r_{\textup{max}}),$ we have $\regind(y)=g(y)$
and $\nabla\regind(y)=\partial g(y)$. Thus, on $B(\bar{x};r_{\textup{max}}),$
the operator $\prox_{\gamma\regind}=\opJ_{\gamma\nabla\regind}$,
and it is firmly nonexpansive and single-valued for $\mu\in(0,\mu_{\textrm{max}}]$.
Any firmly nonexpansive operator $\opA$ has a nonexpansive reflection
operator $2\opA-\id$ on its domain of firm nonexpansiveness \cite[Proposition 4.2]{Bauschke2017}.
Hence, on $B(\bar{x};r_{\textup{max}}),$ for $\mu\in(0,\mu_{\textrm{max}}]$
the operator $2\prox_{\gamma\regind}-\id$ is nonexpansive using (\ref{eq:alternative-representation-reflection}). 

Now we show that $\opR_{\mu}$ is contractive for every $x_{1},x_{2}\in B(\bar{x};r_{\textup{max}})$
and $\mu\in(0,\mu_{\textrm{max}}]$, we have $\|\opR_{\mu}(x_{1})-\opR_{\mu}(x_{2})\|\leq\|(2\prox_{\gamma\regf}-\id)(x_{1})-(2\prox_{\gamma\regf}-\id)(x_{2})\|\leq\text{\ensuremath{\kappa}}\|x_{1}-x_{2}\|$
where the last inequality uses $\kappa$-contractiveness of $2\prox_{\gamma\regf}-\id$
thus proving that $\opR_{\mu}$ acts as a contractive operator on
$B(\bar{x};r_{\textup{max}})$ for $\mu\in(0,\mu_{\textrm{max}}]$.
Similarly, for any $x_{1},x_{2}\in B(\bar{x};r_{\textup{max}}),$
using (\ref{eq:convenient-form-1}) and the triangle inequality we
have $\|\opT_{\mu}(x_{1})-\opT_{\mu}(x_{2})\|\leq(1+\kappa)/2\|x_{1}-x_{2}\|$
and as $\kappa'=(1+\kappa)/2\in[0,1);$ the operator $\opT_{\mu}$
is $\kappa'-$contractive on on $B(\bar{x};r_{\textup{max}}),$ for
$\mu\in(0,\mu_{\textrm{max}}]$.

\subsubsection{Proof to Proposition \ref{Operators-contractive}(ii)}

Recalling $\opT_{\mu}=(1/2)\opR_{\mu}+(1/2)\id$ from (\ref{eq:DRS-refelction-operator}),
using (\ref{eq:alternative-representation-reflection}), and then
expanding, and finally using Lemma \ref{Lem:computing_complicated_projection}
and triangle inequality, we have for any $\mu,\tilde{\mu}\in(0,\mu_{\textrm{max}}],\;x\in B(\bar{x};r_{\textup{max}}),$ and $y=2\prox_{\gamma\regf}(x)-x$:
\begin{align}\left\Vert \opT_{\mu}(x)-\opT_{\tilde{\mu}}(x)\right\Vert  & \leq\left\Vert \left(\mu/(\gamma+\mu(\beta\gamma+1))-\tilde{\mu}/(\gamma+\tilde{\mu}(\beta\gamma+1))\right)\right\Vert \left\Vert y\right\Vert \nonumber\\
 & +\left\Vert \left(\gamma/(\gamma+\mu(\beta\gamma+1))-\gamma/(\gamma+\tilde{\mu}(\beta\gamma+1))\right)\right\Vert \left\Vert \proj\left(y/(\beta\gamma+1)\right)\right\Vert.\label{eq:inequ_lipschit}
\end{align}

 Now, in (\ref{eq:inequ_lipschit}), the coefficient of $\|y\|$ satisfies $\|\mu/(\gamma+\mu(\beta\gamma+1))-\tilde{\mu}/(\gamma+\tilde{\mu}(\beta\gamma+1))\|\leq(1/\gamma)\|\mu-\tilde{\mu}\|$

and similarly the coefficient of $\|\proj\left(y/(\beta\gamma+1)\right)\|$
satisfies

\begin{align*}\|\gamma/(\gamma+\mu(\beta\gamma+1))-\gamma/(\gamma+\tilde{\mu}(\beta\gamma+1))\| & \leq(\beta+(1/\gamma))\|\mu-\tilde{\mu}\|.\end{align*}
Putting the last two inequalities in (\ref{eq:inequ_lipschit}), and
then replacing $y=2\prox_{\gamma\regf}(x)-x$, we have for any
$x\in\mathcal{B}$, and for any $\mu,\tilde{\mu}\in\rl_{++}$, 
\begin{align}
& \left\Vert \opT_{\mu}(x)-\opT_{\tilde{\mu}}(x)\right\Vert \leq(1/\gamma)\left\Vert \mu-\tilde{\mu}\right\Vert \|y\|+\left(\beta+(1/\gamma)\right)\|\mu-\tilde{\mu}\|\left\Vert \proj\left(y/(\beta\gamma+1)\right)\right\Vert \nonumber\\
= & \{(1/\gamma)\|2\prox_{\gamma\regf}(x)-x\|+(\beta+(1/\gamma))\|\proj((2\prox_{\gamma\regf}(x)-x)/(\beta\gamma+1))\|\}\|\mu-\tilde{\mu}\|.
\label{eq:equal_bound_T_mu}
\end{align}

Now, as $B(\bar{x};r_{\textup{max}})$ is a bounded set and $x\in\mathcal{B}$,
norm of the vector $y=2\prox_{\gamma\regf}(x)-x$ can be upper-bounded
over $B(\bar{x};r_{\textup{max}})$ because $2\prox_{\gamma f}-\id$
is continuous (in fact contractive) as shown in (i). Similarly, $\|\proj\left((2\prox_{\gamma\regf}(x)-x)/(\beta\gamma+1)\right)\|$
can be upper-bounded on $B(\bar{x};r_{\textup{max}})$. Combining the
last two-statements, it follows that there exists some $\ell>0$ such
that 
\[
\sup_{x\in B(\bar{x};r_{\textup{max}})}(1/\gamma)\|2\prox_{\gamma\regf}(x)-x\|+(\beta+1/\gamma)\left\Vert \proj\left((2\prox_{\gamma\regf}(x)-x)/(\beta\gamma+1)\right)\right\Vert \leq\ell,
\]
 and putting the last inequality in (\ref{eq:equal_bound_T_mu}),
we arrive at the claim.

\subsection{Proof to Proposition \ref{thm:fixed_point_relationship}\label{subsec:fixed_point_lemma}}

The structure of the proof follows that of \cite[Proposition 25.1(ii)]{Bauschke2017}.
Let $\mu\in(0,\mu_{\textrm{max}}].$ Recalling Definition \ref{Resolvent-and-reflected-resolvent},
and due to Proposition \ref{Attainment-of-local-min}(i), $x_{\mu}\in B(\bar{x};r_{\textrm{max}})$
satisfies 
\begin{flalign}
 & x_{\mu}=\argmin_{B(\bar{x};r_{\textup{max}})}f(x)+\regind(x)=\mathop{{\bf zer}}(\nabla\regf+\nabla\regind)\nonumber \\
\overset{a)}{\Leftrightarrow} & \left(\exists y\in\mathbf{E}\right)\;x_{\mu}=\opJ_{\gamma\nabla\regind}\opR_{\gamma\nabla\regf}(y)\textrm{ and }x_{\mu}=\opJ_{\gamma\nabla\regf}(y),\label{eq:shapiro-1}
\end{flalign}
where $a)$ uses the facts (shown in the proof to Proposition \ref{Operators-contractive})
that: (i) $\opJ_{\gamma\nabla\regf}$ is a single-valued operator
everywhere, whereas $\opJ_{\gamma\nabla\regind}$ is a single-valued
operator on the region of convexity $B(\bar{x};r_{\textrm{max}})$,
and (ii) $x_{\mu}=\opJ_{\gamma\nabla\regf}(y)$
can be expressed as $x_{\mu}=\opJ_{\gamma\nabla\regf}(y)\Leftrightarrow2x_{\mu}-y=\left(2\opJ_{\gamma\nabla\regf}-\id\right)y=\opR_{\gamma\nabla\regf}(y)$.
Also, using the last expression, we can write the first term of (\ref{eq:shapiro-1})
as $\opJ_{\gamma\nabla\regind}\opR_{\gamma\nabla\regf}(y)=x_{\mu}\Leftrightarrow y\in\mathop{{\bf fix}}(\opR_{\gamma\nabla\regind}\opR_{\gamma\nabla\regf})$.
Because for lower-semicontinuous,
proper, and convex function, the resolvent of the subdifferential
is equal to its proximal operator \cite[Proposition 12.27, Proposition 16.34, and Example 23.3]{Bauschke2017},
we have $\opJ_{\gamma\partial\regf}=\prox_{\gamma\regf}$ with both
being single-valued. Using the last fact along with (\ref{eq:shapiro-1}),
$y\in\mathop{{\bf fix}}(\opR_{\gamma\nabla\regind}\opR_{\gamma\nabla\regf})$, we have 
$x_{\mu}\in\prox_{\gamma\regf}\left(\mathop{{\bf fix}}\left(\opR_{\gamma\nabla\regind}\opR_{\gamma\partial\regf}\right)\right)$,
but $x_{\mu}$ is unique due to Proposition \ref{Attainment-of-local-min},
so the inclusion can be replaced with equality. Thus $x_{\mu}$, satisfies $x_{\mu}=\prox_{\gamma\regf}\left(\mathop{{\bf fix}}\left(\opR_{\gamma\nabla\regind}\opR_{\gamma\partial\regf}\right)\right)$
where the sets are singletons due to Proposition \ref{Attainment-of-local-min}
and single-valuedness of $\prox_{\gamma f}$. Also, because $\opT_{\mu}$
in (\ref{eq:NCDRS-operator}) and $\opR_{\mu}$ in (\ref{eq:DRS-refelction-operator})
have the same fixed point set (follows from (\ref{eq:DRS-refelction-operator})),
using (\ref{eq:alternative-representation-reflection}), we arrive at the claim.

\subsection{Proof to Lemma \ref{thm:convergence_outer_iterations}\label{subsec:Proof-to-outer-iterations}}

(i): This follows directly from the proof to Proposition \ref{Attainment-of-local-min}.

(ii): From Lemma \ref{thm:convergence_outer_iterations}(i), and recalling that $\eta'>1$,
for any $\mu\in(0,\mu_{\textup{max}}]$, we have the first equation. Recalling Definition \ref{Resolvent-and-reflected-resolvent}, and
using the fact that for lower-semicontinuous, proper, and convex function,
the resolvent of the subdifferential is equal to its proximal operator
\cite[Proposition 12.27, Proposition 16.34, and Example 23.3]{Bauschke2017},
we have $\opJ_{\gamma\partial\regf}=\prox_{\gamma\regf}$ with both
being single-valued. So, from Proposition \ref{thm:fixed_point_relationship}: $x_{\mu}=\prox_{\gamma\regf}(z_{\mu})=\left(\id+\gamma\partial f\right)^{-1}(z_{\mu})\Leftrightarrow z_{\mu}=x_{\mu}+\gamma\nabla\regf(x_{\mu})$.
Hence, for any $\mu\in(0,\mu_{\textup{max}}]$:
\begin{align*} & \|z_{\mu}-\bar{x}\|=\|x_{\mu}+\gamma\nabla\regf(x_{\mu})-\bar{x}\|\leq\|x_{\mu}-\bar{x}\|+\gamma\|\nabla f(x_{\mu})\|\\
\Leftrightarrow & r_{\textrm{max}}-\|z_{\mu}-\bar{x}\|\geq r_{\textrm{max}}-\|x_{\mu}-\bar{x}\|-\gamma\|\nabla f(x_{\mu})\|\overset{a)}{\geq}(\eta'-1)r_{\textrm{max}}/\eta'-\gamma\|\nabla f(x_{\mu})\|,
\end{align*}
where $a)$ uses the first equation of Lemma \ref{thm:convergence_outer_iterations}(ii). Because, for the
strongly convex and smooth function $\regf,$ its gradient is bounded
over a bounded set $B(\bar{x};r_{\textup{max}})$ \cite[Lemma 1, \S 1.4.2]{Polyak1987},
then for $\gamma$ satisfying the fourth equation of Lemma \ref{thm:convergence_outer_iterations}(ii) and the definition
of $\psi$ in the third equation of Lemma \ref{thm:convergence_outer_iterations}(ii), we have the second equation of Lemma \ref{thm:convergence_outer_iterations}(ii)
for any $\mu\in(0,\mu_{\textup{max}}].$ To prove the final equation of Lemma \ref{thm:convergence_outer_iterations}(ii),
note that 
\begin{align}
 & \lim_{\mu\to0}\left(r_{\textup{max}}-\|z_{\mu}-\bar{x}\|\right)-\psi\nonumber\\
 & \overset{a)}{=}\lim_{\mu\to0}\left(r_{\textup{max}}-\|x_{\mu}+\gamma\nabla\regf(x_{\mu})-\bar{x}\|\right)-(\eta'-1)r_{\textup{max}}/\eta'+\gamma\;\textup{max}_{x\in B(\bar{x};r_{\textup{max}})}\|\nabla f(x)\|\nonumber\\
 & \overset{b)}{=}\left(r_{\textup{max}}-\|\bar{x}+\gamma\nabla\regf(\bar{x})-\bar{x}\|\right)-(\eta'-1)r_{\textup{max}}/\eta'+\gamma\;\textup{max}_{x\in B(\bar{x};r_{\textup{max}})}\|\nabla f(x)\|\nonumber\\
 & =(1/\eta')r_{\textup{max}}+\gamma\left(\textup{max}_{x\in B(\bar{x};r_{\textup{max}})}\|\nabla f(x)\|-\|\nabla\regf(\bar{x})\|\right)>0,
\label{eq:ineq_5}
\end{align}
 where in $a)$ we have used $z_{\mu}=x_{\mu}+\gamma\nabla\regf(x_{\mu})$ and the third equation of Lemma \ref{thm:convergence_outer_iterations}(ii),
in $b)$ we have used smoothness of $f$ along with Proposition \ref{Attainment-of-local-min}(ii).
Inequality (\ref{eq:ineq_5}) along with the second equation of Lemma \ref{thm:convergence_outer_iterations}(ii)
implies the final equation of Lemma \ref{thm:convergence_outer_iterations}(ii). 

\subsection{Proof to Theorem \ref{Theorem_Main-convergence-result} \label{subsec:Proof-to-Theorem}}

We use the following result from \cite{dontchev2009implicit} in proving
Theorem \ref{Theorem_Main-convergence-result}.
\begin{thm}[{Convergence of local contraction mapping \cite[pp. 313-314]{dontchev2009implicit}}]
\label{Thm:local_contraction_mapping_principle} Let $\opA:\eu\to\eu$
be some operator. If there exist $\tilde{x}$, $\omega\in(0,1)$,
and $r>0$ such that (a) $\opA$ is $\omega$-contractive on $B(\tilde{x};r)$,
i.e., for all $x_{1},x_{2}$ in $B(\tilde{x};r)$,
and (b) $\|\opA(\tilde{x})-\tilde{x}\|\leq(1-\omega)r$. 
Then $\opA$ has a unique fixed point in $B(\tilde{x};r)$ and the
iteration scheme $x_{n+1}=\opA(x_{n})$ with the initialization $x_{0}:=\tilde{x}$
linearly converges to that unique fixed point.
\end{thm}
Furthermore, recall that \textsf{NExOS} (Algorithm \ref{alg:exterior-point-method})
can be compactly represented using (\ref{eq:convenient-form-1}) as
follows. For any $m\in\{1,2,\ldots,N\}$ (equivalently for each $\mu_{m}\in\{\mu_{1},\ldots,\mu_{N}\}$),
\begin{equation}
\begin{aligned}z_{\mu_{m}}^{n+1} & =\opT_{\mu_{m}}\left(z_{\mu_{m}}^{n}\right),\end{aligned}
\label{eq:ADMM_orig-1-1-1}
\end{equation}
where $z_{\mu_{m}}^{0}$ is initialized at $z_{\mu_{m-1}}$. From Proposition \ref{Operators-contractive}, for any $\mu\in\mathfrak{M},$
the operator $\opT_{\mu}$ is a $\kappa'$-contraction mapping over
the region of convexity $B(\bar{x};r_{\textrm{max}})$, where $\kappa'\in(0,1)$.
From Proposition \ref{Attainment-of-local-min}, there will be a unique
local minimum $x_{\mu}$ of problem \eqref{eq:smoothed-opt} over $B(\bar{x};r_{\textrm{max}})$. Suppose,
instead of the exact fixed point $z_{\mu_{m-1}}\in\fix\opT_{\mu_{m-1}},$
we have computed $\widetilde{z}$, which is an $\epsilon$-approximate
fixed point of $\opT_{\mu_{m-1}}$ in $B(\bar{x};r_{\textrm{max}})$,
i.e., $\ensuremath{\|\widetilde{z}-\opT_{\mu_{m-1}}(\widetilde{z})\|\leq\epsilon}$
and $\|\widetilde{z}-z_{\mu_{m-1}}\|\leq\epsilon$,
where $\epsilon\in [0,\overline{\epsilon})$. Then, we have: 
\begin{equation}
\|\opT_{\mu_{m-1}}(\widetilde{z})-z_{\mu_{m-1}}\|=\|\opT_{\mu_{m-1}}(\widetilde{z})-\opT_{\mu_{m-1}}(z_{\mu_{m-1}})\|\overset{a)}{\leq}\kappa'\underbrace{\|\widetilde{z}-z_{\mu_{m-1}}\|}_{\leq\epsilon}\leq\epsilon, \label{eq:sufficiently_small}
\end{equation}
where $a)$ uses $\kappa'$-contractive nature of $\opT_{\mu_{m-1}}$
over $B(\bar{x};r_{\textrm{max}})$. Hence, using triangle inequality,
\begin{align*}
\|\widetilde{z}-\bar{x}\|\overset{a)}{\leq}\|\widetilde{z}-\opT_{\mu_{m-1}}(\widetilde{z})\|+\|\opT_{\mu_{m-1}}(\widetilde{z})-z_{\mu_{m-1}}\|+\|z_{\mu_{m-1}}-\bar{x}\|\overset{b)}{\leq}2\epsilon+\|z_{\mu_{m-1}}-\bar{x}\|,
\end{align*}
where $a)$ uses triangle inequality and $b)$ uses (\ref{eq:sufficiently_small}).
As $\epsilon\in [0,\overline{\epsilon}),$ where $\overline{\epsilon}$
is defined in (\ref{eq:epsilon_upper_bound}), due to the second equation of Lemma \ref{thm:convergence_outer_iterations}(ii),
we have $r_{\textrm{max}}-\|\widetilde{z}-\bar{x}\|>\psi$.

Define $\Delta=\left((1-\kappa')\psi-\epsilon\right)/\ell,$ which
will be positive due to $\epsilon\in [0,\overline{\epsilon})$
and (\ref{eq:epsilon_upper_bound}). Next, select $\theta\in(0,1)$
such that $\overline{\Delta}=\theta\Delta<\mu_{1},$ hence there exists
a $\rho\in(0,1)$ such that $\overline{\Delta}=(1-\rho)\mu_{1}.$
Now reduce the penalty parameter using 
\begin{equation}
{\mu}_{m} =\mu_{m-1}-\rho^{m-2}\overline{\Delta}=\rho\mu_{m-1}=\rho^{m-1}\mu_{1}
\label{eq:how_mu_is_changed-1}
\end{equation}
for any $m \geq 2$. Next, we initialize the iteration scheme $z_{\mu_{m}}^{n+1}=\opT_{\mu_{m}}\left(z_{\mu_{m}}^{n}\right)$
at $z_{\mu_{m}}^{0}\coloneqq\widetilde{z}.$ Around this initial point,
let us consider the open ball $B(\widetilde{z},\psi)$. For any $x\in B(\widetilde{z};\psi)$,
we have $\|x-\bar{x}\|\leq\|x-\widetilde{z}\|+\|\widetilde{z}-\bar{x}\|<\psi+\|\widetilde{z}-\bar{x}\|<r_{\textrm{max}},$ where the last inequality follows from $r_{\textrm{max}}-\|\widetilde{z}-\bar{x}\|>\psi$. Thus
we have shown that $B(\widetilde{z};\psi)\subseteq B(\bar{x};r_{\textrm{max}})$.
Hence, from Proposition \ref{Operators-contractive}, on $B(\widetilde{z};\psi)$,
the Douglas-Rachford operator $\opT_{\mu_{m}}$ is contractive. Next,
we have $\|\opT_{\mu_{m}}(\widetilde{z})-\widetilde{z}\|\leq(1-\kappa')\psi$, because 
$\|\opT_{\mu_{m}}(\widetilde{z})-\widetilde{z}\|\overset{a)}{\leq}\|\opT_{\mu_{m}}(\widetilde{z})-\opT_{\mu_{m-1}}(\widetilde{z})\|+\|\opT_{\mu_{m-1}}(\widetilde{z})-\widetilde{z}\|\overset{b)}{\leq}\ell\|\mu_{m}-\mu_{m-1}\|+\epsilon\overset{c)}{\leq}\epsilon+\ell\Delta\overset{d)}{\leq}(1-\kappa')\psi,$
where $a)$ triangle inequality, $b)$ uses Proposition \ref{Operators-contractive}(ii)
and $\|\widetilde{z}-\opT_{\mu_{m-1}}(\widetilde{z})\|\leq\epsilon$, $c)$ uses (\ref{eq:how_mu_is_changed-1})
and ~$\|\mu_{m}-\mu_{m-1}\|\leq\overline{\Delta}\leq\Delta$ $d)$ uses the definition of $\Delta$.
Thus, both conditions of Theorem \ref{Thm:local_contraction_mapping_principle}
are satisfied, and $z_{\mu_{m}}^{n}$ in (\ref{eq:ADMM_orig-1-1-1})
will linearly converge to the unique fixed point $z_{\mu_{m}}$of
the operator $\opT_{\mu_{m}}$, and $x_{\mu_{m}}^{n},\text{\ensuremath{y_{\mu_{m}}^{n}}}$
will linearly converge to $x_{\mu_{m}}$. This completes the proof.

\subsection{Proof to Lemma \ref{Convergence-result-for-fixed-mu}\label{subsec:Proof-to-Theorem-fixed_mu}}

First, we show that, for the given initialization of $z_{\textrm{init}}$,
the iterates $z_{\mu_{1}}^{n}$ stay in $\overline{B}(z_{\mu_{1}};\|z_{\textrm{init}}-z_{\mu_{1}}\|)$
for any $n\in\nt$ via induction. The base case is true via given.
Let, $z_{\mu_{1}}^{n}\in\overline{B}(z_{\mu_{1}};\|z_{\textrm{init}}-z_{\mu_{1}}\|)$.
 Then, $\|z_{\mu_{1}}^{n+1}-z_{\mu_{1}}\|\overset{a)}{=}\|\opT_{\mu_{1}}(z_{\mu_{1}}^{n})-\opT_{\mu_{1}}(z_{\mu_{1}})\|\overset{b)}{\leq}\kappa'\|z_{\mu_{1}}^{n}-z_{\mu_{1}}\|\overset{c)}{\leq}\kappa'\|z_{\textrm{init}}-z_{\mu_{1}}\|$,
where $a)$ uses $z_{\mu_{1}}\in\fix\opT_{\mu},$ and $b)$ uses Proposition
\ref{Operators-contractive}, and $c)$ uses $\|z_{\mu_{1}}^{n}-z_{\mu_{1}}\|\leq\|z_{\textrm{init}}-z_{\mu_{1}}\|$. So,
the iterates $z_{\mu_{1}}^{n}$ stay in $\overline{B}(z_{\mu_{1}};\|z_{\textrm{init}}-z_{\mu_{1}}\|)$.
As, $\kappa'\in(0,1)$, this inequality also implies
that $z_{\mu}^{n}$ linearly converges to $z_{\mu}$ with the rate
of at least $\kappa'$. Then using similar reasoning presented in
the proof to Theorem \ref{Theorem_Main-convergence-result}, we have
$x_{\mu}^{n}$ and $y_{\mu}^{n}$ linearly converge to the unique
local minimum $x_{\mu}$ of problem (\ref{eq:smoothed-opt}). This completes
the proof.  

\subsection{Proof to Theorem \ref{Theorem_Main-convergence-result-weak}\label{subsec:Proof-to-Theorem-NExOS-weak}}

The proof is based on the results in \cite[Theorem 4]{li2016douglas} and \cite[Theorem 4.3]{themelis2020douglas}. The function $f$ is $L$-Lipschitz continuous and strongly smooth,
hence $f$ is a coercive function satisfying $\liminf_{\|x\|\to\infty}f(x)=\infty$
and is bounded below \cite[Corollary 11.17]{Bauschke2017}. Also,
$\regind(x)$ is jointly continuous hence lower-semicontinuous in
$x$ and $\mu$ and is bounded below by definition. Let the proximal
parameter $\gamma$ be smaller than or equal to $1/L$. Then due to
\cite[(14), (15) and Theorem 4]{li2016douglas}, $\{x_{\mu}^{n},y_{\mu}^{n},z_{\mu}^{n}\}$ (iterates of the inner algorithm of \textsf{NExOS }for
any penalty parameter $\mu$) will be bounded. This boundedness implies
the existence of a cluster point of the sequence, which allows us
to use \cite[Theorem 4 and Theorem 1]{li2016douglas} to show that for any $z_{\textrm{init}}$, the iterates $x_{\mu}^{n}$ and $y_{\mu}^{n}$ subsequentially converges
to a first-order stationary point  $x_{\mu}$ satisfying
$\nabla\left(f+\regind\right)(x_\mu)=0$. The rate  $\min_{n\leq k}\|\nabla\left(f+\regind\right)(x_{\rho\mu}^{n})\|\leq ((1-\gamma L)/2L) o(1/\sqrt{k})$ is a direct application of \cite[Theorem 4.3]{themelis2020douglas} as our setup satisfies all the conditions to apply it. %

\end{document}